\PassOptionsToPackage{dvipsnames}{xcolor}
\documentclass[final]{article}

\usepackage{geometry}

\usepackage{amsmath,amsfonts,amssymb}
\usepackage{amsthm}
\usepackage{authblk}
\usepackage{bm}
\usepackage[noadjust]{cite}
\usepackage{comment}
\usepackage{enumitem}
\setlist[enumerate]{label=(\roman*)}  
\setlist[enumerate,2]{label=(\alph*)} 
\usepackage{mathtools}
\usepackage{tikz}
	\usetikzlibrary{external,arrows}
\usepackage{tikz-cd}
\usepackage{tikzit}				% String diagrams
\tikzstyle{morphism}=[fill=white, draw=black, shape=rectangle]
\tikzstyle{medium box}=[fill=white, draw=black, shape=rectangle, minimum width=0.7cm, minimum height=0.7cm]
\tikzstyle{large morphism}=[fill=white, draw=black, shape=rectangle, minimum width=1.7cm, minimum height=1cm]
\tikzstyle{bn}=[fill=black, draw=black, shape=circle, inner sep=1.5pt]
\tikzstyle{state}=[fill=white, draw=black, regular polygon, regular polygon sides=3, minimum width=0.8cm, shape border rotate=180, inner sep=0pt]
\tikzstyle{medium state}=[fill=white, draw=black, regular polygon, regular polygon sides=3, minimum width=1.3cm, inner sep=0pt, shape border rotate=180]
\tikzstyle{large state}=[fill=white, draw=black, regular polygon, regular polygon sides=3, minimum width=2.2cm, shape border rotate=180, inner sep=0pt]
\tikzstyle{wide state}=[fill=white, draw=black, shape=isosceles triangle, minimum width=0.8cm, shape border rotate=270, inner sep=1.4pt, minimum height=0.5cm, isosceles triangle apex angle=80]
\tikzstyle{wn}=[fill=white, draw=black, shape=circle, inner sep=1.5pt]
\tikzstyle{blue morphism}=[fill=white, draw={rgb,255: red,15; green,0; blue,150}, shape=rectangle, text={rgb,255: red,15; green,0; blue,150}, tikzit category=blue]
\tikzstyle{red morphism}=[fill=white, draw={rgb,255: red,150; green,0; blue,2}, shape=rectangle, text={rgb,255: red,150; green,0; blue,2}, tikzit category=red]
\tikzstyle{blue state}=[fill=white, draw={rgb,255: red,15; green,0; blue,150}, shape=circle, regular polygon, regular polygon sides=3, minimum width=0.8cm, shape border rotate=180, inner sep=0pt, text={rgb,255: red,15; green,0; blue,150}, tikzit category=blue]
\tikzstyle{blue node}=[fill={rgb,255: red,15; green,0; blue,150}, draw={rgb,255: red,15; green,0; blue,150}, shape=circle, tikzit category=blue, inner sep=1.5pt]
\tikzstyle{blue}=[text={rgb,255: red,15; green,0; blue,150}, tikzit draw={rgb,255: red,191; green,191; blue,191}, tikzit category=blue, tikzit fill=white, inner sep=0mm]
\tikzstyle{blue wide state}=[fill=white, draw={rgb,255: red,15; green,0; blue,150}, text={rgb,255: red,15; green,0; blue,150}, shape=isosceles triangle, minimum width=0.8cm, shape border rotate=270, inner sep=1.4pt, minimum height=0.5cm, isosceles triangle apex angle=80]
\tikzstyle{red node}=[fill={rgb,255: red,150; green,0; blue,2}, draw={rgb,255: red,150; green,0; blue,2}, shape=circle, inner sep=1.5pt]
\tikzstyle{Purple node}=[fill={rgb,255: red,120; green,0; blue,120}, draw={rgb,255: red,120; green,0; blue,120}, text={rgb,255: red,120; green,0; blue,120}, shape=circle, inner sep=1.5pt]
\tikzstyle{red}=[text={rgb,255: red,150; green,0; blue,2}, inner sep=0mm, tikzit fill=white, tikzit draw={rgb,255: red,191; green,191; blue,191}]
\tikzstyle{purple}=[text={rgb,255: red,150; green,0; blue,150}, inner sep=0mm, tikzit fill=white, tikzit draw={rgb,255: red,191; green,191; blue,191}]
\tikzstyle{white morphism}=[fill=white, draw=white, shape=rectangle, tikzit draw={rgb,255: red,139; green,139; blue,139}]
\tikzstyle{leak morphism}=[fill=white, draw={rgb,255: red,120; green,0; blue,85}, shape=rectangle, text={rgb,255: red,120; green,0; blue,85}, tikzit category=leak]
\tikzstyle{leak}=[text={rgb,255: red,120; green,0; blue,85}, inner sep=0mm, tikzit fill=white, tikzit draw={rgb,255: red,191; green,191; blue,191}, tikzit category=leak]
\tikzstyle{leak node}=[fill={rgb,255: red,120; green,0; blue,85}, draw={rgb,255: red,120; green,0; blue,85}, shape=circle, inner sep=1.5pt, tikzit category=leak]
\tikzstyle{horiz state}=[fill=white, draw=black, regular polygon, regular polygon sides=3, minimum width=1cm, shape border rotate=90, inner sep=0pt]

\tikzstyle{arrow}=[->]
\tikzstyle{dashed box}=[-, dashed]
\tikzstyle{blue arrow}=[-, draw={rgb,255: red,15; green,0; blue,150}, tikzit category=blue]
\tikzstyle{red arrow}=[-, draw={rgb,255: red,150; green,0; blue,2}, tikzit category=red]
\tikzstyle{purple arrow}=[->, draw={rgb,255: red,120; green,0; blue,120}, >=stealth, shorten <=2pt, shorten >=2pt]
\tikzstyle{protected purple arrow}=[->, draw={rgb,255: red,120; green,0; blue,120}, >=stealth, shorten <=2pt, shorten >=2pt, preaction={line width=1.8pt, white, draw}]
\tikzstyle{mapsto}=[{|->}]
\tikzstyle{double wire}=[-, double]
\tikzstyle{protected double wire}=[-, double, preaction={line width=2.5pt,white,draw}]
\tikzstyle{triple wire}=[-, draw, line width=0.4pt, preaction={-, draw, line width=1.4pt, white, preaction={-, draw, line width=2.2pt}}]
\tikzstyle{protected}=[-, preaction={line width=1.8pt,white,draw}]
\tikzstyle{leak arrow}=[-, tikzit draw={rgb,255: red,150; green,0; blue,120}]
\tikzstyle{protected leak arrow}=[-, tikzit draw={rgb,255: red,150; green,0; blue,120}]
\tikzstyle{hollow arrow}=[-, very thin, white, preaction={line width=0.7pt,draw={rgb,255: red,120; green,0; blue,85}}, tikzit category=leak, tikzit draw={rgb,255: red,150; green,0; blue,120}]
\tikzstyle{protected hollow arrow}=[-, very thin, white, preaction={line width=0.7pt,draw={rgb,255: red,120; green,0; blue,85},preaction={line width=2.1pt,white,draw}}, tikzit category=leak, tikzit draw={rgb,255: red,150; green,0; blue,120}]
\tikzstyle{over arrow}=[-, black, preaction={draw=white, double}]
\tikzstyle{purple line}=[-, draw={rgb,255: red,120; green,0; blue,120}]
\tikzstyle{orange line}=[-, draw={rgb,255: red,255; green,100; blue,0}]
\tikzstyle{red line}=[-, draw={rgb,255: red,150; green,0; blue,2}]
\tikzstyle{blue line}=[-, draw={rgb,255: red,15; green,0; blue,150}]
\tikzstyle{blue double arrow}=[-, double, draw={rgb,255: red,15; green,0; blue,150}, tikzit category=blue]
\tikzstyle{red double arrow}=[-, double, draw={rgb,255: red,150; green,0; blue,2}, tikzit category=red]
\tikzstyle{purple double line}=[-, double, draw={rgb,255: red,120; green,0; blue,120}]
\tikzstyle{orange double line}=[-, double, draw={rgb,255: red,255; green,100; blue,0}]
\tikzstyle{curly brace}=[decorate, decoration={brace,amplitude=5pt}]

\tikzstyle{d-wire1 plate}=[-, double=red!20!white]
\tikzstyle{d-wire2 plate}=[-, double=red!32!white]
\tikzstyle{dotted_plate}=[-, densely dotted, draw=red, fill opacity=0.4, fill=red!50!white, rounded corners]
\tikzstyle{twire1}=[-, draw, line width=0.4pt, preaction={-, draw, line width=1.4pt, red!20!white, preaction={-, draw, line width=2.2pt}}]
\tikzstyle{twire2}=[-, draw, line width=0.4pt, preaction={-, draw, line width=1.4pt, red!32!white, preaction={-, draw, line width=2.2pt}}]
\usepackage[obeyFinal]{todonotes}

\definecolor{myurlcolor}{rgb}{0,0,0.3}
\definecolor{mycitecolor}{rgb}{0,0.3,0}
\definecolor{myrefcolor}{rgb}{0.3,0,0}
\usepackage[pagebackref,draft=false]{hyperref}
\hypersetup{colorlinks,
	linkcolor=myrefcolor,
	citecolor=mycitecolor,
	urlcolor=myurlcolor}

\usepackage[notref,notcite]{showkeys}
\usepackage[capitalize,noabbrev]{cleveref}

\newtheorem{theorem}{Theorem}[section]
\newtheorem{proposition}[theorem]{Proposition}
\newtheorem{lemma}[theorem]{Lemma}
\newtheorem{corollary}[theorem]{Corollary}

\newtheorem{definition}[theorem]{Definition}

\theoremstyle{definition}
\newtheorem{example}[theorem]{Example}
\newtheorem{remark}[theorem]{Remark}
\newtheorem{warning}[theorem]{Warning}
\newtheorem{construction}[theorem]{Construction}

\numberwithin{equation}{section}

\makeatletter
\def\cref@thmoptarg[#1]#2#3#4{%
	    \ifhmode\unskip\unskip\par\fi%
	    \normalfont%
	    \trivlist%
	    \let\thmheadnl\relax%
	    \let\thm@swap\@gobble%
	    \thm@notefont{\fontseries\mddefault\upshape}%
	    \thm@headpunct{.}% add period after heading
	    \thm@headsep 5\p@ plus\p@ minus\p@\relax%
	    \thm@space@setup%
	    #2% style overrides
	    \@topsep \thm@preskip               % used by thm head
	    \@topsepadd \thm@postskip           % used by \@endparenv
	    \def\@tempa{#3}\ifx\@empty\@tempa%
	      \def\@tempa{\@oparg{\@begintheorem{#4}{}}[]}%
	    \else%
	      \refstepcounter[#1]{#3}%  <<< cleveref modification
	      \@namedef{cref@#3@alias}{#1}% added
	      \def\@tempa{\@oparg{\@begintheorem{#4}{\csname the#3\endcsname}}[]}%
	    \fi%
	    \@tempa}%
\makeatother

\newcommand{\N}{\mathbb{N}}

\newcommand{\R}{\mathbb{R}}
\newcommand{\Rnneg}{\mathbb{R}_{\ge 0}}

\newcommand{\IF}{\mathbf{I}}	% indicator function

\let\P\undefined
\DeclarePairedDelimiterXPP{\P}[1]{\operatorname{\mathbf{P}}}[]{}{#1}% probability
\newcommand{\E}[1]{\mathbf{E}\left[#1\right]}	% expectation value
\newcommand{\cat}[1]{{\mathsf{#1}}} 
\newcommand{\op}{\mathrm{op}}

\newcommand{\FinSub}[1]{\mathsf{FinSub}\left(#1\right)}
\newcommand{\Sub}{\mathsf{Sub}}
\newcommand{\Kl}{\mathsf{Kl}}		% Kleisli category construction
\newcommand{\FinKl}{\mathsf{FinKl}}		% Kleisli category construction

\newcommand{\id}{\mathrm{id}} 		% identity
\newcommand{\Par}[1]{\mathsf{Partial}\left({#1}\right)}	% category of partial maps, as in \cite{cockettlack2002partialmaps}
\newcommand{\distMons}{\mathcal{M}} % distinguished monomorphisms
\newcommand{\dom}[1]{\mathrm{dom}(#1)}		% domain

\tikzset{pullback/.style={minimum size=1.2ex,path picture={	% pullback symbol in diagram
			\draw[opacity=1,black,-,#1] (-0.5ex,-0.5ex) -- (0.5ex,-0.5ex) -- (0.5ex,0.5ex);%
}}}
\newcommand{\comp}{ 		% Command for sequential composition
	\mathchoice{\,}{\,}{}{} 	% First two are for displaystyle and text style, the remaining two for smaller math.
}
\newcommand{\domext}{\sqsupseteq}

\newcommand{\ph}{{\kern0.06em}\mathord{\rule[-0.05em]{0.6em}{0.05em}}{\kern0.06em}}		% Argument placeholder
\newcommand{\phsm}{{\kern0.04em}\mathord{\rule[-0.035em]{0.4em}{0.035em}}{\kern0.04em}}		% Scriptstyle argument placeholder
\newcommand{\sot}{ \ : \ } % Notation for <so that> in logical formulas
\newcommand{\supp}[1]{\mathrm{Supp}\left(#1\right)}	% support of a distribution 

\newcommand{\cC}{\mathsf{C}}		% Markov cat
\renewcommand{\det}{\mathrm{det}}	% deterministic morphisms
\newcommand{\copyable}{\mathrm{cop}}	% sampling map
\newcommand{\samp}{\mathsf{samp}}	% sampling map
\newcommand{\Rel}{\mathsf{Rel}}
\newcommand{\FinStoch}{\mathsf{FinStoch}}

\newcommand{\BorelStoch}{\mathsf{BorelStoch}}
\newcommand{\SetMulti}{\mathsf{SetMulti}}
\newcommand{\FinSetMulti}{\mathsf{FinSetMulti}}
\newcommand{\Dist}{\mathsf{Dist}}
\newcommand{\as}[1]{% 					almost surely
		\def\relstate{#1}%
		\ifx\relstate\empty
		  \text{a.s.}%
		\else
		  {#1\text{-a.s.}}%
		\fi
	}
\newcommand{\ase}[1]{=_{#1\text{-a.s.}}}					% almost surely equal
\DeclareMathOperator{\cop}{copy}
\DeclareMathOperator{\discard}{del}

	\DeclarePairedDelimiter{\abs}{\lvert}{\rvert}
	\DeclarePairedDelimiter{\norm}{\lVert}{\rVert}
	\DeclarePairedDelimiterXPP{\pnorm}[2]{}{\lVert}{\rVert}{_{#1}}{#2}
	\makeatletter
		\let\oldabs\abs
		\def\abs{\@ifstar{\oldabs}{\oldabs*}}
		\let\oldnorm\norm
		\def\norm{\@ifstar{\oldnorm}{\oldnorm*}}
		\let\oldpnorm\pnorm
		\def\pnorm{\@ifstar{\oldpnorm}{\oldpnorm*}}
	\makeatother

	\DeclarePairedDelimiterX{\Set}[1]{\{}{\}}{%
		
		#1
	}
	\makeatletter
		\let\oldSet\Set
		\def\Set{\@ifstar{\oldSet}{\oldSet*}}
	\makeatother

	\DeclarePairedDelimiterX{\Family}[1]{(}{)}{%
		
		#1
	}
	\makeatletter
		\let\oldFamily\Family
		\def\Family{\@ifstar{\oldFamily}{\oldFamily*}}

\newcommand{\newterm}[1]{\textbf{#1}}

\let\originalleft\left
\let\originalright\right
\renewcommand{\left}{\mathopen{}\mathclose\bgroup\originalleft}
\renewcommand{\right}{\aftergroup\egroup\originalright}

\title{Partializations of Markov categories}

\author{Areeb Shah Mohammed\thanks{This research was funded in whole or in part by the Austrian Science Fund (FWF) [doi:\href{https://www.doi.org/10.55776/P35992}{10.55776/P35992}]. For open access purposes, the author has applied a CC BY public copyright license to any author accepted manuscript version arising from this submission.}}
\affil{Department of Mathematics, University of Innsbruck, Austria}
\date{5\textsuperscript{th} September 2025}
\begin{document}

\maketitle

\begin{abstract}
	The present work develops a construction of a CD category of ``partial kernels'' from a particular type of Markov category called a ``partializable Markov category''.
	These are a generalization of earlier models of categories of partial morphisms such as $p$-categories, dominical categories, restriction categories, etc.\ to a non-deterministic/non-cartesian setting.
	Here all morphisms are quasi-total, with a natural poset enrichment corresponding to one morphism being a ``restriction'' of the other.
	Furthermore, various properties important to categorical probability are preserved, such as positivity, representability, conditionals, Kolmogorov products, and splittings of idempotents.
	We additionally discuss an alternative notion of Kolmogorov product suitable for partial maps, as well as partial algebras for probability monads.

	The primary example is that of the partialization of the category of standard Borel spaces and Markov kernels.
	Other examples include variants where the distributions are finitely supported, or where one considers multivalued maps instead.\footnote{The idea being that the randomness is in terms of ``possibilities'' rather than probabilities.}
\end{abstract}

\tableofcontents

\section{Introduction}\label{introSec}
Markov categories have had success in capturing several probabilistic phenomena.
Intuitively, a Markov category is a category where the morphisms behave like stochastic maps, and the categorical structure is meant to model the information flow involved in composite probabilistic processes.
However, many operations one performs in probability theory (particularly in the standard Kolmogorov model) may involve constructions such as limits or integrals that are not always defined (or finite).
For instance, even a relatively innocuous operation like taking the average $\lim_{n \to \infty} \frac{1}{n} \sum_{i=1}^n X_i$ of a sequence $\left(X_i\right)$ of random variables is only defined when the limit exists.\footnote{Another example might be trying to construct a map taking a distribution on $\mathbb R$ to its mean, which need not exist in general.}
In particular, a stochastic \emph{map}
\[
	\R^\N \to \R \qquad\qquad \left(X_i\right)_{i \in \N} \mapsto \lim_{n \to \infty} \frac{1}{n} \sum_{i=1}^n X_i
\]
would only be partially defined at best, and hence does not neatly fit into a typical Markov category.

Similarly, the notion of expectation (that is involved in the standard law of large numbers) in a categorical setting would also take the form of a suitable map.
This would now take as input a \emph{distribution}, and (deterministically!) return the expectation.
Indeed, this can be formalized as an algebra for the \newterm{distribution monad} $P$ on a Markov category, and for instance expectation defines an algebra $P\left[a,b\right] \to \left[a,b\right]$ for every compact interval $\left[a,b\right]$.
However, in the unbounded case, such as the common case of distributions on $\mathbb R$ or $\Rnneg$, the expectation may not be defined/finite.
Thus, such a map would have to be partially defined as well.

A typical manner of formalizing the notion of a ``partially defined map'' $X \to Y$ is as a ``totally defined'' operation defined on some subobject $D \subseteq X$.
This is akin to the shift from sets and functions to sets and partial functions.
Here, the ``partial maps'' are encoded as equivalence classes of spans $X \hookleftarrow D \to Y$ with $D \hookrightarrow X$ a suitable kind of monomorphism.
Categories of spans have a rich structure, and that of spans with the ``left maps'' being monomorphisms have particularly rich axiomatizations in terms of induced idempotents on their domains.\footnote{Some such axiomatizations are $p$-categories, dominical categories, poset bicategories, and restriction categories. We refer to \cref{domainIdempsHistory,stableClassAliases} for a discussion.}
This formalism for partiality has already shown fruitfulness in other fields of mathematics such as for instance, algebraic topology, homotopy theory, and algebraic geometry~\cite{carmeliAmbidexterityChromaticHomotopy2022,harpazAmbidexterityUniversalityFinite2020,haugsengIteratedSpansClassical2018,lurieClassificationTopologicalField,srinivasExactCategoriesQuillens1996}.
However when monoidal structures on this category of partial maps are discussed in existing literature, they reduce to the cartesian structure on the subcategory of ``total'' maps.
This is insufficient for a probabilistic setting, as the totally defined maps in a category of partial \emph{stochastic} maps ought to be a Markov category.
On the other hand, the cartesian Markov categories correspond to those where every map is deterministic~\cite{foxCoalgebrasCartesianCategories1976}, and hence there is no scope for randomness.

In the present work, we develop a construction of a CD category of ``partial stochastic maps'' from a particular type of Markov category called a \newterm{partializable Markov category}. These are a generalization of the aforementioned partial map construction to a non-deterministic/non-cartesian setting.
This CD category is further such that ``all domain inclusions are deterministic''\footnote{The precise condition is that all maps are quasi-total in the sense of~\cite[Definition 3.1]{dilavore2023evidential} (see also \cref{def:quasi-total}).} and the hom-sets are equipped with a natural poset enrichment corresponding to one map being a ``restriction'' of the other.
Furthermore, we prove that properties such as representability, positivity, conditionals, Kolmogorov products, and splitting of idempotents are preserved.

Additionally, in the representable case, we can make sense of partial algebras for the associated distribution monad.
This allows us to construct, for example, the integration map computing the mean of a distribution on $\Rnneg$ (when defined) as a partial algebra.
Partial algebras provide a framework to encode these kinds of ``partially defined'' probabilistic constructions systematically.

Finally, we introduce the notion of \newterm{lax Kolmogorov products}, a variant of Kolmogorov products particularly suitable for such categories.
Such a notion allows us to make sense of a functorial infinite tensor of maps in a CD category with all maps quasi-total.
We show that when a partializable Markov category has a Kolmogorov product, it extends to a lax Kolmogorov product (in addition to the classic ``strict'' one) in the corresponding CD category of partial maps.

\subsection*{Structure of the paper}
In the remainder of the \cref{introSec}, we summarize the main definitions and results of the paper, and discuss some other works related to the topic.

\cref{backgroundSec} exposits the background material relevant to the paper, in particular the basic definitions of restriction categories and categories of spans, as well as the fragment of the language of Markov and CD categories that we will use.

In \cref{sec:partializable}, we introduce the main contribution of the paper, a CD category of partial maps $\Par{\cC}$ in a Markov category $\cC$ of a suitable type, called a \newterm{partializable Markov category}.
We show that these partializations are positive quasi-Markov categories, which are a class of CD categories where the theory of restriction categories applies, with a corresponding poset enrichment.

Then, in \cref{section:transferProperties} we show that the partialization construction preserves several important properties of Markov categories, such as representability, conditionals, Kolmogorov products, and splitting of idempotents in a natural manner.

Furthermore, we show that in the representable case, the phenomenon of expectation defining an algebra for the distribution monad extends to the non-compact case, with expectation defining a partial algebra for the distribution monad on the nonnegative reals.

The final section introduces the notion of \newterm{lax Kolmogorov product}, a variant of the classical Kolmogorov product for Markov categories that defines a functorial infinite tensor of partial maps.\footnote{This is unfortunately not the case for the direct generalization of the original notion of Kolmogorov product, see \cref{KolProdsNonFunctorialWarning}.}

\subsection{Main definitions and results}
Here we summarize briefly the contents of the paper, linking to the relevant sections for more details (and references).
The definitions of (most) terms used can be found in \cref{backgroundSec}.
\begin{proposition}[Elaborated at {\cref{positiveQuasiMarkovRestriction}}]
	Every positive quasi-Markov category $\cC$ is a restriction category (in the sense of~\cite[Section 2.1.1]{cockettlack2002partialmaps}, see \cref{restrictionCategoryDefinition}).
\end{proposition}
\begin{definition}[Elaborated at {\cref{def:partializable}}]
	A \newterm{partializable Markov category} is a Markov category $\cC$ such that:
	\begin{enumerate}
		\item It is positive (so in particular all isomorphisms are deterministic);
		\item Pullbacks of deterministic monomorphisms exist and are themselves deterministic;
		\item Deterministic monomorphisms are closed under tensoring.
	\end{enumerate}
\end{definition}

\noindent To such a Markov category, we associate the following CD category of ``partial stochastic maps''.
\begin{definition}[Combining {\cref{def:partial_C,proposition:SpanMonStr}}]
	Given a partializable Markov category $\cC$, the CD category of ``partial maps in $\cC$'', the \newterm{partialization} $\Par{\cC}$ has:
	\begin{enumerate}
		\item Objects those of the original category $\cat{C}$;
		\item Maps $X\to Y$ equivalence classes of spans
			\[
            X \xhookleftarrow{i} D \xrightarrow{f} Y
        \]
		with $i$ a \emph{deterministic} monomorphism;
		\item Composition is done by pullback: For maps represented by spans $X \xhookleftarrow{i} D_f \xrightarrow{f} Y$ and $Y \xhookleftarrow{j} D_g \xrightarrow{g} Z$,
		\[
		\begin{tikzcd}[sep=tiny, ampersand replacement=\&]
			\& \& E \arrow[ld, "u"', hook] \arrow[rd, "v"] \arrow["\lrcorner"{anchor=center, pos=0.125, rotate=-45}, draw=none, dd] \& \& \\
			\& D_f \arrow[ld, "i"', hook] \arrow[rd, "f"'] \& \& D_g \arrow[ld, "j", hook] \arrow[rd, "g"] \& \\
			X \& \& Y \& \& Z
		\end{tikzcd}
		\]
		the composite is represented by $X \xhookleftarrow{i\comp u} E \xrightarrow{g \comp v} Z$;
	\item Tensoring is done componentwise: For maps $X \xhookleftarrow{i} D_f \xrightarrow{f} Y$ and $X' \xhookleftarrow{j} D_g \xrightarrow{g} Y'$, the tensor is represented by
        \[
            X \otimes X' \xhookleftarrow{i \otimes j} D_f \otimes D_g \xrightarrow{f \otimes g} Y \otimes Y'
        \]
		\item The CD structure is inherited from $\cC$.
\end{enumerate}
\end{definition}
\noindent The original category $\cC$ embeds fully faithfully into $\Par{\cC}$, being identified with the subcategory of total maps (\cref{proposition:SpanCDTotal}).

\begin{example}
Our main example of a partializable Markov category is the category $\BorelStoch$ of standard Borel spaces and Markov kernels (\cref{example:BorelStochPartializable}).
The maps $X \to Y$ in $\Par{\BorelStoch}$ can be identified with stochastic maps $D \to Y$ for a measurable $D \subseteq X$, capturing the intuition of ``partially defined stochastic maps''.
Similarly, the subcategory $\FinStoch$ whose objects can be identified with finite sets and maps with stochastic matrices is another example (\cref{FinStochPartializable}).

Another class of examples of frequent interest in categorical probability is that of distributions valued in an entire zerosumfree semiring (\cref{discDistPartializable}).\footnote{Particularly in light of the relative ease of computation in this setting.}
This class involves categories in typical use such as 
\begin{itemize}
	\item The category $\Dist$ of discrete measurable spaces and finitely supported stochastic maps (\cref{distAndMultiPartializable}~\ref{distPartializable});
	\item The category $\FinSetMulti$ of finite sets and multivalued maps (where the distribution of the ``image'' of a point in the source is replaced by the ``possible images'' it might have in the target) (\cref{distAndMultiPartializable}~\ref{FinSetMultiPartializable}).
\end{itemize}

Expanding on the $\FinSetMulti$ example is the category $\SetMulti$ of sets and multivalued maps (\cref{SetMultiPartializable}).
This can be seen as a ``possibilistic'' analogue of $\Dist$.
\end{example}

\begin{warning}[Clarified in detail at {\cref{parCatsExmpsDiff}}]
	Another notion of partiality that has classically been used in a probabilistic context is that of sub-distributions.
	These are a probabilistic analogue of the model $\mathsf{Set}_{\ast}$ of the category of partial set functions given by pointed sets and pointed morphisms.\footnote{Indeed, for $\mathsf{Set}$ there is an equivalence between $\mathsf{Set}_{\ast}$ and the category of sets and spans, in the sense of $\Par{\mathsf{Set}}$ as we have considered partial maps so far.}
	While similar, the partializations $\Par{\BorelStoch}, \Par{\Dist}$, and $\Par{\SetMulti}$ are however \emph{not} the usual categories $\BorelStoch_{\le 1}, \Kl{\left(\mathbf D_{\le 1}\right)}$, and $\Rel$ of (finitely supported) sub-stochastic distributions and relations.
Essentially, while the objects and hom-sets of these categories can be compared, the compositions are different.
\end{warning}

\noindent The notions of partiality associated to such categories of spans have a formulation in terms of the CD structure.
\begin{proposition}[{\cref{partializationRestrictionStructuresAgree}}]
	The two domain idempotents (in terms of spans and the CD structure, see \cref{restrictionCategoryDefinition,domainDefinition} respectively) agree.
	In particular, the domain of a map in the sense of a CD structure defines a restriction operator.
\end{proposition}

\begin{corollary}
	This has as immediate consequences:
	\begin{itemize}
		\item $\Par{\cC}$ is quasi-Markov (\cref{partializationsQuasiMarkov});

		\item The two notions of poset enrichment (\cref{def:dom_ext,eq:rest_ord}) agree (\cref{prop:ParExtParOrd}).
		\item The CD and span notions of totality agree (\cref{proposition:SpanCDTotal}).
	\end{itemize}
\end{corollary}

One can also relate the copyable maps of $\Par{\cC}$ to those of $\cC$.
\begin{proposition}[{\cref{proposition:ParCDDeterminism}}]
	A map $X \hookleftarrow D \xrightarrow{f} Y$ of $\Par{\cC}$ is copyable if and only if $f$ is deterministic.
\end{proposition}

\begin{proposition}[{\cref{proposition:ParPos}}]
	Positivity transfers from $\cC$ to $\Par{\cC}$.
\end{proposition}

\noindent Representability also interacts naturally with partialization.
\begin{proposition}[{\cref{proposition:ParDistObs,proposition:ParPushforwardComputation}}]
	Consider a partializable and representable Markov category $\cC$.
	The distribution objects and sampling maps of $\cC$ define distribution objects and sampling maps in $\Par{\cC}$ as well.
	In particular, $\Par{\cC}$ is representable.
	
	Under this correspondence, the copyable counterpart of a map $X \hookleftarrow D \xrightarrow{f} Y$ in $\Par{\cC}$ is the map $X \hookleftarrow D \xrightarrow{f^\sharp} PY$.
	Furthermore, the pushforward $P\bigl(X \hookleftarrow D \xrightarrow{f} Y\bigr)$ of a map $X \hookleftarrow D \xrightarrow{f} Y$ in $\Par{\cC}$ is represented by $\bigl(PX \xhookleftarrow{Pi} PD \xrightarrow{Pf} PY\bigr)$.
\end{proposition}

\noindent When $\cC$ is representable, the distribution monad on $\cC_\det$ extends to a monad on the subcategory of copyable maps $\Par{\cC}_\copyable \cong \Par{\cC_\det}$ of $\Par{\cC}$, and we can thus speak of algebras for this monad, the \newterm{partial algebras}.
\begin{proposition}[{\cref{parAlgChar,IntParAlg}}]
	Partial algebras have a natural characterization in terms of spans in $\cC_\det$.

	Furthermore in $\Par{\BorelStoch}$, the integration map (which computes expectations) defines a partial algebra structure on $\Rnneg$, with domain the distributions with finite expectation.
\end{proposition}
We also discuss the issues with doing the same for $\R$ in \cref{BochnerIntNotStrictAlg}.

\noindent $\Par{\cC}$ also has conditionals when $\cC$ does, and these are essentially defined in the ``greatest possible domain''.
Thus, in this case $\Par{\cC}$ is a ``partial Markov category'' in the sense of~\cite[Definition 3.2]{dilavore2023evidential}.
\begin{proposition}[{\cref{ConditionalsExist}}]
	Consider a partializable Markov category $\cC$ with conditionals.
	Given a morphism $u \colon A \to X\otimes Y$ in $\Par{\cC}$ represented by a span $\bigl(A \xhookleftarrow{i} D \xrightarrow{f} X \otimes Y\bigr)$, the conditional $u_{\vert X} \colon X \otimes A \to Y$ exists and is represented by the span $\bigl(X \otimes A \xhookleftarrow{X \otimes i} X \otimes D \xrightarrow{f_{\vert X}} Y\bigr)$.

	In particular, $\Par{\cC}$ has conditionals.
\end{proposition}

\noindent As mentioned earlier, idempotents have proven to be useful both in terms of categorical structures for partiality and for probability.
Fortunately, there is a natural correspondence between the idempotents of $\Par{\cC}$ and those of $\cC$.

\begin{proposition}[{\cref{proposition:IdempSpans,proposition:SplitParIdemps,proposition:PartialIdempStaticStrongBalanced}}]\label{thm:preservedprops}
	For a partializable Markov category $\cC$:
	\begin{enumerate}
		\item Idempotents in $\Par{\cC}$ are idempotents on their domains (in $\cC$);
		\item All idempotents in $\Par{\cC}$ split if they do in $\cC$;
		\item Idempotents in $\Par{\cC}$ are static/strong/balanced if and only if the corresponding idempotent in $\cC$ is.
	\end{enumerate}
\end{proposition}

\noindent Finally, $\Par{\cC}$ also inherits Kolmogorov products from $\cC$.
\begin{proposition}[{\cref{proposition:ParLaxKolProds}, first assertion}]
	Consider a partializable Markov category $\cC$ with Kolmogorov products of size $K$.
	Given a family of objects ${\left(X_k\right)}_{k \in K}$ of $\cC$, the inclusion of the Kolmogorov product projections into $\Par{\cC}$ defines a Kolmogorov product.
\end{proposition}
\begin{warning}\label{KolProdsNonFunctorialWarning}
	This notion of Kolmogorov product is no longer functorial in morphisms of $\Par{\cC}$.
	To be precise, the universal property of the Kolmogorov product does not suffice to induce a map $\bigotimes_{i \in I} X_i \to \bigotimes_{i \in I} Y_i$ from an arbitrary \emph{family of maps} ${\left(X_i \to Y_i\right)}_{i \in I}$ in $\Par{\cC}$.\footnote{The key point here is that the typical construction no longer yields a cone to apply the universal property to. For a precise illustration of the problem, see the discussion at the start of \cref{section:KolmogorovProducts}.}
\end{warning}
\noindent The following modification leads to a functorial notion of Kolmogorov product in $\Par{\cC}$.
\begin{definition}[{\cref{def:lax-cone,def:lax-kolmogorov}}]
	Consider a $K$-indexed family of objects $\left(X_k\right)_{k\in K}$ in a quasi-Markov category $\cC$, and the diagram $X^{\left(\phsm\right)}\colon \FinSub{K}^\op\to\cat{C}$ formed by finite products and marginalizations (as with the usual infinite tensor products~\cite[Section 3]{fritzrischel2019zeroone}).
	
	A \newterm{lax cone} over the diagram $X^{\left(\phsm\right)}$ is an object $A$ of $\cat{C}$ together with arrows $f_F\colon A\to X^F$ for all finite $F\subseteq K$, such that for all subsets $G\subseteq F\subseteq K$, the following diagram commutes laxly,
	\[
	\begin{tikzcd}[sep=small]
		&&& X^F \ar{dd}{\pi_{F,G}} \\ 
		A \ar{urrr}{f_F} \ar{drrr}[swap]{f_G} && |[overlay,xshift=2mm]| \ge \\
		&&& X^G
	\end{tikzcd}
	\]
	where $\pi_{F,G}$ denotes the functor action (marginalization) of $X^{\left(\phsm\right)}$ on the inclusion $G\subseteq F$.

A \newterm{lax infinite tensor product} is a lax cone $\big(X^K \xrightarrow{\pi_F} X^F\big)_{F \subseteq K \mathrm{ finite}}$ which is universal in the following sense: for any other lax cone $\big(A \xrightarrow{f_F} X^F\big)_{F \subseteq K \mathrm{ finite}}$ there is a greatest morphism $A \xrightarrow{g} X^K$ such that $\pi_{F}\comp g \le f_F$ for each finite $F \subseteq K$, and this lax limit is further preserved by tensoring with an arbitrary object $Y$.
	
	We call such a lax infinite tensor product a \newterm{lax Kolmogorov product} when the projections $\pi_F$ are deterministic.\footnote{It can be shown that in such a case, a cone with copyable components induces a copyable map.}
\end{definition}

\begin{proposition}[{\cref{proposition:ParLaxKolProds,proposition:ParMorInfProd,proposition:ParCDInfCopies}}]
	Consider a partializable Markov category $\cC$ admitting Kolmogorov products of size $K$.
	Given a family of objects $\left(X_k\right)_{k \in K}$ in $\cC$, the inclusion of the Kolmogorov product projections into $\Par{\cC}$ define a lax Kolmogorov product as well.

	A $K$-indexed family of maps $\bigl(X_k \xhookleftarrow{i_k} D_k \xrightarrow{f_k} Y_k\bigr)_{k \in K}$ in $\Par{\cC}$ induces a map $\bigotimes_{k \in K} X_k \to \bigotimes_{k \in K} Y_k$.
	This map is represented by
	\[
		\bigotimes_{k \in K} X_k \xhookleftarrow{\bigotimes_{k \in K} i_k} \bigotimes_{k \in K} D_k \xrightarrow{\bigotimes_{k \in K} f_k} \bigotimes_{k \in K} Y_k
	\]
	Similarly, the infinite copy $u^{\left(K\right)}$ of a map $u = \bigl(X \xhookleftarrow{i} D \xrightarrow{f} Y\bigr)\colon X \to Y$ in $\Par{\cC}$ is represented by $\bigl(X \xhookleftarrow{i} D \xrightarrow{f^{\left(K\right)}} Y^K\bigr)$.
\end{proposition}

\subsection{Related work}
In this section we discuss some related literature.
As much of this work revolves around extending material in the theory of partial maps and restriction categories to a nondeterministic/stochastic setting, we group the discussion by theme rather than individual works (as there is typically partial overlap at multiple recurring themes).
Again, the definitions of (most) terms used can be found in \cref{backgroundSec}.
\begin{remark}[{Empirical sampling as a partial map}]\label{companionWork}
The current article is companion work to a paper on categorical versions of the law of large numbers by Fritz et al.~\cite{fritz2025empirical}.
Central to this is the construction of an ``empirical sampling map'', which intuitively takes an infinite sequence of points and returns a sample from the empirical distribution of that sequence.
However, such an empirical distribution need not always be defined (the relative frequencies need not converge), and consequently we should only expect such a map to be partially defined.
The present work provides the framework used by~\cite{fritz2025empirical} where such constructions are formalized.
\end{remark}

\begin{remark}[{Similar ``categories of partial maps''}]\label{stableClassAliases}
	We construct our partializations as span categories associated to a ``particular stable system of monomorphisms'' in a Markov category, the deterministic monomorphisms.
	The notion of stable system of monomorphisms (and the associated category of partial maps) have featured in several works in the literature, at least as old as Rosolini's PhD thesis~\cite[Section 2.1]{rosoliniContinuityEffectivenessTopoi1986} under the term ``Dominions''.
	However, these works tend to assume that the base category is cartesian, in order to shift focus to a particular abstraction called $p$-categories (amongst other names), where the focus is on an abstraction of the kind of category that $\Par{\cC}$ is (rather than on particular categories and stable systems of monics).
	This is in a sense unavoidable, as the axiomatizations are all in terms of ``pairing'' and ``projection'' maps required to make sense of domains of definition in terms of idempotents on the source.

	A version of $p$-categories also appears in Carboni's~\cite{carboniBicategoriesPartialMaps1987} for bicategories, for which there is such a ``bicategory of partial morphisms'' for every cartesian category.
	Such a notion also appears from the perspective of recursion theory in the work of di Paola and Heller~\cite{paolaDominicalCategoriesRecursion1987} under the name ``Dominical categories''.
	These various models have been compared and shown to be essentially equivalent by Robinson and Rosolini in~\cite{robinsonrosoliniCategoriesPartialMaps1988}, where stable systems of monics are now referred to as ``admissible''.
	In these notions of partial maps, there is a notion of a partial order on appropriate morphisms corresponding to one being a restriction of another (in terms of domains).\footnote{See \cref{restricitonPartialOrderHistory} for a discussion on CD generalizations of this order.} Such a notion is focused on in Curien and Obtu{\l}owicz's~\cite{curienPartialityCartesianClosedness1989}, where they develop a notion of ``partial cartesian closed category'' extending the notion of cartesian closedness (they also develop a stronger notion of ``partial topos'').
	Thus there are a wide variety of equivalent axiomatizations of categories of partial maps in the literature.

	On a purely $1$-categorical level (that is, without considering a symmetric monoidal structure) the literature on every one of these equivalent axiomatizations has as a prominent example categories of spans using the same construction that we use.
	However, every one of these formulations (inspired by partial but ultimately nondeterministic phenomena) assumes that the underlying category of ``total'' morphisms has products, and it is those categorical products in particular that recur throughout the constructions performed.
	Even the more modern approach of Cockett and Lack’s restriction categories~\cite{cockettlack2002partialmaps} only avoids this issue by not constructing a symmetric monoidal structure on arbitrary restriction categories a priori.
	However, the typical symmetric monoidal structure on restriction categories, restriction products~\cite[Proposition 4.3]{cockettlackRestrictionCategoriesIII2007}, runs into the same issue (see \cref{partializationRestrictionProducts}).
	Indeed, restriction categories with restriction products are equivalent to $p$-categories (and hence all the other equivalent notions mentioned above, see \cref{partializationRestrictionProducts}).
	Thus for our purposes it is not enough to simply use the theory developed in the preceding works.\footnote{However, many of the results in the works mentioned above can be extended to/reconstructed for the categories we consider. This may be as simple as observing that the phenomenon in question occurs in the subcategory of copyable morphisms, or at times substituting particular results about CD categories for basic facts about cartesian categories.}

	What we do in this work is to construct a symmetric monoidal structure on our category of spans that yields a CD category where monoidal product is not a restriction product.
	This is necessary for and naturally arises from our motivation in terms of stochastic maps, as the monoidal product is a restriction product if and only if the category is copyable, that is, there is no real nondeterminism.
	Our work can be seen as a non-cartesian extension of the theory of $p$-categories (and equivalent notions).
	While they are not $p$-categories in general, they always contain subcategories defined by the copyable maps which are in fact $p$-categories (essentially corresponding to the subcategory with no nondeterminism).
\end{remark}

\begin{remark}[Monoidal restriction categories]\label{tensorRestrictionDiscussion}
	Heunen and Lemay have introduced a notion of \newterm{monoidal restriction category}~\cite[Definition 4.2]{heunenpacaudlemay2020tensor}.
	These are monoidal categories that also have restriction structures commuting with tensoring.
	The main example of this is that of sets and partial functions, which is a (copyable) span category $\Par{\mathsf{Set}}$.
	However, they do not construct any non-cartesian examples which would be the probabilistically interesting ones.
	Indeed, the main focus of this paper is on a particular class of monoidal restriction categories called ``tensor restriction categories''~\cite[Section 5]{heunenpacaudlemay2020tensor}.
	They provide several equivalent definitions of this notion (which is motivated by tensor topology), one of which is that all tensor categories are essentially obtained from a construction $\mathcal S\left[\ph\right]$~\cite[Definition 3.1]{heunenpacaudlemay2020tensor}.
	However, this is not a span construction!\footnote{The construction is based on \emph{subunits}, not \emph{spans}.}
	Indeed, they note in~\cite[Examples 4.3,5.21]{heunenpacaudlemay2020tensor} that the category of sets and partial functions is not a tensor restriction category.

	The partializations $\Par{\cC}$ we construct are span categories with a compatible restriction structure, and are indeed monoidal restriction categories in this sense.
	However, as with (non-stochastic) partial functions, our partializations are generally not tensor restriction categories either.
\end{remark}

\begin{remark}[Positivity and restriction structures]\label{positiveRestrictionDiscussion}
	Another context in which a class of ``positive CD categories'' turns out to be a restriction category is that of Cioffo et al.'s recent~\cite{cioffoTaxonomyCategoriesRelations2025}.
	There, it is shown in~\cite[Proposition 4.10 and Corollary 4.12]{cioffoTaxonomyCategoriesRelations2025} that positive ``strict oplax cartesian categories''~\cite[Definition 4.5]{cioffoTaxonomyCategoriesRelations2025} are restriction categories (technically, all one needs is an inequality between the two terms in the positivity equation).
	This is in fact generalized by the present work, as it has already been shown in~\cite[Proposition 3.6]{fritz2023lax} that strict oplax cartesian categories are quasi-Markov, and we show in \cref{positiveQuasiMarkovRestriction} that positive quasi-Markov categories are restriction categories in the same way as~\cite[Corollary 4.12]{cioffoTaxonomyCategoriesRelations2025} (see \cref{posQuasiMarkovRestGeneralizesOplax} for more details).

	While oplax cartesian categories (like cartesian CD categories) a priori exclude non-deterministic behavior, the precise arguments involved are more directly related to the ``oplax discardability'' condition, rather than the unsuitable (for our purposes) ``oplax copyability''.\footnote{There is however, subtlety in that these oplax cartesian categories are assumed to be poset \emph{enriched}. Thus for instance, something like positivity would also be involved in showing that a quasi-Markov category is ``oplax discardable'' (oplax cartesian without the oplax copyability condition), as enrichment is not the case for arbitrary quasi-Markov categories (see \cref{restPositivityRequirement}).}
	Indeed, the oplax copyability is only required to show the first restriction category axiom ``R.1'', for which quasi-totality suffices.
	In light of this one may conjecture a relation between positivity and restriction structures.
	However, the relation between quasi-Markov categories and the more general~\cite[Proposition 4.10]{cioffoTaxonomyCategoriesRelations2025} (where the positivity condition is relaxed) remains to be explored.
\end{remark}

\begin{remark}[{Domain idempotents}]\label{domainIdempsHistory}
	As mentioned in \cref{stableClassAliases}, the idea of encoding the domain of a partially defined morphism or relation as an endomorphism on the source has been explored in the cartesian case in depth for $p$-categories~\cite{rosoliniContinuityEffectivenessTopoi1986} and dominical categories~\cite{paolaDominicalCategoriesRecursion1987}.
	This can be extended (as $p$-categories are merely ``restriction categories with restriction products'') to restriction categories; indeed the main data of a restriction category is precisely such a domain operator satisfying certain axioms.
	This definition can also be extended to the non-cartesian CD categorical setting by replacing the cartesian product and its associated diagonal and projections with the CD structure data, as has been done by Di Lavore et al.~\cite{di2022monoidal}, Fritz et al.~\cite[Definition 2.13]{fritz2023lax}, and Gonda et al.~\cite[Definition 2.3]{gonda2024framework} to recover \cref{domainDefinition}. However, without an additional assumption such as quasi-totality, the domain of a morphism is not necessarily an idempotent.
	In fact, quasi-totality can also be seen as one of the restriction category axioms.
	But even typical examples of quasi-Markov categories in the literature such as the category of relations $\Rel$ that have idempotent domain operators are not necessarily restriction categories, as they fail other axioms.

	We work with quasi-Markov categories throughout, so all the results for general CD categories in~\cite{di2022monoidal,fritz2023lax,gonda2024framework} apply, and furthermore the domain endomorphisms are indeed idempotent.
	Additionally, we show in the current work that given the further axiom of ``positivity'', the domain idempotents indeed are domain operators for a restriction category. This connects the two directions of generalization of the notion of domain idempotent in a $p$-category (even without the cartesian hypothesis).
\end{remark}
\begin{remark}[{Partial orders}]\label{restricitonPartialOrderHistory}
	As with domain idempotents (compare \cref{domainIdempsHistory}), in the cartesian setting the construction of \cref{def:dom_ext} is known to define an enrichment in posets for $p$-categories, dominical categories, etc.~\cite{rosoliniContinuityEffectivenessTopoi1986,paolaDominicalCategoriesRecursion1987,carboniBicategoriesPartialMaps1987} and more generally for restriction categories~\cite{cockettlack2002partialmaps}.
The construction can nonetheless be adapted to the non-cartesian CD case to recover \cref{def:dom_ext}, as has been done by Lorenz and Tull~\cite[Definition 97]{lorenz2023causalmodels} and Gonda et al.~\cite[Definition 2.5]{gonda2024framework}.\footnote{To begin, on replacing the terms  $p_{X,Y}$, $\times$, and $\Delta_X$ in~\cite[Definition preceding Lemma 2.1.3]{rosoliniContinuityEffectivenessTopoi1986} with $\id\otimes\discard$, $\otimes$, and $\cop$ one recovers the CD categorical notion of domain from \cref{domainDefinition}.
Then, the $p$-categorical notion of partial order $f \le g \iff f = g \comp \dom f$ is precisely that of the extension partial ordering $\domext$ of \cref{def:dom_ext}.
}
However, it should be noted that the relation is only transitive and antisymmetric in general, while reflexivity is precisely the quasi-totality condition.
Furthermore, even if quasi-totality is assumed (so that the relation is a partial order), it need not be an enrichment (as in the case of $\Rel$).

In the companion work~\cite[Proposition 2.15]{fritz2025empirical} it was shown that the relation $\domext$ is indeed a poset enrichment for \emph{positive} quasi-Markov categories (\cref{positiveQuasiMarkovEnriched}), even without the cartesian hypothesis.
On the other hand, restriction categories are always poset enriched, with the partial order essentially being given by extension.
In the present work we show that positive quasi-Markov categories are restriction categories (\cref{positiveQuasiMarkovRestriction}) in a way such that the partial order associated to the restriction structure is precisely $\domext$.
This thus provides a way to see that positive quasi-Markov categories are poset enriched as a consequence of the theory of restriction categories.
\end{remark}

\begin{remark}[Restriction products]\label{partializationRestrictionProducts}
	A symmetric monoidal structure on restriction categories (such as $\Par{\cC,\distMons}$) that has appeared in the literature is that of ``restriction products'', as in Cockett and Lack's~\cite[Proposition 4.3 and the preceding discussion]{cockettlackRestrictionCategoriesIII2007}.
	As mentioned earlier, the restriction categories approach to partiality has the advantage of not assuming a cartesian hypothesis a priori at the level of mere $1$-categories.
	However, while restriction products are not cartesian products in a general, they do define cartesian products on the subcategory of total morphisms.
	Indeed, if the tensor product of say, a positive quasi-Markov category is a restriction product, it is necessarily the case that all morphisms are copyable (\cref{definition:copyable} and~\cite[Proposition 4.3 (ii)]{cockettlackRestrictionCategoriesIII2007}).
	Thus, they do not suffice to define a monoidal structure for partial \emph{stochastic} maps, as the tensor product of a Markov category is cartesian only in the deterministic case~\cite{foxCoalgebrasCartesianCategories1976}.

	On the other hand, the tensor product on $\Par{\cC}$ is a restriction product when restricted to the subcategory $\Par{\cC}_\copyable$ of copyable morphisms (equivalently the partialization $\Par{\cC_\det}$).
	Indeed, any CD category with all morphisms copyable (in particular, a positive quasi-Markov category) is a $p$-category.
	But a $p$-category is equivalent to a restriction category with restriction products~\cite[Proposition 4.5]{cockettlackRestrictionCategoriesIII2007}.
\end{remark}

\begin{remark}[Domain preserving monads]
	In the recent~\cite{cioffo2025MarkovRestriction}, Cioffo et al.\ follow up on~\cite{cioffoTaxonomyCategoriesRelations2025} by considering new categories of relations: Domain categories, Mass categories, and weakly Markov categories.
	The general goal is to study classes of CD categories that have particular properties that hold in categories of relations, such as $\Rel$.
	The first of these, domain categories, are precisely quasi-Markov categories.
	However, the examples they consider are categories of weighted relations, and in particular are not span categories such as the ones we consider.
	Indeed, the focus of the paper is on the relationship between mass and domain categories, and in turn mass and domain preserving functors and monads.\footnote{The third category, weakly Markov categories are only quasi-Markov in the trivial case, and hence are unrelated to our work.}
	However given positivity, which is the setting we wok in, mass and domain categories coincide (as quasi-totality is equivalent to the domains being copyable).

	Nonetheless, they consider characterizations of when the Kleisli category of a monad is a domain category (or a mass category, or a weakly Markov category) in terms of the monad being ``domain preserving'' (or ``mass preserving'', or ``weakly Markov'' --- indeed, the focus of the paper is more on the functors and monads than the categories themselves).
	This characterization of domain categories is complementary to our results on representability in \cref{sec:representability}.
	Indeed, we show that the partialization $\Par{\cC}$ of a partializable Markov category $\cC$ is always quasi-Markov, that is, a domain category in their sense.
	Furthermore, when $\cC$ is representable, so is $\Par{\cC}$, and the distribution monad on the copyable subcategory $\Par{\cC}_\copyable$ has $\Par{\cC_\det}$ as its Kleisli category.
	But a copyable CD category is precisely a cartesian restriction category, and hence the fact that $\Par{\cC}$ can also be seen (in the representable case) as a consequence of the distribution monad being domain preserving.
	However our approach is independent of theirs, and works without the assumption of representability as we construct the partialization directly as a span category without assuming the existence of a monad to begin with.\footnote{While our approach can be seen as taking a category of stochastic maps extending it to partially defined maps, the approach their results entail is to start with a category of partial deterministic maps and add non-determinism.}
\end{remark}

\begin{remark}[{Partial Markov categories}]\label{partialMarkovCats}
Another formulation for a ``CD category of partial maps'' termed \newterm{partial Markov categories}~\cite[Definition 3.1]{dilavore2024partial} has been developed by Di Lavore et al.~\cite{dilavore2023evidential,dilavore2024partial} for the purpose of studying evidential reasoning and more general forms of inference.
These are defined to be CD categories with conditionals.
When a partializable Markov category has conditionals, so does its partialization, and hence it is a partial Markov category in this sense (although the partializations we construct, while similar in form to their examples, have different composition operations).
Indeed, our main example will be a category of partial stochastic maps between standard Borel spaces \emph{such that} all maps are quasi-total \emph{and} positive.
These conditions guarantee a restriction structure, and hence also a notion of restriction partial order.
However these properties are not satisfied by the examples of partial Markov categories they consider, indeed they do not concern themselves with restriction structures.
In fact, the example $\BorelStoch_{\le 1}$ of standard Borel spaces and sub-stochastic maps they consider is such that the quasi-total maps are not closed under composition, and thus do not even form a category.\footnote{In fact, the maps of our main example $\Par{\BorelStoch}$ can be seen as the maps of their example $\BorelStoch_{\le 1}$ that factor (as Kleisli maps) through $D\left(-\right) \sqcup \ast \rightarrowtail D\left(-\sqcup \ast\right)$. But the multiplication of the latter doesn't give the composition we need (see \cref{parCatsExmpsDiff,parCompDomainsSmall} for more detail).}

\emph{In addition}, the present work is companion work to work on formalizing the law of large numbers in a categorical setting (see \cref{companionWork}).
For this, in addition to positivity and quasi-totality of all maps, we also need the properties of Markov categories mentioned in \cref{thm:preservedprops}.
In this sense (along with having a disjoint set of examples) this work is also independent in direction of theirs; as properties such as conditionals and a distribution monad that are central to their work are not of fundamental interest here (although we do show that they transfer from a category to its partialization).\end{remark}

\begin{remark}[Conditional preorders]
	In the recent~\cite[Definition 3.1]{dilavore2025OrderPartialMarkov}, Di Lavore et al.\ introduce (in addition to the restriction partial order of \cref{def:dom_ext} that we consider) a \newterm{conditional preorder} on the morphisms of a partial Markov category (a CD category with conditionals).
	They note that conditional preorder can differ from the restriction partial order in general (even for quasi-Markov categories, such as $\Rel$~\cite[Example 3.13]{dilavore2025OrderPartialMarkov}).

	However, when the restriction partial order defines a poset \emph{enrichment} (as in the case of positive quasi-Markov categories with conditionals), their results (in particular~\cite[Proposition 3.6 and Lemma 3.7]{dilavore2025OrderPartialMarkov}) quickly imply that the conditional preorder is equivalent to the restriction partial order (as the restriction partial order is subunital).

	We do not assume the existence of conditionals for the objects of consideration in the present work, and do not consider the conditional preorder.
	However when a partializable Markov category has conditionals, we show that its partialization (a positive quasi-Markov category) has conditionals as well.
	Thus, this is a class of examples where the conditional preorder and restriction partial order coincide.
\end{remark}

\subsection*{Acknowledgements}
The author thanks Tobias Fritz, Tom\'a\v{s} Gonda, Antonio Lorenzin, Paolo Perrone, and Dario Stein for valuable suggestions and discussions, both regarding the topic of the current work and also as part of the companion work~\cite{fritz2025empirical} that motivated this work.
The author also thanks Dario Stein, Jean-Simon Lemay, Reuben Van Belle, and Chad Nester for further discussions and suggestions, particularly with regard to categorical formulations of partial maps, restriction categories, and bicategories of relations.
The author would also like to thank Cipriano Junior Cioffo, Fabio Gadducci, and Davide Trotta for comments on draft of this work.
\section{Background on partial maps and categorical probability}\label{backgroundSec}
\subsection{Categorical frameworks for partial maps}
\subsubsection{Categories of spans}
As mentioned in the introduction, a fruitful way of thinking of partial maps in a category is as equivalence classes of spans $X \hookleftarrow D \to Y$, with two spans being identified when there is an isomorphism between the apices that commutes with the span legs.
The necessary conditions for this to define a well behaved category are usually encoded by requiring the left leg of the span, the \newterm{wrong way map} $D \hookrightarrow X$ to belong to a suitable class of maps.
This leads to the following definition.\footnote{We use the terminology of Cockett and Lack~\cite{cockettlack2002partialmaps}, however this is a much older notion. A discussion of the history of this concept can be found at \cref{stableClassAliases}.}
\begin{definition}[{\cite[Section 3.1]{cockettlack2002partialmaps}}]\label{stableSysDefinition}
	A \newterm{stable system of monomorphisms} in a category is a class of monomorphisms $\distMons$ such that:
	\begin{enumerate}
		\item $\distMons$ is closed under composition;
		\item $\distMons$ contains all isomorphisms;
		\item $\distMons$ is closed under pullbacks along arbitrary maps.
	\end{enumerate}
\end{definition}
\begin{construction}[{\cite[Section 3.1]{cockettlack2002partialmaps}}]\label{spanCatDefinition}
	Given a stable system of monomorphisms $\distMons$ in a category $\cC$, one constructs a \newterm{category of spans} $\Par{\cC,\distMons}$ with wrong way maps in $\distMons$.
	This category has:
	\begin{enumerate}
		\item Objects those of the original category $\cat{C}$;
		\item Maps $X\to Y$ equivalence classes of spans
			\[
            X \xhookleftarrow{i} D \xrightarrow{f} Y
        \]
		with $i$ in $\distMons$, where two spans $X \xhookleftarrow{i} D_f \xrightarrow{f} Y$ and $X \xhookleftarrow{j} D_g \xrightarrow{g} Y$ are considered equivalent if and only if there is an isomorphism $D_f\xrightarrow{\simeq} D_g$ making the following diagram commute:
		\[
		\begin{tikzcd}[sep=tiny]
			& D_f \ar[hook]{dl}[swap]{i} \ar{dr}{f} \ar[dd, leftrightarrow,"\simeq" {anchor=south, rotate=270}] \\
			X && Y \\
			& D_g \ar[hook']{ul}{j} \ar{ur}[swap]{g}
		\end{tikzcd}
		\]
		\item Composition is done by pullback: For maps represented by spans $X \xhookleftarrow{i} D_f \xrightarrow{f} Y$ and $Y \xhookleftarrow{j} D_g \xrightarrow{g} Z$,
		\[
		\begin{tikzcd}[sep=tiny, ampersand replacement=\&]
			\& \& E \arrow[ld, "u"', hook] \arrow[rd, "v"] \arrow["\lrcorner"{anchor=center, pos=0.125, rotate=-45}, draw=none, dd] \& \& \\
			\& D_f \arrow[ld, "i"', hook] \arrow[rd, "f"'] \& \& D_g \arrow[ld, "j", hook] \arrow[rd, "g"] \& \\
			X \& \& Y \& \& Z
		\end{tikzcd}
		\]
		the composite is represented by $X \xhookleftarrow{i\comp u} E \xrightarrow{g \comp v} Z$.
\end{enumerate}
\end{construction}

There is a canonical inclusion of $\cC$ into $\Par{\cC,\distMons}$ sending $X \xrightarrow{f} Y$ to the (equivalence class) of the span $X = X \xrightarrow{f}Y$.
We call the morphisms in the image of this inclusion \newterm{total}.
The inclusion is an equivalence on to the subcategory of $\Par{\cC,\distMons}$ comprising the total morphisms.

Such categories of spans are canonically enriched in posets, where a span $X \xhookleftarrow{i} D_f \xrightarrow{f} Y$ is bounded above by $X \xhookleftarrow{j} D_g \xrightarrow{g} Y$ if and only if there is a $D_f\to D_g$ making the following diagram commute:
		\[
		\begin{tikzcd}[sep=tiny]
			& D_f \ar[hook]{dl}[swap]{i} \ar{dr}{f} \ar[dd] \\
			X && Y \\
			& D_g \ar[hook']{ul}{j} \ar{ur}[swap]{g}
		\end{tikzcd}
		\]
		\subsubsection{Restriction categories}
Restriction categories are a versatile framework for dealing with partial morphisms, and have the advantage that they do not involve any cartesianity (at a $1$-categorical level).
Furthermore, the restriction structure also provides an enrichment in posets, which generalizes that of $p$-categories (and equivalent formulations)~\cite[Section 2.1.4 and Examples 2.1.3 (6)]{cockettlack2002partialmaps}.

\begin{definition}[{\cite[Section 2.1.1]{cockettlack2002partialmaps}}]\label{restrictionCategoryDefinition}
	A \newterm{restriction structure} on a category $\cC$ is an assignment of an endomorphism\footnote{The intuition is that pre-composing with $\bar f$ restricts a morphism to the ``domain of definition'' of $f$.} $\bar f\colon X \to X$ to each morphism $f\colon X \to Y$, such that:
	\begin{enumerate}[label={(R.\arabic*):}]
		\item $f\comp \bar f = f$ (restricting $f$ to its domain leaves it unchanged);
		\item $\bar f \comp \bar g = \bar g \comp \bar f$ when $f$ and $g$ have the same source (restrictions commute);
		\item $\overline{g \comp \bar f} = \bar g \comp \bar f$ when $f$ and $g$ have the same source (restricting to the domain of a restriction is the same as restricting separately);
		\item $\bar g \comp f = f \comp \overline{g \comp f}$ when $f$ and $g$ are composable (restricting ``onto the image'' is the same as restricting to the domain of the composite).
	\end{enumerate}
	A \newterm{restriction category} is a category equipped with a restriction structure.

	When $\cC$ is a CD category, we will be interested in the case where it has a restriction structure given by $\bar f = \dom f$ with $\dom f$ the domain of $f$ as in \cref{domainDefinition}.
\end{definition}

Categories of spans such as $\Par{\cC,\distMons}$ have a natural restriction structure~\cite[Proposition 3.1]{cockettlack2002partialmaps}, where the ``domain idempotent'' of a map $X \xhookleftarrow{i} D \xrightarrow{f} Y$ is the map $X \xhookleftarrow{i} D \xhookrightarrow{i} x$.
The poset enrichment induced by this restriction structure is precisely the one mentioned above.

\subsection{Review of categorical probability}
\subsubsection{CD and Markov categories}
CD and Markov categories are a modern abstract generalization of the category of measurable spaces and Markov kernels.
Much of the foundational material on categorical probability via Markov categories revolves around the idea that the main concepts of probability theory, such as statistical (in)dependence, determinism, conditioning, etc., can be meaningfully extended from categories of Markov kernels to more general Markov or CD categories.
Many of these notions can be captured in categorical terms by means of universal properties (limits, colimits, representable functors, etc.) or in terms of categorical structures such as symmetric monoidal categories.

A general reference for many of these results can be found in the work of Cho and Jacobs~\cite{chojacobs2019strings}, and Fritz~\cite{fritz2019synthetic}.
Results that have been developed like this over the past few years feature the Kolmogorov and Hewitt--Savage zero-one laws~\cite{fritzrischel2019zeroone}, the de Finetti theorem~\cite{fritz2021definetti,moss2022probability}, the d-separation theorem for graphical models~\cite{fritz2022dseparation}, the ergodic decomposition theorem~\cite{moss2022ergodic,ensarguet2023ergodic}, the Blackwell--Sherman--Stein theorem~\cite{fritz2023representable} and the Aldous--Hoover theorem~\cite{chen2024aldoushoover}.

\begin{definition}[{\cite[Definition 13 and Definition 2.2]{corradini1999algebraic,chojacobs2019strings}}]
	A \newterm{CD category}\footnote{These are also called \newterm{gs-monoidal} categories.} is a symmetric monoidal category $\cC$ in which every object $X$ is equipped with a distinguished commutative comonoid structure
    \[
		\cop_X\colon X \to X \otimes X, \qquad \discard_X \colon X \to \mathcal I,
	\]
    denoted in string diagram\footnote{Our string diagrams run bottom to top (and left to right).} notation by
	
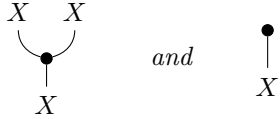
\[
		\begin{tikzpicture}
	\begin{pgfonlayer}{nodelayer}
		\node [style=none] (0) at (0, -0.75) {};
		\node [style=none] (1) at (0.75, 0.75) {};
		\node [style=none] (2) at (-0.75, 0.75) {};
		\node [style=bn] (3) at (0, 0) {};
		\node [style=none] (4) at (-0.75, 1.25) {$X$};
		\node [style=none] (5) at (0.75, 1.25) {$X$};
		\node [style=none] (6) at (0, -1.25) {$X$};
	\end{pgfonlayer}
	\begin{pgfonlayer}{edgelayer}
		\draw (3) to (0.center);
		\draw [in=-90, out=15] (3) to (1.center);
		\draw [in=165, out=-90] (2.center) to (3);
	\end{pgfonlayer}
\end{tikzpicture} \qquad \text{ and } \qquad \begin{tikzpicture}
	\begin{pgfonlayer}{nodelayer}
		\node [style=bn] (0) at (0, 0.75) {};
		\node [style=none] (1) at (0, -0.25) {};
		\node [style=none] (2) at (0, -0.75) {$X$};
	\end{pgfonlayer}
	\begin{pgfonlayer}{edgelayer}
		\draw (0) to (1.center);
	\end{pgfonlayer}
\end{tikzpicture}

	\]
	which is compatible with the tensor product in the sense that for all objects $X$ and $Y$,\footnote{This condition is often called \newterm{uniformity}~\cite[Definition 2.1]{dilavore2023evidential}.}
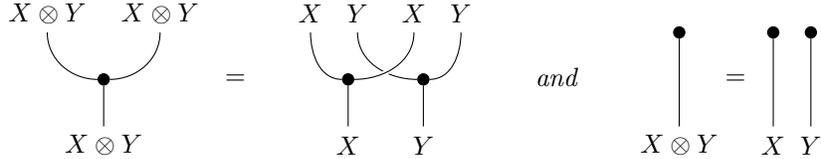
	\[
        \begin{tikzpicture}
	\begin{pgfonlayer}{nodelayer}
		\node [style=bn] (0) at (-3.5, 0) {};
		\node [style=none] (3) at (-2, 1.25) {};
		\node [style=none] (4) at (-5, 1.25) {};
		\node [style=none] (5) at (-3.5, -1.25) {};
		\node [style=none] (8) at (-2, 1.75) {$X\otimes Y$};
		\node [style=none] (9) at (-5, 1.75) {$X\otimes Y$};
		\node [style=none] (10) at (-3.5, -1.75) {$X\otimes Y$};
		\node [style=bn] (11) at (3, 0) {};
		\node [style=bn] (12) at (5, 0) {};
		\node [style=none] (13) at (2, 1.25) {};
		\node [style=none] (14) at (6, 1.25) {};
		\node [style=none] (15) at (4.75, 1.25) {};
		\node [style=none] (16) at (3.25, 1.25) {};
		\node [style=none] (17) at (3, -1.25) {};
		\node [style=none] (18) at (5, -1.25) {};
		\node [style=none] (19) at (3, -1.75) {$X$};
		\node [style=none] (20) at (5, -1.75) {$Y$};
		\node [style=none] (21) at (0, 0) {$=$};
		\node [style=none] (22) at (2, 1.75) {$X$};
		\node [style=none] (23) at (3.25, 1.75) {$Y$};
		\node [style=none] (24) at (4.75, 1.75) {$X$};
		\node [style=none] (25) at (6, 1.75) {$Y$};
	\end{pgfonlayer}
	\begin{pgfonlayer}{edgelayer}
		\draw (0) to (5.center);
		\draw [in=-90, out=0] (0) to (3.center);
		\draw [in=180, out=-90] (4.center) to (0);
		\draw [in=-90, out=180] (11) to (13.center);
		\draw [in=180, out=-90] (16.center) to (12);
		\draw [in=-90, out=0] (12) to (14.center);
		\draw (11) to (17.center);
		\draw (18.center) to (12);
		\draw [style=protected, in=0, out=-90] (15.center) to (11);
	\end{pgfonlayer}
\end{tikzpicture}  \qquad \text{ and } \qquad  \begin{tikzpicture}
	\begin{pgfonlayer}{nodelayer}
		\node [style=none] (0) at (0, 0) {$=$};
		\node [style=none] (1) at (-1.5, -1.25) {};
		\node [style=none] (3) at (-1.5, -1.75) {$X \otimes Y$};
		\node [style=bn] (4) at (-1.5, 1.25) {};
		\node [style=bn] (5) at (1, 1.25) {};
		\node [style=bn] (6) at (2, 1.25) {};
		\node [style=none] (7) at (2, -1.75) {$Y$};
		\node [style=none] (8) at (1, -1.75) {$X$};
		\node [style=none] (9) at (1, -1.25) {};
		\node [style=none] (10) at (2, -1.25) {};
	\end{pgfonlayer}
	\begin{pgfonlayer}{edgelayer}
		\draw (4) to (1.center);
		\draw (5) to (9.center);
		\draw (6) to (10.center);
	\end{pgfonlayer}
\end{tikzpicture}
	\]
    and further such that $\cop_{\mathcal I} = \discard_{\mathcal I} = \id_{\mathcal I}$ (up to monoidal coherence).
 \end{definition}
 The intuition here is that maps correspond to ``random processes'' such that composition corresponds to running them in sequence, while the monoidal product corresponds to running them in parallel.
 The copy and delete maps then correspond to the ability to copy a random datum and to discard it, with the compatibility conditions then corresponding to intuitive properties of these operations.
 The CD approach to probability is based on the idea that these properties contain the fundamental behavior of information flow underlying much of probability theory.

 \begin{example}\label{CDCatsExamples}
We recall a few CD categories from the categorical probability literature that appear in this work.
We will focus on the intuition behind the definitions, and defer to the cited references for full details.
	\begin{enumerate}
		\item\label{FinStochBackground} The category $\FinStoch$ of finite sets and stochastic maps has as objects finite sets and maps $X \to Y$ defined by the data of discrete distributions $f\left(\ph\,\vert\, x\right)$ on $Y$ for each $x \in X$~\cite[Example 2.5]{fritz2019synthetic}.
	These discrete distributions are defined entirely by the data of probabilities $f\left(y \,\vert\, x\right)$ for $y\in Y$, corresponding to the probability that the map returns $y$ on input $x$.
	Thus, the morphisms are equivalently described by stochastic matrices (ones such that the sum of each row is $1$), and the composition is given by matrix multiplication (the Chapman--Kolmogorov equation).

	The symmetric monoidal structure is given by cartesian product on objects, and matrix tensor on maps. This corresponds to the idea that the two stochastic maps are tensored independently, that is, $\left(f \otimes g\right)\left(y_1,y_2\,\vert\,x_1,x_2\right) = f\left(y_1\,\vert\,x_1\right)g\left(y_2\,\vert\,x_2\right)$.
	Ordinary (non-stochastic) functions are the maps of this category such that the distributions $f\left(\ph\,\vert\, x\right)$ are Dirac measures (concentrated on the ``image'' of $x$ under $f$).
	The CD structure is given in this manner by the diagonal maps and the unique maps to the one-point space.
\item\label{SetMultiBackground} The category $\SetMulti$ has as objects sets and as maps multivalued functions, where the maps $X \xrightarrow{f} Y$ are given by the data of a \emph{nonempty} image \emph{subset} $f\left(x\right) \subseteq Y$ for each point $x \in X$~\cite[Example 2.6]{fritz2019synthetic} (it can equivalently be seen as the subcategory of $\Rel$ defined by total relations).
	This can be seen as the Kleisli category of the nonempty power set monad on $\cat{Set}$.
	It can be equipped with a CD structure in a manner analogous to that of $\FinStoch$, with the tensor product being given by the product of relations, and the copy and delete maps being given by the diagonal and unique relations to the one-point space.

			It also has a full Markov subcategory $\FinSetMulti$ spanned by the finite sets.
			$\SetMulti$ (or better yet $\FinSetMulti$) can be seen as a ``possibilistic'' analogue of $\FinStoch$, where instead of probabilities of particular outcomes, we record whether they are possible or not.
\item\label{Kl(DR)Background} For a commutative semiring $R$, Coumans and Jacobs in~\cite[Section 5.1]{coumans2013scalars} have constructed a monad of ``$R$-valued finitely supported distributions'' (the idea being to consider flavours of discrete probability in more general semirings than the reals) on the category $\mathsf{Set}$ of sets and functions.
	As observed in~\cite[Example 3.3]{fritz2023representable}, its Kleisli category $\Kl\left(D_R\right)$ is a Markov category whose objects can be thought of as the category of discrete measurable spaces and whose maps are kernels
	\[
		f\colon X \to Y,\qquad x \mapsto \sum_{y \in Y} f\left(y\, \vert\, x\right)\delta_y
	\]
	such that $\supp{f\left(\ph\,\vert\,x\right)} \coloneqq \left\{ y \in Y : f\left(y\,\vert\,x\right) \ne 0\right\}$ is finite.

	As with $\FinStoch$, composition is given by the Chapman--Kolmogorov equation, so that the composite of a sequence $X \xrightarrow{f}Y \xrightarrow{g} Z$ is described by the equation
	\[
		\left(g \circ f\right)\left(z \,\vert\, x\right) = \sum_{y \in Y} g\left(z \,\vert\, y\right) f\left(y \,\vert\, x\right)
	\]
	The tensor and copy--discard structure are defined in the same way as in $\FinStoch$ (see special case~\ref{DistBackground}).

This has two particular instances that we will keep returning to: 
	\begin{enumerate}
		\item\label{DistBackground} The category $\Dist$ of sets and discrete distributions, where the morphisms $X \to Y$ are given by the data of a discrete distribution on $Y$ for each $x \in X$.
			This is the instance where the semiring is $\Rnneg$, and is essentially a version of $\FinStoch$ where the sets are allowed to be infinite.
		\item\label{SetFinMultiBackground} The category $\cat{SetFinMulti}$ of sets and ``finitely multivalued maps''. This is the wide subcategory of $\SetMulti$ whose maps are the multivalued functions $X \xrightarrow{f} Y$ such that for each $x \in X$, the image $f\left(x\right)$ is finite (and nonempty).
			This is the instance where the semiring is the Boolean semiring $\left\{0,1\right\}$ (with addition and multiplication being disjunction and conjunction).\footnote{The reason this category only contains the ``finitely multivalued'' functions is that the distribution monad only considers finitely supported distributions.}
			Its full Markov subcategory spanned by the finite sets is also $\FinSetMulti$.
	\end{enumerate}

\item\label{BorelStochBackground} Finally, our main example $\BorelStoch$~\cite[Section 4]{fritz2019synthetic} is an analogous category of standard Borel spaces and general stochastic maps (Markov kernels).
	In the non-discrete case, probabilities of individual points no longer suffice to determine the entire distribution.
	Thus we take here as objects standard Borel spaces (Polish spaces)\footnote{These tend to be much better behaved than arbitrary measurable spaces.} and one takes as stochastic maps Markov kernels.

	These are assignments of a probability measure $f\left(\ph\,\vert\, x\right)$ on the target space $Y$ for each point $x \in X$, such that the map is measurable in the sense that for every Borel set $B \subseteq Y$, the map $\mathrm{ev}_B\colon X \to \left[0,1\right]$ given by $x \mapsto f\left(B\,\vert\, x\right)$ is measurable.
	Composition is performed by a continuous variant of the Chapman--Kolmogorov equation.
	For a sequence of maps $X \xrightarrow{f} Y \xrightarrow{g} Z$, the composite is given by
	\[
		\left(g \circ f\right)\left(B\,\vert\, x\right) = \int_Y g\left(B\,\vert\, y\right) f\left(dy\,\vert\, x\right)
	\]
	As with $\FinStoch$, the tensor product is given by the product of spaces, and on maps by taking the independent product of distributions. That is, $(f \otimes g)\left(\ph\,\vert\, x, y\right)$ is the unique distribution such that $(f \otimes g)\left(A \times B\,\vert\, x, y\right) = f\left(A\,\vert\, x\right) g\left(B\,\vert\, y\right)$ for arbitrary Borel sets $A,B$.
	The CD structure is defined identically to the discrete case as well.
\end{enumerate}
\end{example}

 \subsubsection{Totality and copyability}
The maps of a CD category are typically	interpreted as ``partially defined random maps''.
One intuitive special case of this notion is those maps intuitively have full domain, the total maps (in terms of being totally defined).
If we think of deletion in a CD category as discarding the output of a map, this still retains the data of whether that output was defined or not.
This lets us capture the idea of totality in terms of CD structure data alone.

 \begin{definition}[{\cite[Definition 2.3 and Definition 2.1]{chojacobs2019strings,fritz2019synthetic}}]
        A map $f\colon X \to Y$ is \newterm{total} if it commutes with deletion
			\[
				\begin{tikzpicture}
	\begin{pgfonlayer}{nodelayer}
		\node [style=none] (0) at (0, 0) {$=$};
		\node [style=morphism] (1) at (-2, 0) {$f$};
		\node [style=none] (2) at (2, -1.5) {};
		\node [style=none] (3) at (-2, -1.5) {};
		\node [style=bn] (4) at (2, 1.5) {};
		\node [style=bn] (5) at (-2, 1.5) {};
	\end{pgfonlayer}
	\begin{pgfonlayer}{edgelayer}
		\draw (5) to (1);
		\draw (1) to (3.center);
		\draw (2.center) to (4);
	\end{pgfonlayer}
\end{tikzpicture}
			\]
        A \newterm{Markov category} is a CD category in which every map is total.
\end{definition}
			Indeed, this phenomenon can be observed in the (non-probabilistic) category of sets and partial functions, where the total maps are the functions (with full domain).\footnote{Indeed, many authors use the terminology ``functional'' for this notion.}

The complementary special case is that of the ``non-probabilistic'' partially defined maps, whose sub-case in the theory of Markov categories is the concept of determinism.
This works by characterizing the ``non-probabilistic'' maps in terms of the copy maps.
We term the direct generalization of the definition to CD categories ``copyability'', reserving the term ``determinism'' for those morphisms that are additionally total.\footnote{In this, we follow Carboni and Walters~\cite{carboniwalters1987bicartesian} in distinguishing between the comultiplication homomorphisms that are counit homomorphisms and those that aren't.}

\begin{definition}[{\cite[Definition 10.1]{fritz2019synthetic}}]\label{definition:copyable}
	A morphism $X \xrightarrow{f} Y$ in a CD category $\cC$ is called \newterm{copyable} if
	\[
 				\begin{tikzpicture}
	\begin{pgfonlayer}{nodelayer}
		\node [style=none] (0) at (-3, -1.5) {};
		\node [style=bn] (1) at (-3, -0.5) {};
		\node [style=none] (2) at (-4, 0.5) {};
		\node [style=none] (3) at (-2, 0.5) {};
		\node [style=none] (4) at (-2, 0.5) {};
		\node [style=morphism] (5) at (-4, 0.75) {$f$};
		\node [style=morphism] (6) at (-2, 0.75) {$f$};
		\node [style=none] (7) at (3, -1.5) {};
		\node [style=bn] (8) at (3, 0.5) {};
		\node [style=none] (9) at (2, 1.5) {};
		\node [style=none] (10) at (4, 1.5) {};
		\node [style=none] (11) at (4, 1.5) {};
		\node [style=none] (12) at (-4, 1.75) {};
		\node [style=none] (13) at (-2, 1.75) {};
		\node [style=none] (14) at (2, 1.75) {};
		\node [style=none] (15) at (4, 1.75) {};
		\node [style=none] (16) at (0, 0) {$=$};
		\node [style=morphism] (17) at (3, -0.5) {$f$};
		\node [style=none] (18) at (-4, 2.25) {$Y$};
		\node [style=none] (19) at (-2, 2.25) {$Y$};
		\node [style=none] (20) at (2, 2.25) {$Y$};
		\node [style=none] (21) at (4, 2.25) {$Y$};
		\node [style=none] (22) at (-3, -2) {$X$};
		\node [style=none] (23) at (3, -2) {$X$};
	\end{pgfonlayer}
	\begin{pgfonlayer}{edgelayer}
		\draw [style=none] (0.center) to (1);
		\draw [style=none, in=270, out=165] (1) to (2.center);
		\draw [style=none, in=270, out=15] (1) to (3.center);
		\draw [style=none] (7.center) to (8);
		\draw [style=none, in=270, out=165] (8) to (9.center);
		\draw [style=none, in=270, out=15] (8) to (10.center);
		\draw (12.center) to (5);
		\draw (13.center) to (6);
		\draw (14.center) to (9.center);
		\draw (15.center) to (11.center);
		\draw (2.center) to (5);
		\draw (4.center) to (6);
	\end{pgfonlayer}
\end{tikzpicture}
	\]
	The copyable morphisms form a wide subcategory $\cC_{\copyable}$ of $\cC$.

	We call a morphism of a CD category \newterm{deterministic} if it is both total and copyable\footnote{In particular, copyability and determinism coincide in Markov categories.}.
\end{definition}

These are intuitively the maps that have no randomness; the definition can be read as saying that copying the output of a single run of the process is equivalent to running two independent copies of the process (on the same input).
In practice, the deterministic maps tend to correspond to some kind of (measurable) function.

\subsubsection{Domains and quasi-totality}
Categories of spans have natural restriction structures, which involve a notion of ``domain idempotent'' defined in the language of spans.
On the other hand, one can associate to any map of a CD category an idempotent on its domain. In practice, this comprises the data of its ``domain of definition''.

\begin{definition}[{\cite[Definition 2.13 and Definition 2.3]{fritz2023lax,gonda2024framework}}]\label{domainDefinition}
	We define the domain $\dom f$ of a morphism $f$ of a CD category $\cC$ to be the (endo)morphism
	\[
		\begin{tikzpicture}
	\begin{pgfonlayer}{nodelayer}
		\node [style=morphism] (12) at (1, 0.5) {$f$};
		\node [style=none] (13) at (-1, 0.25) {};
		\node [style=none] (14) at (1, 0.25) {};
		\node [style=none] (19) at (-1, 2) {};
		\node [style=none] (20) at (-1, 2.5) {$X$};
		\node [style=bn] (21) at (0, -1) {};
		\node [style=bn] (22) at (1, 1.75) {};
		\node [style=none] (23) at (0, -2) {};
		\node [style=none] (24) at (0, -2.5) {$X$};
	\end{pgfonlayer}
	\begin{pgfonlayer}{edgelayer}
		\draw [in=-90, out=165] (21) to (13.center);
		\draw [in=-90, out=15] (21) to (14.center);
		\draw (12) to (22);
		\draw (13.center) to (19.center);
		\draw (23.center) to (21);
	\end{pgfonlayer}
\end{tikzpicture}
	\]
\end{definition}
The intuition here is that the composite $\discard \comp f$ encapsulates the ``domain of definition'' of $f$ (it has also been called the ``probability of success''~\cite[Remark 2.18]{dilavore2023evidential}\footnote{This is somewhat revisionist, as the cited article actually calls it the ``probability of failure''.}).
However it is often more convenient to work with its pre-composition with the copy map, in particular as we will work in the context where this endomorphism has the richer structure of an idempotent.\footnote{A further discussion can be found in \cref{domainIdempsHistory}.}

Note that a morphism is total if and only if its domain is the identity (which coincides with the domain being total).
On the other hand, we will also be interested in the complementary case where the domain is copyable.
This leads to the following definition (which is technically stronger a-priori, however equivalent in the categories we consider)

\begin{definition}[{\cite[Definition 3.1 and Definition 8]{dilavore2023evidential,lorenz2023causalmodels}}]\label{def:quasi-total}
		A morphism $f \colon A \to X$ in a CD category is \newterm{quasi-total}\footnote{We follow the terminology ``quasi-total'' as used by Di Lavore and Rom\'an in~\cite[Definition~3.1]{dilavore2023evidential}, generalizing the notion of totality for CD categories. Essentially the same notion has also been termed a ``partial channel'' by Lorenz and Tull in~\cite[preceding Definition~8]{lorenz2023causalmodels}.} if absorbs its domain, that is if
		\begin{equation}\label{eq:quasi-total}
			\begin{tikzpicture}
	\begin{pgfonlayer}{nodelayer}
		\node [style=morphism] (12) at (1, 0.75) {$f$};
		\node [style=none] (13) at (-1, 0.25) {\hspace*{\fill}~};
		\node [style=none] (14) at (1, 0.25) {};
		\node [style=none] (19) at (-1, 2) {};
		\node [style=none] (20) at (-1, 2.5) {$X$};
		\node [style=bn] (21) at (0, -1) {};
		\node [style=bn] (22) at (1, 2) {};
		\node [style=none] (23) at (0, -2) {};
		\node [style=none] (24) at (0, -2.5) {$A$};
		\node [style=morphism] (33) at (-1, 0.75) {$f$};
	\end{pgfonlayer}
	\begin{pgfonlayer}{edgelayer}
		\draw [in=-90, out=180] (21) to (13.center);
		\draw [in=-90, out=0] (21) to (14.center);
		\draw (12) to (22);
		\draw (13.center) to (19.center);
		\draw (23.center) to (21);
		\draw (12) to (14.center);
	\end{pgfonlayer}
\end{tikzpicture}\qquad = \qquad\begin{tikzpicture}
	\begin{pgfonlayer}{nodelayer}
		\node [style=morphism] (26) at (0, 0) {$f$};
		\node [style=none] (29) at (0, 2) {};
		\node [style=none] (30) at (0, 2.5) {$X$};
		\node [style=none] (31) at (0, -2) {};
		\node [style=none] (32) at (0, -2.5) {$A$};
	\end{pgfonlayer}
	\begin{pgfonlayer}{edgelayer}
		\draw (31.center) to (26);
		\draw (26) to (29.center);
	\end{pgfonlayer}
\end{tikzpicture}
		\end{equation}
		We call a CD category \newterm{quasi-Markov} if every morphism is quasi-total.
	\end{definition}
	The upshot of quasi-totality is that the domain of a quasi-total map is both copyable and an idempotent.
	In fact, if the CD category is positive (in the sense of \cref{positivityDefinition}), then quasi-totality is equivalent to the copyability of the domain.\footnote{This is an immediate consequence of the definition of positivity, however the result appears already Di Lavore and Rom\'an's~\cite[Proposition~3.5]{dilavore2023evidential} under the alternative hypothesis of conditionals. While it is not the case that conditionals imply positivity for general quasi-Markov categories, both positivity and conditionals independently imply ``copyable marginal independence'', an adaptation of the ``deterministic marginal independence'' of Fritz et al.'s~\cite[Definition 2.4]{fritz2022dilations} with copyability instead of determinism. It again follows rapidly from the definition that given CMI, copyability of the domain implies quasi-totality.}

	In terms of normalizations,\footnote{In the sense of Di Lavore and Rom\'an~\cite[Definition 3.20]{dilavore2023evidential}, or Lorenz and Tull in~\cite[Definition 97]{lorenz2023causalmodels} (indeed they define what they call quasi-total morphisms, ``partial channels'' in this manner).} quasi-total morphisms can be seen as those that are their own normalization.
	Gonda et al.~\cite{gonda2024framework} observe that both copyable and total morphisms are quasi-total, thus the class of quasi-Markov categories includes both the copyable CD categories (where there is intuitively no randomness) and Markov categories.
	Note that the former class is equivalently the class of restriction categories with restriction products of~\cite[Proposition 4.3]{cockettlackRestrictionCategoriesIII2007}, which is in turn equivalent to classical notions of partial maps such as $p$-categories~\cite[Proposition 4.5]{cockettlackRestrictionCategoriesIII2007} (and others, see \cref{stableClassAliases}).
	Indeed, both these classes have a common generalization in terms of (strict) oplax cartesian categories~\cite[Definition 4.5]{cioffoTaxonomyCategoriesRelations2025}, which are themselves special cases of quasi-Markov categories~\cite[Proposition 3.6]{fritz2023lax}.

	Pre-composition with the domain of a morphism can be seen as restricting to its domain of definition.
	Intuitively, this lets one define an order on hom-sets purely in terms of CD structure, as follows.
	\begin{definition}[{\cite[Definition 97 and Definition 2.5]{lorenz2023causalmodels,gonda2024framework}}]\label{def:dom_ext}
	For any two parallel morphisms $f, g \colon A \to X$ in a quasi-Markov category, we say that $f$ \newterm{extends} $g$, denoted $f \domext g$, if we have
	\begin{equation}\label{eq:eq_on_domain}
		\begin{tikzpicture}
	\begin{pgfonlayer}{nodelayer}
		\node [style=morphism] (12) at (-2, 0.5) {$g$};
		\node [style=none] (13) at (-4, 0.25) {};
		\node [style=none] (14) at (-2, 0.25) {};
		\node [style=none] (19) at (-4, 2) {};
		\node [style=none] (20) at (-4, 2.5) {$X$};
		\node [style=bn] (21) at (-3, -1) {};
		\node [style=bn] (22) at (-2, 1.5) {};
		\node [style=none] (23) at (-3, -2) {};
		\node [style=none] (24) at (-3, -2.5) {$A$};
		\node [style=none] (25) at (0, 0) {$=$};
		\node [style=morphism] (26) at (2.5, 0) {$g$};
		\node [style=none] (29) at (2.5, 2) {};
		\node [style=none] (30) at (2.5, 2.5) {$X$};
		\node [style=none] (31) at (2.5, -2) {};
		\node [style=none] (32) at (2.5, -2.5) {$A$};
		\node [style=morphism] (33) at (-4, 0.5) {$f$};
	\end{pgfonlayer}
	\begin{pgfonlayer}{edgelayer}
		\draw [in=-90, out=165] (21) to (13.center);
		\draw [in=-90, out=15] (21) to (14.center);
		\draw (12) to (22);
		\draw (23.center) to (21);
		\draw (31.center) to (26);
		\draw (26) to (29.center);
		\draw (13.center) to (33);
		\draw (33) to (19.center);
	\end{pgfonlayer}
\end{tikzpicture}
	\end{equation}
	This defines a partial order on the hom-sets of a quasi-Markov category~\cite[Lemma 98]{lorenz2023causalmodels}.\footnote{To be precise, Lorenz and Tull~\cite[Lemma 98]{lorenz2023causalmodels} show that this defines an antisymmetric and transitive relation on the hom-sets of a CD category.
Reflexivity of this relation is merely the quasi-totality equation~\ref{eq:quasi-total}, as observed by~\cite{gonda2024framework}.}
\end{definition}
Thinking of $\discard \comp g$ as the domain of definition of $g$, we can see this as the statement that $g$ is the restriction of $f$ to the domain of $g$.\footnote{A discussion of the non-probabilistic precursor to this construction can be found at \cref{restricitonPartialOrderHistory}.}

\subsubsection{Positivity and enrichment}
A property typical of ``probability-like'' Markov categories is ``positivity'', its name deriving from the fact that typical examples correspond to ``non-signed''/nonnegative probability measures, for instance in Fritz's~\cite[Example 11.25]{fritz2019synthetic}.

\begin{definition}[{\cite[Definition~2.11]{fritz2025empirical}, generalizing \cite[Definition 11.22]{fritz2019synthetic}}]\label{positivityDefinition}
	We call a quasi-Markov category \newterm{positive} if for every composable pair $X \xrightarrow{u} Y \xrightarrow{v} Z$ whose composite $v\comp u\colon X\to Y$ is copyable,
	\[
		\begin{tikzpicture}
	\begin{pgfonlayer}{nodelayer}
		\node [style=none] (0) at (3, -2) {};
		\node [style=bn] (1) at (3, 0) {};
		\node [style=none] (2) at (4, 1) {};
		\node [style=none] (3) at (2, 1) {};
		\node [style=none] (4) at (2, 1) {};
		\node [style=morphism] (5) at (4, 1.25) {$v$};
		\node [style=none] (6) at (2, 1.25) {};
		\node [style=none] (7) at (-3, -2) {};
		\node [style=bn] (8) at (-3, -1.25) {};
		\node [style=none] (9) at (-2, 0) {};
		\node [style=none] (10) at (-4, 0) {};
		\node [style=none] (11) at (-4, 0) {};
		\node [style=none] (12) at (4, 2.25) {};
		\node [style=none] (13) at (2, 2.25) {};
		\node [style=none] (14) at (-2, 2.25) {};
		\node [style=none] (15) at (-4, 2.25) {};
		\node [style=none] (16) at (0, 0) {$=$};
		\node [style=morphism] (17) at (3, -1) {$u$};
		\node [style=none] (18) at (4, 2.75) {$Z$};
		\node [style=none] (19) at (2, 2.75) {$Y$};
		\node [style=none] (20) at (-2, 2.75) {$Z$};
		\node [style=none] (21) at (-4, 2.75) {$Y$};
		\node [style=none] (22) at (3, -2.5) {$X$};
		\node [style=none] (23) at (-3, -2.5) {$X$};
		\node [style=morphism] (24) at (-2, 0) {$u$};
		\node [style=morphism] (25) at (-4, 0) {$u$};
		\node [style=morphism] (26) at (-2, 1.25) {$v$};
	\end{pgfonlayer}
	\begin{pgfonlayer}{edgelayer}
		\draw [style=none] (0.center) to (1);
		\draw [style=none, bend right=45] (1) to (2.center);
		\draw [style=none, bend left=45] (1) to (3.center);
		\draw [style=none] (7.center) to (8);
		\draw [style=none, bend right=45] (8) to (9.center);
		\draw [style=none, bend left=45] (8) to (10.center);
		\draw (12.center) to (5);
		\draw (13.center) to (6.center);
		\draw (14.center) to (9.center);
		\draw (15.center) to (11.center);
		\draw (2.center) to (5);
		\draw (4.center) to (6.center);
	\end{pgfonlayer}
\end{tikzpicture}
	\]
\end{definition}

\begin{proposition}[{\cite[Proposition~2.15]{fritz2025empirical}}]\label{positiveQuasiMarkovEnriched}
	Consider a positive quasi-Markov category. The partial order of \cref{def:dom_ext} is an enrichment\footnote{This subtlety regarding compatibility with composition vanishes in the cartesian case, as positivity holds (as with quasi-totality) merely by the universal property characterizing a map to a cartesian product by its two components.} in posets.
\end{proposition}

\subsubsection{Representability}
Maps in CD categories often behave like (partially defined) Markov kernels or multivalued maps.
One convenient property of both these classes of maps is that their data can equivalently be encoded as non-probabilistic ``functions'' into a space of distributions, essentially moving the randomness from the map to its target.
For Markov categories, this has been formalized as the inclusion of the copyable subcategory into the whole category having a right adjoint, which we call the ``distribution'' functor~\cite[Definition 3.10]{fritz2023representable}.
Much the same can be done in the non-Markov case, where this adjoint copyable map may be only partially defined.
\begin{definition}[{\cite[Definition 2.17]{fritz2025empirical}}]\label{RepresentabilityDefinition}
	Consider a quasi-Markov category $\cC$.
	A \newterm{distribution object} $PX$ for an object $X$ of $\cC$ is a representation of the functor $\cC\left(\ph,X\right)\colon \cC_{\copyable}^{\op} \to \mathsf{Set}$.
	In particular, it consists of an object $PX$ and isomorphisms natural in $A \in \cC_{\copyable}$,
	\[
	\cC\left(A,X\right) \cong \cC_\copyable\left(A,PX\right)
\]

A quasi-Markov category is \newterm{representable} if every object $X$ has a distribution object $PX$.
In such a case, the distribution objects assemble into a functor $P\colon \cC \to \cC_\copyable$ that is right adjoint to the inclusion $\cC_\copyable \hookrightarrow \cC$.
\end{definition}

We denote the unit of the adjunction at an object $X$ (the counterpart of $\id_X$) by the (copyable)
\[
	\delta_X\colon X \rightarrowtail PX
\]
and the counit (the counterpart of $\id_{PX}$) by
\[
	\samp_X\colon PX \twoheadrightarrow X
\]
with the intuition that (as in $\BorelStoch$), $\delta$ takes a point $x$ and deterministically\footnote{The unit $\delta$ has been defined to only be a-priori copyable. However it is a section of the counit, and one can show that monomorphisms in a quasi-Markov category are always total~\cite[Lemma 2.6]{fritz2025empirical}. Thus $\delta$ is indeed deterministic.} produces the distribution $\delta_x$ concentrated at $x$, and $\samp$ takes a distribution\footnote{One caveat is that it remains open as to whether sampling is total in general. Thus, this intuitive statement may only apply to some distributions on $X$. However, given a slightly stronger notion of representability called ``observational representability'', $\samp$ can be shown to be total~\cite[Lemma 2.20]{fritz2025empirical}.

While it can be shown that the partialization of an observationally representable partializable Markov category is observationally representable, we will in fact construct sampling maps in \cref{proposition:ParDistObs} that are total in the partialization of an arbitrary representable partializable Markov category.} and returns a sample from it.
In particular, $\samp_X \comp \delta_X = \id_X$.
We further denote the copyable counterpart of a morphism $f\colon X \to Y$ by $f^\sharp\colon X \to PY$ (which is merely $(Pf) \comp \delta$).

This adjunction further induces a monad $\left(P,\mu,\delta\right)$ on the domain $\cC_\copyable$ of the left adjoint.

\subsubsection{Conditionals}
Another recurring notion in categorical probability is that of the conditional of a joint distribution with respect to one (or some) of its factors.
This can be defined purely in terms of the CD structure, and in typical examples corresponds to the usual notion of conditional probability (or more precisely, the better behaved notion of ``regular conditional distribution'').
This can be defined in a ``partial'' sense as well, intuitively corresponding to conditioning being restricted to the domain of (joint) definition.
\begin{definition}\label{ConditionalsDefn}
	Consider a quasi-Markov category $\cC$.
	A \newterm{conditional} of a morphism $f \colon A \to X \otimes Y$ with respect to an output $X$ is a morphism $f_{\vert X} \colon X \otimes A \to Y$ such that
	\begin{equation}\label{ConditionalDefnEqn}
		\begin{tikzpicture}
	\begin{pgfonlayer}{nodelayer}
		\node [style=morphism] (0) at (-1.75, 0) {$\ f \ $};
		\node [style=none] (1) at (-2.25, 2) {};
		\node [style=none] (2) at (-1.25, 2) {};
		\node [style=none] (3) at (-1.75, -1.5) {};
		\node [style=none] (4) at (-2.25, 0.25) {};
		\node [style=none] (5) at (-1.25, 0.25) {};
		\node [style=none] (6) at (-2.25, 2.5) {$X$};
		\node [style=none] (7) at (-1.25, 2.5) {$Y$};
		\node [style=none] (8) at (-1.75, -2) {$A$};
		\node [style=none] (9) at (0.75, 0) {$=$};
		\node [style=morphism] (10) at (3.5, -0.5) {$\ f \ $};
		\node [style=none] (11) at (3, 1.25) {};
		\node [style=bn] (12) at (4, 0.75) {};
		\node [style=none] (13) at (4.25, -2.75) {};
		\node [style=none] (14) at (3, -0.25) {};
		\node [style=none] (15) at (4, -0.25) {};
		\node [style=none] (16) at (1.75, 5) {$X$};
		\node [style=none] (18) at (4.25, -3.25) {$A$};
		\node [style=bn] (20) at (3, 1.25) {};
		\node [style=none] (21) at (1.75, 4.5) {};
		\node [style=none] (22) at (4.25, 2.5) {};
		\node [style=morphism] (23) at (4.75, 3) {$\, f_{|X} \, $};
		\node [style=none] (24) at (4.75, 4.5) {};
		\node [style=none] (25) at (4.75, 5) {$Y$};
		\node [style=bn] (26) at (4.25, -2) {};
		\node [style=none] (27) at (5.25, 2.5) {};
	\end{pgfonlayer}
	\begin{pgfonlayer}{edgelayer}
		\draw (1.center) to (4.center);
		\draw (2.center) to (5.center);
		\draw (0) to (3.center);
		\draw (11.center) to (14.center);
		\draw (12) to (15.center);
		\draw [in=-90, out=180, looseness=0.75] (20) to (21.center);
		\draw [in=-90, out=0] (20) to (22.center);
		\draw (24.center) to (23);
		\draw [in=0, out=-90, looseness=0.50] (27.center) to (26);
		\draw (26) to (13.center);
		\draw [in=-90, out=180, looseness=0.75] (26) to (10);
	\end{pgfonlayer}
\end{tikzpicture}
	\end{equation}
\end{definition}

\begin{remark}
	This definition is a direct extension of~\cite[Definition 11.5]{fritz2019synthetic} for Markov categories (and thus, also of the more specific definition for states due to Cho and Jacobs~\cite[Definition 3.5]{chojacobs2019strings}).
	In Di Lavore and Rom\'an's~\cite[Definition 2.3]{dilavore2023evidential} a definition of conditionals for CD categories is given which allows the role of the marginal of $f$ in \cref{ConditionalDefnEqn} to be played by any morphism of suitable type.
	However, in a quasi-Markov category for instance,~\cite[Theorem 3.14]{dilavore2023evidential} shows that that definition is equivalent to \cref{ConditionalsDefn}.
\end{remark}

\subsubsection{Kolmogorov products}
The Kolmogorov extension theorem is a typical tool for working with joint distributions of infinitely many random variables.
Intuitively, it states that a joint distribution on a family of random variables is uniquely characterized by a compatible family of ``finite marginals''.
Fritz and Rischel have categorified this in terms of ``infinite tensor'' and ``Kolmogorov'' products~\cite{fritzrischel2019zeroone}, which we briefly recall.\footnote{While originally defined for semicartesian categories, the same definition can be instantiated in a CD category, as done by Moss and Perrone~\cite[Definition 8.1]{moss2022probability}.}
However, for reasons we mention at \cref{KolProdsNonFunctorialWarning} and discuss in detail at \cref{section:KolmogorovProducts}, we qualify them as ``strict'' when instanced outside a Markov (in particular semicartesian) category.

\begin{definition}[{\cite[Definition 3.1]{fritzrischel2019zeroone}}]\label{InfTensorProductDefn}
Consider a $K$-indexed family of objects ${\left(X_k\right)}_{k \in K}$ in a quasi-Markov category $\cC$.
Let $\FinSub{K}$ be the poset of finite subsets of $K$ and inclusions.
This defines a diagram $X^{\left(\phsm\right)}\colon \FinSub{K}^{\op} \to \cC$, $F \mapsto X^F\coloneqq \bigotimes_{i \in F}X_i$.
The image of a containment $F \subseteq F'$ is denoted $\pi_{F',F}\colon X^{F'} \to X^F$.

A limit cone ${\left(X^K \xrightarrow{\pi_F} X^F\right)}_{F \subseteq K \text{ finite}}$ over this diagram is a \newterm{strict infinite tensor product} if it is preserved by tensoring with an arbitrary object $Y$.
The cone maps are called the \newterm{finite marginalization maps}.
It is further a \newterm{strict Kolmogorov product} if the finite marginalizations are deterministic.
In this case it can be shown that a cone with copyable components ${\left(A \to X^F\right)}_{F \subseteq K \text{ finite}}$ induces a copyable map $A \to X^K$.
\end{definition}
When $\cC$ is a Markov category, this is precisely the same notion as in~\cite[Sections 3 and 4]{fritzrischel2019zeroone}.

\begin{remark}
In a Markov category with $K$-sized Kolmogorov products, a $K$-indexed family of morphisms $\big(X_k \xrightarrow{f_k} Y_k\big)_{k \in K}$ induces a universal morphism $f \colon X^K \to Y^K$.
This is the map induced by the cone over the diagram of finite products $\left(Y^F\right)_{F \subseteq K \text{ finite}}$, where the cone maps are the $\big(X^K \xrightarrow{\pi_F} X^F \xrightarrow{f^F} Y^F\big)_{F \subseteq K \text{ finite}}$.
However, a crucial fact that enables this construction is the fact that the maps $f_i$ are \emph{total}, so that the family of proposed cone maps is compatible with marginalization (and thus actually forms a cone).
\end{remark}

\subsubsection{Idempotents}
Idempotents have (in addition to their incarnation as domain idempotents) proven to be of independent interest for probability (such as support idempotents).
Of particular note are the following three classes of idempotents.	
\begin{definition}[{\cite[Definition 4.4.1]{fritz2023supports}}]
	An idempotent $e\colon X \to X$ in a quasi-Markov category is:
	\begin{itemize}
		\item \newterm{Balanced} if
			
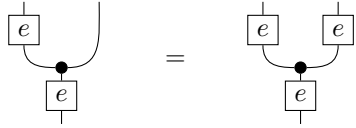
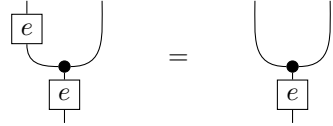
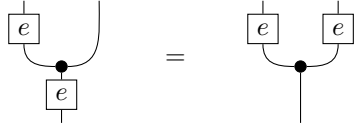
\[
				\begin{tikzpicture}
	\begin{pgfonlayer}{nodelayer}
		\node [style=bn] (0) at (0, -0.25) {};
		\node [style=morphism] (1) at (-1, 0.75) {$e$};
		\node [style=morphism] (3) at (0, -1) {$e$};
		\node [style=none] (4) at (0, -1.75) {};
		\node [style=none] (5) at (-1, 1.5) {};
		\node [style=none] (6) at (1, 1.5) {};
	\end{pgfonlayer}
	\begin{pgfonlayer}{edgelayer}
		\draw (5.center) to (1);
		\draw (3) to (4.center);
		\draw [in=180, out=-90, looseness=1.25] (1) to (0);
		\draw (0) to (3);
		\draw [in=-90, out=0, looseness=1.25] (0) to (6.center);
	\end{pgfonlayer}
\end{tikzpicture}\qquad = \qquad\begin{tikzpicture}
	\begin{pgfonlayer}{nodelayer}
		\node [style=bn] (0) at (0, -0.25) {};
		\node [style=morphism] (1) at (-1, 0.75) {$e$};
		\node [style=morphism] (2) at (1, 0.75) {$e$};
		\node [style=morphism] (3) at (0, -1) {$e$};
		\node [style=none] (4) at (0, -1.75) {};
		\node [style=none] (5) at (-1, 1.5) {};
		\node [style=none] (6) at (1, 1.5) {};
	\end{pgfonlayer}
	\begin{pgfonlayer}{edgelayer}
		\draw (5.center) to (1);
		\draw (6.center) to (2);
		\draw (3) to (4.center);
		\draw [in=180, out=-90, looseness=1.25] (1) to (0);
		\draw [in=-90, out=0, looseness=1.25] (0) to (2);
		\draw (0) to (3);
	\end{pgfonlayer}
\end{tikzpicture}
			\]
		\item \newterm{Static} if
			\[
				\begin{tikzpicture}
	\begin{pgfonlayer}{nodelayer}
		\node [style=bn] (0) at (0, -0.25) {};
		\node [style=morphism] (1) at (-1, 0.75) {$e$};
		\node [style=morphism] (3) at (0, -1) {$e$};
		\node [style=none] (4) at (0, -1.75) {};
		\node [style=none] (5) at (-1, 1.5) {};
		\node [style=none] (6) at (1, 1.5) {};
	\end{pgfonlayer}
	\begin{pgfonlayer}{edgelayer}
		\draw (5.center) to (1);
		\draw (3) to (4.center);
		\draw [in=180, out=-90, looseness=1.25] (1) to (0);
		\draw (0) to (3);
		\draw [in=-90, out=0, looseness=1.25] (0) to (6.center);
	\end{pgfonlayer}
\end{tikzpicture}\qquad = \qquad\begin{tikzpicture}
	\begin{pgfonlayer}{nodelayer}
		\node [style=bn] (0) at (0, -0.25) {};
		\node [style=morphism] (3) at (0, -1) {$e$};
		\node [style=none] (4) at (0, -1.75) {};
		\node [style=none] (5) at (-1, 1.5) {};
		\node [style=none] (6) at (1, 1.5) {};
	\end{pgfonlayer}
	\begin{pgfonlayer}{edgelayer}
		\draw (3) to (4.center);
		\draw (0) to (3);
		\draw [in=-90, out=0, looseness=1.25] (0) to (6.center);
		\draw [in=-180, out=-90, looseness=1.25] (5.center) to (0);
	\end{pgfonlayer}
\end{tikzpicture}
			\]
		\item \newterm{Strong} if
			\[
				\begin{tikzpicture}
	\begin{pgfonlayer}{nodelayer}
		\node [style=bn] (0) at (0, -0.25) {};
		\node [style=morphism] (1) at (-1, 0.75) {$e$};
		\node [style=morphism] (3) at (0, -1) {$e$};
		\node [style=none] (4) at (0, -1.75) {};
		\node [style=none] (5) at (-1, 1.5) {};
		\node [style=none] (6) at (1, 1.5) {};
	\end{pgfonlayer}
	\begin{pgfonlayer}{edgelayer}
		\draw (5.center) to (1);
		\draw (3) to (4.center);
		\draw [in=180, out=-90, looseness=1.25] (1) to (0);
		\draw (0) to (3);
		\draw [in=-90, out=0, looseness=1.25] (0) to (6.center);
	\end{pgfonlayer}
\end{tikzpicture}\qquad = \qquad\begin{tikzpicture}
	\begin{pgfonlayer}{nodelayer}
		\node [style=bn] (0) at (0, -0.25) {};
		\node [style=morphism] (1) at (-1, 0.75) {$e$};
		\node [style=morphism] (2) at (1, 0.75) {$e$};
		\node [style=none] (4) at (0, -1.75) {};
		\node [style=none] (5) at (-1, 1.5) {};
		\node [style=none] (6) at (1, 1.5) {};
	\end{pgfonlayer}
	\begin{pgfonlayer}{edgelayer}
		\draw (5.center) to (1);
		\draw (6.center) to (2);
		\draw [in=180, out=-90, looseness=1.25] (1) to (0);
		\draw [in=-90, out=0, looseness=1.25] (0) to (2);
		\draw (0) to (4.center);
	\end{pgfonlayer}
\end{tikzpicture}
			\]
	\end{itemize}
\end{definition}
Note that balance is implied by either the static or strong condition, and that copyable idempotents are all three (and conversely)~\cite[Remark 4.1.2]{fritz2023supports}.\footnote{The cited source does this for Markov categories, but the arguments only depend on the CD structure and thus still hold.}

\section{Partializable Markov categories}\label{sec:partializable}
\subsection{Partial morphisms in categorical probability}\label{subsec:PartialMorphisms}

Restriction categories are a modern framework that has shown to be useful for the study of partial maps in a variety of contexts.
Recent work suggests quasi-Markov categories as a suitable framework for partial maps in categorical probability, as they are CD categories with a notion of domain inducing a ``restriction partial order''.\footnote{To be precise, works like~\cite{lorenz2023causalmodels,gonda2024framework} have studied order-like relations on the hom-sets of a CD category, but only given quasi-totality do these truly define partial orders on the hom-sets. See \cref{restricitonPartialOrderHistory} for a discussion.}

We will in fact show that positive quasi-Markov categories are restriction categories, so that all the results of~\cite{cockettlack2002partialmaps} apply (this also supplies another way to see the poset enrichment of \cref{positiveQuasiMarkovEnriched}).\footnote{However we warn that Cockett and Lack’s results on restriction products~\cite{cockettlackRestrictionCategoriesIII2007} do not necessarily apply a priori, as they restrict to the cartesian product on the subcategory of total maps.}

\begin{proposition}[{Generalizing~\cite[Corollary 4.12]{cioffoTaxonomyCategoriesRelations2025} to a non-copyable setting}]\label{positiveQuasiMarkovRestriction}
	Every positive quasi-Markov category is a restriction category, with the restriction $\bar f$ associated to a map $f$ given by the domain $\dom f$ of \cref{domainDefinition}.
\end{proposition}
\begin{proof}
	We verify the four conditions $\left(\mathrm{R.1}\right)$\textendash $(\mathrm{R.}4)$ of \cref{restrictionCategoryDefinition}:
	\begin{enumerate}[label={(R.\arabic*):}]
		\item For the proposed restriction structure, this is quasi-totality (\cref{def:quasi-total});
		\item This is a consequence of the copy maps defining co-commutative co-monoids;
		\item This is again a consequence of the copy maps defining co-monoids;
		\item Consider a composable pair $X \xrightarrow{f} Y \xrightarrow{g} Z$.
		A consequence of quasi-totality is that the composite $\discard \comp g \comp f$ is copyable.
		Thus, by positivity (\cref{positivityDefinition}), we have
		\[
			\begin{tikzpicture}
	\begin{pgfonlayer}{nodelayer}
		\node [style=bn] (0) at (-2.75, -1.25) {};
		\node [style=none] (1) at (-2.75, -2) {};
		\node [style=morphism] (2) at (-1.75, 0) {$f$};
		\node [style=morphism] (3) at (-3.75, 0) {$f$};
		\node [style=morphism] (4) at (-1.75, 1.5) {$g$};
		\node [style=bn] (5) at (-1.75, 2.5) {};
		\node [style=none] (6) at (-3.75, 2.5) {};
		\node [style=none] (7) at (0, 0.25) {$=$};
		\node [style=bn] (8) at (2.5, 0) {};
		\node [style=morphism] (9) at (2.5, -1) {$f$};
		\node [style=morphism] (10) at (3.5, 1.5) {$g$};
		\node [style=none] (11) at (2.5, -2) {};
		\node [style=bn] (12) at (3.5, 2.5) {};
		\node [style=none] (13) at (1.5, 1.25) {};
		\node [style=none] (14) at (-1.75, 0) {};
		\node [style=none] (15) at (-3.75, 0) {};
		\node [style=none] (16) at (3.5, 1.25) {};
		\node [style=none] (17) at (1.5, 2.5) {};
	\end{pgfonlayer}
	\begin{pgfonlayer}{edgelayer}
		\draw (0) to (1.center);
		\draw (4) to (5);
		\draw (4) to (2);
		\draw [in=-90, out=90] (3) to (6.center);
		\draw (10) to (12);
		\draw [in=-90, out=165] (8) to (13.center);
		\draw (8) to (9);
		\draw (9) to (11.center);
		\draw [in=-90, out=15] (0) to (14.center);
		\draw [in=-90, out=165] (0) to (15.center);
		\draw [in=-90, out=15] (8) to (16.center);
		\draw (13.center) to (17.center);
	\end{pgfonlayer}
\end{tikzpicture}
		\]
		This is the desired equation $f \comp \overline{g \comp f} = \bar g \comp f$.
	\end{enumerate}
\end{proof}

\begin{remark}\label{posQuasiMarkovRestGeneralizesOplax}
	Cioffo et al.\ show in their recent~\cite[Corollary 4.12]{cioffoTaxonomyCategoriesRelations2025} that positive ``strict oplax cartesian categories'' are restriction categories.
	These are poset\footnote{Oplax cartesian categories are defined to be preorder enriched, but they are termed strict when the preorders are posets.} \emph{enriched} CD categories such that copyability and totality hold up to inequality~\cite[Definition 4.5]{cioffoTaxonomyCategoriesRelations2025}.
	It has been shown that strict oplax cartesian categories are all quasi-Markov~\cite[Proposition 3.6]{fritz2023lax}, and thus \cref{positiveQuasiMarkovRestriction} generalizes their result to a non-oplax-copyable setting.

	However, the categories we consider here will generally not be oplax cartesian.
	Indeed, positive oplax cartesian categories are equivalent to cartesian restriction categories (restriction categories with restriction products)~\cite[Remark 4.11]{cioffoTaxonomyCategoriesRelations2025}, and thus the partializations we construct will not be instances of these (see also \cref{partializationRestrictionProducts}).
\end{remark}

\begin{remark}\label{restPositivityRequirement}
	The positivity hypothesis is necessary\footnote{Even in the absence of positivity a quasi-Markov category is still a ``support category''~\cite[Section 2.1]{cockett2012range} (the author would like to thank Jean-Simon Lemay for suggesting this). Positivity is only needed to derive the ``R.$4$'' axiom for restriction categories, whereas the ``weak R.$4$'' axiom for support categories holds for general CD categories.} here, as even for a common quasi-Markov category like $\Rel$ the restriction partial order of \cref{def:dom_ext} need not be compatible with pre-composition\footnote{Consider relations $S,T\colon \{u,v\} \to \{x,y\}$, where $S=\{(u,x)\}$ while $T = \{(u,x),(v,y)\}$. Then $S$ is the restriction of $T$ to its domain $\{u\}$.
On the other hand, pre-composing with $R\colon \{a\} \to \{u,v\}$ one gets two distinct total relations, but $SR$ has range $x$ while $TR$ has full range.} (and thus cannot be induced by a restriction structure, as these are always poset enrichments\footnote{This also shows that positivity is required for \cref{positiveQuasiMarkovEnriched}.}).
	This suggests some connection between positivity and restriction structures, but a precise characterization of when the restriction partial order is an enrichment remains open (of course, one can have enrichment without positivity, as holds trivially for non-positive Markov categories).
\end{remark}
\subsection{The Partialization of a Markov category}\label{subsec:partialization}
We now show how to add partial morphisms to a Markov category $\cat{C}$ to obtain a new quasi-Markov category.
This main idea is as follows: 
We form a new category $\Par{\cC}$ where the objects are the ones of $\cat{C}$, and the morphisms $X \to Y$ are equivalence classes of spans of the form
\[
\begin{tikzcd}
	X & D \ar[hook']{l}[swap]{i} \ar{r}{f} & Y
\end{tikzcd}
\]
where the map $i$ is a particular kind of monomorphism. 
We occasionally denote such a span by a triple $\left(D,i,f\right)$ or merely a pair $\left(i,f\right)$ whenever the objects are clear from context. This is a typical categorical construction of a category of partial maps in terms of spans.

The intuition is that $f$ is only defined on $D$, which is included within $X$ via $i$. 
The domain of $f$ is thus in general not the whole of $X$.
However, as our morphisms are intuitively partial \emph{stochastic} maps, we insist upon the condition that the domain inclusions $D \xhookrightarrow{i} X$ be \emph{deterministic} monomorphisms.\footnote{So all the randomness is in the ``action'' map $X \xrightarrow{f} Y$.}
In order for the construction to work, we require a couple of conditions on the original Markov category.
\begin{definition}\label{def:partializable}
	A Markov category $\cC$ is called \newterm{partializable} if:
	\begin{enumerate}
		\item\label{it:isos_det} It is positive;\footnote{In such a case, all isomorphisms are deterministic~\cite[Remark 11.28]{fritz2019synthetic}.}
		\item\label{it:pullback_cond} Pullbacks of deterministic monomorphisms exist and are themselves deterministic;
		\item If $i \colon D \to X$ is a deterministic monomorphism then every $i \otimes \id_A$ is a monomorphism.
	\end{enumerate}
\end{definition}

Some useful facts about partializable Markov categories are given in \Cref{sec:par_facts}.

The three conditions ensure that in a partializable Markov category, the deterministic monomorphisms form what Cockett and Lack call in~\cite[Section~3.1]{cockettlack2002partialmaps} a \newterm{stable system of monics} (as in \cref{stableSysDefinition}).
This is in essence, the condition for the construction of a category $\Par{\cC}$ with objects those of $\cC$ and morphisms equivalence classes of spans $X \xhookleftarrow{i} A \xrightarrow{f} Y$, with $i$ a deterministic monomorphism (as in \cref{spanCatDefinition}).

\begin{definition}\label{def:partial_C}
	Let $\cat{C}$ be a partializable Markov category. 
	The category $\Par{\cC}$, termed \newterm{partialization} of $\cC$, is constructed as follows:
	\begin{itemize}
		\item The objects are the objects of the original category $\cat{C}$;
		
		\item The morphisms $X\to Y$ are equivalence classes of spans of the form
		\[
		\begin{tikzcd}
			X & D \ar[hook']{l}[swap]{i} \ar{r}{f} & Y
		\end{tikzcd}
		\]
		where $i$ is a \emph{deterministic} monomorphism and where two spans $\left(D,i,f\right)$ and $\left(D',i',f'\right)$ are considered equivalent if and only if there is an isomorphism $D\to D'$ making the following diagram commute:
		\[
		\begin{tikzcd}[sep=small]
			& D \ar[hook]{dl}[swap]{i} \ar{dr}{f} \ar[dd, leftrightarrow,"\simeq" {anchor=south, rotate=270}] \\
			X && Y \\
			& D' \ar[hook']{ul}{i'} \ar{ur}[swap]{f'}
		\end{tikzcd}
		\]
		
		\item The composition of two morphisms represented by spans $X \xhookleftarrow{i} A \xrightarrow{f} Y$ and $Y \xhookleftarrow{j} B \xrightarrow{g} Z$ is obtained by forming the pullback
		\[
		\begin{tikzcd}[sep=small]
			&                                           & C \arrow[ld, "u"', hook] \arrow[rd, "v"] \arrow["\lrcorner"{anchor=center, pos=0.125, rotate=-45}, draw=none, dd]&                                         &   \\
			& A \arrow[ld, "i"', hook] \arrow[rd, "f"'] &                                          & B \arrow[ld, "j", hook] \arrow[rd, "g"] &   \\
			X &                                           & Y                                        &                                         & Z
		\end{tikzcd}
		\]
		and is represented by the span $X \xhookleftarrow{i\comp u} C \xrightarrow{g \comp v} Z$.
	\end{itemize}
\end{definition}

There is a canonical inclusion of $\cC$ into $\Par{\cC}$ sending $X \xrightarrow{f} Y$ to the (equivalence class) of the span $X = X \xrightarrow{f}Y$.

\begin{example}[{standard Borel spaces and Markov kernels}]\label{example:BorelStochPartializable}
	Our main example will be the Markov category $\BorelStoch$~\cite[Section 4]{fritz2019synthetic} of standard Borel spaces (recall \cref{CDCatsExamples}~\ref{BorelStochBackground}).
	Fritz has shown in~\cite[Example~11.25]{fritz2019synthetic} that $\BorelStoch$ is positive.
	Since the deterministic morphisms in $\BorelStoch$ are exactly the measurable maps, it follows that the deterministic monomorphisms are exactly the measurable injections. 
	It is a non-trivial fact that every measurable injection also preserves measurable sets in the forward direction, and is therefore a Borel isomorphism onto its image.\footnote{See Kechris's~\cite[Corollary~(15.2)]{kechris} for a source.} 
	In other words, the deterministic subobjects of a standard Borel space $X$ are exactly its measurable subsets. 
	Hence the morphisms of type $X \to Y$ in $\Par{\BorelStoch}$ are pairs $\left(S,f\right)$ where $S \subseteq X$ is measurable and $f \colon S \to Y$ is a measurable map.
	
	As an immediate consequence, we have that for such a deterministic monomorphism $i$ and object $A$, the tensor $i \otimes \id_A$ is again a deterministic monomorphism (it is again a measurable subset inclusion).

	Finally, concerning the pullback condition, consider a deterministic subobject $T \in \Sigma_Y$ and a Markov kernel $f \colon X \to Y$. 
	Our goal is to construct a pullback diagram in $\BorelStoch$ of the form
	\[\begin{tikzcd}
		S \ar[hook]{d} \ar{r}{f\vert_S} \arrow[rd, "\lrcorner"{anchor=center, pos=0.125}, draw=none]& T \ar[hook]{d} \\
		X \ar{r}{f} & Y
	\end{tikzcd}\]
	with $S \in \Sigma_X$.

	Note that a necessary condition for such a square to even commute is that $f\left(T\, \vert\, x\right) = 1$ for all $x \in S$.
	Indeed for $x \in S$, $f\left(T\, \vert\, x\right)$ is the probability that the image of $s$ under the lower-left path in the diagram is in $T$.
	If the diagram commutes, this must also be the probability under the upper-right path, $f\vert_S\left(T\, \vert\, x\right)$.
	However as this is a Markov kernel to $T$, the probability is $1$.

	Motivated by this, we define\footnote{Where $T^c \coloneqq Y \setminus T$ is the complement of $T$.}
	\[
	S \coloneq \left\{ x \in X \sot f\left(T\, \vert\, x\right) = 1 \right\} = \left\{ x \in X \sot f\left(T^c\, \vert\, x\right) = 0 \right\}
	\]
	which is a measurable subset of $X$ by the measurability of $f\left(T\,\vert \,\ph\right)$.

	We noted that in general, a Markov kernel $h \colon W \to X$ factors through a deterministic subobject $S \subseteq X$ if and only if $h\left(S\,\vert\,w\right) = 1$ for all $w \in W$. 
	This lets us establish the pullback property as follows. 
	Given $h \colon W \to X$, we assume that $f \comp h$ factors through $T$, which means that
	\[
	\left(f \comp h\right)\left(T\,\vert\,w\right) = \int_{x \in X} f\left(T\,\vert\,x\right) \, h\left(dx\,\vert\,w\right) = 1 \qquad \forall w \in W,
	\]
	or (by normalization of probability) equivalently that
	\[
	\int_{x \in X} f\left(T^c\, \vert\, x\right) \, h\left(dx\, \vert\, w\right) = 0 \qquad \forall w \in W.
	\]
	An integral of a nonnegative function vanishes if and only if the integrand vanishes almost everywhere. 
	Taking complements again, we can therefore conclude 
	\[
	f\left(T\,\vert\, \ph \right) \ase{h\left(\ph\,\vert\, w\right)} 1
	\]
	By the definition of $S$, this is equivalent to $h\left(S\,\vert\, w\right) = 1$, which means that $h$ indeed factors through $S$ (necessarily uniquely). 
	Thus, the pullback condition holds as well.
\end{example}

\begin{remark}\label{FinStochPartializable}
	The same argument of \cref{example:BorelStochPartializable} shows that the full subcategory $\FinStoch$~\cite[Example 2.5]{fritz2019synthetic} (recall \cref{CDCatsExamples}~\ref{FinStochBackground}) of finite (discrete) spaces is also partializable.
	This can also be seen as a consequence of the fact that all distributions involved are discrete, which is developed more generally in \cref{discDistPartializable}.
\end{remark}
\begin{example}[Sets and multivalued functions]\label{SetMultiPartializable}
The category $\cat{SetMulti}$ of sets and multivalued maps (\cref{CDCatsExamples}~\ref{SetMultiBackground}) is partializable.
To see this, note first that it has been shown to be positive (for instance, as a consequence of having conditionals).
Furthermore, it is the Kleisli category of the \emph{nonempty} power set monad on $\cat{Set}$, and additionally representable.
In particular, the deterministic monomorphisms are precisely the injective functions.
Thus, the deterministic monomorphisms are also closed under tensor.

It only remains to check the pullback condition.
For this, one need only observe that a multivalued function (total relation) $X \xrightarrow{f} Y$ factors across a deterministic subobject $T \subseteq Y$ if and only if each image set $f\left(x\right) \subseteq T$ (for all $x \in X$).

Thus, if we simply define $S \coloneqq \left\{ x \in X \sot f\left(x\right) \subseteq T\right\} \subseteq X$, then the diagram
	\[\begin{tikzcd}
		S \ar[hook]{d} \ar{r}{f\vert_S} \arrow[rd, "\lrcorner"{anchor=center, pos=0.125}, draw=none]& T \ar[hook]{d} \\
		X \ar{r}{f} & Y
	\end{tikzcd}\]
	is a pullback.
	Indeed, if any multivalued map $W \xrightarrow{h} X$ is such that $f\comp h$ factors through $T$ (necessarily uniquely), then for each $w \in W$, $\left(f \comp h\right)\left(w\right) = \bigcup_{x \in h\left(w\right)} f\left(x\right) \subseteq T$, and thus $h\left(w\right) \subseteq S$.
	Thus, $h$ would have to factor through $S$ (necessarily uniquely).
	
	Thus $\cat{SetMulti}$ satisfies all three conditions for partializability.
\end{example}

\begin{example}[{Distributions valued in an entire zerosumfree semiring}]\label{discDistPartializable}
	For a semiring $R$ (we assume all semirings to be commutative), Coumans and Jacobs in~\cite[Section 5.1]{coumans2013scalars} have constructed a monad of ``$R$-valued finitely supported distributions'' on the category $\mathsf{Set}$ of sets and functions.
	Recall (\cref{CDCatsExamples}~\ref{Kl(DR)Background}) that as observed in~\cite[Example 3.3]{fritz2023representable}, its Kleisli category $\Kl\left(D_R\right)$ is a Markov category (recall \cref{CDCatsExamples}~\ref{Kl(DR)Background}) whose objects can be thought of as the category of discrete measurable spaces and whose morphisms are kernels
	\[
		f\colon X \to Y,\qquad x \mapsto \sum_{y \in Y} f\left(y\, \vert\, x\right)\delta_y
	\]
	such that $\supp{f\left(\ph\,\vert\,x\right)} \coloneqq \left\{ y \in Y : f\left(y\,\vert\,x\right) \ne 0\right\}$ is finite.

	Fritz et al.\ show later in~\cite[Proposition 3.6]{fritz2023representable} that when the semiring is entire, the category $\Kl\left(D_R\right)$ is representable, and the deterministic morphisms are functions between sets.
	In particular, the deterministic \emph{monomorphisms} are (like with $\BorelStoch$ in \cref{example:BorelStochPartializable}) the inclusions of subsets.

	In~\cite[Proposition 2.12]{fritz2022dilations}, it is further shown that for an entire (commutative) semiring $R$, the category $\Kl\left(D_R\right)$ is positive if and only if $R$ is zerosumfree.
	In fact, we show that for an entire zerosumfree semiring $R$, the category $\Kl\left(D_R\right)$ is a partializable Markov category.

	The fact that deterministic monomorphisms are measurable injections implies that they are closed under tensoring with identities.
	Thus it only remains to check the pullback condition~\ref{it:pullback_cond} of \cref{def:partializable}.
	For this, consider again the pullback of a subobject $T \subseteq Y$ along a $f\colon X \to Y$.
	We claim that the subset
	\[
		S \coloneqq \left\{x \in X \sot \supp{f\left(\ph\,\vert\, x\right)}\subseteq T\right\}
	\]
	defines the pullback
	\[\begin{tikzcd}
		S \ar[hook]{d} \ar{r}{f\vert_S} \arrow[rd, "\lrcorner"{anchor=center, pos=0.125}, draw=none]& T \ar[hook]{d} \\
		X \ar{r}{f} & Y
	\end{tikzcd}\]

	For this, note that for a kernel $h \colon W \to X$ to factor across a subset $A \subseteq X$, a necessary and sufficient condition is that $\supp{h\left(\ph\,\vert\, w\right)} \subseteq A$ for all $w \in W$.
	Thus, if for such an $h$ the composite $f\comp h$ factors across $T$, we must have $\supp{\left(f\comp h\right)\left(\ph\,\vert\, w\right)} \subseteq T$ for all $w \in W$.

	Recall that
	\[
		\left(f \comp h\right) \left(y\, \vert \, w\right) = \sum_{x \in X} f\left(y\,\vert\,x\right) h\left(x\,\vert\,w\right)
	\]
	Thus if $\left(f \comp h\right) \left(y\, \vert \, w\right) = 0$, zerosumfreeness implies that
	\[
		f\left(y\,\vert\,x\right) h\left(x\,\vert\,w\right) = 0
	\]
	for all $x \in X$.
	Hence as $R$ is entire, one of the two factors must be zero.
	In particular when $x \in \supp{h\left(\ph\,\vert\, w\right)}$, we must have $f\left(y\,\vert\,x\right) = 0$.

	Consider an $x \in \supp{h\left(\ph\,\vert\, w\right)}$.
	We have just seen that $\supp{f\left(\ph\, \vert\,x\right)} \subseteq \supp{\left(f\comp h\right)\left(\ph\, \vert\,w\right)}$, which is in turn a subset of $T$.
	Thus $x \in S$ by construction, and as $x \in \supp{h\left(\ph\,\vert\, w\right)}$ was arbitrary, we conclude that $h$ factors through $S$ (necessarily uniquely).
	Thus all three conditions have been verified and we conclude that $\Kl\left(D_R\right)$ is partializable.
\end{example}

\begin{remark}
	A consequence of zerosumfreeness in \cref{discDistPartializable} is that for a kernel $f\colon X \to Y$ and a subset $A \subseteq Y$, the condition $\supp{f\left(\ph\,\vert\, x\right)} \subseteq A$ is equivalent to $f\left(A^c \,\vert\, x\right) = 0$ for all $x \in X$, as in \cref{example:BorelStochPartializable}.
	Indeed, the entire argument could be recast in terms of the kernels assigning the complement of a subobject probability $0$.

	However, while (as with $\BorelStoch$) it might be tempting to restate the condition in terms of the subobject having probability $1$, this is not equivalent in general.
	Indeed, for semirings such as the Booleans (sub-example~\ref{setFinMultiPartializable} of \cref{distAndMultiPartializable}) the condition $f\left(A^c \,\vert\, x\right) = 0$ is the assertion that $A$ contains \emph{all} possible images of $x$, whereas $f\left(A \,\vert\, x\right) = 1$ merely requires $A$ to contain \emph{one} of them.
	When $R$ is for instance, cancellative, such a problem is not possible and the two conditions are in fact equivalent.
\end{remark}

\begin{remark}[Subcategories of finite spaces]\label{FinKl(DR)Partializable}
	For a (commutative) semiring $R$, the full subcategory $\FinKl\left(D_R\right)$ spanned by the finite sets is closed under monoidal tensor (as these represented by cartesian products of finite sets), and this subcategory also contains the CD structure data (monoidal coherence maps and comonoid structures).
	Consequently, these are Markov categories as well.\footnote{However they are no longer Kleisli categories (nor representable) in general, as the space of (finite) distributions need not be finite.}

	Furthermore, the argument of \cref{discDistPartializable} shows that $\FinKl\left(D_R\right)$ is partializable as well.
	One need only note that they are positive as sub-Markov categories of positive ones.
	The deterministic monomorphisms remain the finite measurable injections and are hence closed under tensor.
	And finally, the pullback ``$S$'' in \cref{discDistPartializable} is constructed as a subset of $X$, so is finite if $X$ is.
\end{remark}
\begin{example}\label{distAndMultiPartializable}
	Sub-examples of \cref{discDistPartializable,FinKl(DR)Partializable} for two familiar semirings are of particular common interest:
	\begin{enumerate}
		\item\label{distPartializable} The category $\cat{Dist}$ of discrete probability distributions (\cref{CDCatsExamples}~\ref{DistBackground}), which can be seen as the subcategory of the category of measurable spaces and Markov kernels consisting of the discrete spaces and discrete kernels (ones that assign to each point of the source a mixture of point distributions on the target).
			This is the Kleisli category $\Kl\left(D_R\right)$ with $R$ the nonnegative reals $\R_{\ge 0}$, and is thus partializable.
		\item The full subcategory of $\cat{\Dist}$ spanned by the finite sets, $\FinKl\left(\Rnneg\right)$, is precisely $\FinStoch$.
			This supplies another way of seeing that $\FinStoch$ is partializable.
		\item\label{setFinMultiPartializable} The category $\cat{SetFinMulti}$ (\cref{CDCatsExamples}~\ref{SetFinMultiBackground}) is the wide subcategory of $\SetMulti$ comprising the finitely multivalued maps.\footnote{In other words those such that every element of the domain is related to some element of the codomain, and only finitely many.}
			This is the Kleisli category $\Kl\left(D_R\right)$ with $R$ the Booleans $\left(\left\{\bot,\top\right\},\lor,\land\right)$,\footnote{The elements being false and true, with addition and multiplication being disjunction and conjunction respectively.} and is thus partializable as well.
		\item\label{FinSetMultiPartializable} The full subcategory of $\cat{SetFinMulti}$ spanned by the finite sets is precisely $\FinSetMulti$, the category of finite sets and multivalued maps and is thus partializable as well.
	\end{enumerate}
\end{example}

\begin{warning}\label{parCatsExmpsDiff}
	We have seen in \cref{example:BorelStochPartializable,distAndMultiPartializable,SetMultiPartializable} that $\BorelStoch,\cat{Dist}$, and $\cat{SetMulti}$ are partializable Markov categories.
	They thus have partializations $\Par{\BorelStoch}$, $\Par{\cat{Dist}}$, and $\Par{\cat{SetMulti}}$ as in \cref{def:partial_C}, which can intuitively be thought of as categories with morphisms partially defined Markov kernels, discrete partially defined kernels, and partially defined multivalued maps.

	A partially defined multivalued map is nothing but a relation, since every relation defines a multivalued map on its domain.
	Thus, one may suspect that the category $\Par{\cat{SetMulti}}$ is merely the category of relations $\Rel$.
	Similarly, notions of categories of ``partially defined Markov kernels'' have appeared in earlier work on CD categories, such as Di Lavore and Rom\'an's categories $\BorelStoch_{\le 1}$\footnote{Appearing under that name at~\cite[Definition 2.15]{dilavore2023evidential}, although it is the Kleisli category of the ``subprobability measure'' monad $\mathbf G_{\le 1}$ which appeared as $\Pi$ in Panangaden's~\cite[Section 4]{panangaden1999CategoryMarkovKernels} (or more accurately, is the restriction to standard Borel spaces).} and $\Kl{\left(\mathbf D_{\le 1}\right)}$ (the Kleisli category of the finite \emph{sub-distribution} monad~\cite[Definition 2.11]{dilavore2023evidential}).\footnote{$\BorelStoch_{\le 1}$ has also appeared in Lorenz and Tull's~\cite{lorenz2023causalmodels} as $\cat{BorelProb}$. In the later work~\cite{dilavore2024partial}, $\Kl{\left(\mathbf D_{\le 1}\right)}$ is referred to as $\cat{subStoch}$.}

	Both $\BorelStoch$ and $\Dist$ are Kleisli categories for some kind of distribution monad $D$, and it can even be shown that the morphisms $X \rightharpoonup Y$ of $\Par{\BorelStoch}$ and $\Par{\cat{Dist}}$ are precisely the Kleisli morphisms $X \to D\left(Y\sqcup \ast\right)$ that factor across $DY \sqcup \ast$ (with the domain being the subobject that maps into $DY$).\footnote{In other words, the image of each $x\in X$ is either a probability measure or zero/undefined.}
	In light of these observations, the reader may well suspect that $\Par{\BorelStoch}$, $\Par{\cat{Dist}}$, and $\Par{\cat{SetMulti}}$ are at the very least subcategories of $\BorelStoch_{\le 1}$, $\Kl{\left(\mathbf D_{\le 1}\right)}$, and $\Rel$ under the aforementioned intuitive identifications of morphisms.
	However, this is not the case as the composition operations differ.

	For instance, consider a fair coin flip, composed with the partially defined (copyable!) kernel that assigns a value $a$ to heads.
	This can be seen as defining a composable sequence $\ast \xrightarrow{p} \left\{H,T\right\} \xrightarrow{f} \left\{a\right\}$.
	As morphisms of $\Par{\BorelStoch}$ or $\Par{\cat{Dist}}$, these are represented by $p$ assigning probability $\frac{1}{2}$ to both heads and tails, and $f$ having domain of definition $\left\{H\right\}$.
	Following \cref{example:BorelStochPartializable,discDistPartializable}, the domain of definition of the composite is the subspace of $\ast$ comprising the events whose corresponding distribution over $\left\{H,T\right\}$ assigns the domain $\left\{H\right\}$ of $f$ probability $1$, and is thus in fact empty.
	The corresponding substochastic kernel assigns probability $0$ to $a$.

	On the other hand, as subprobability measures $f$ would instead assign probability $1$ to heads and $0$ to tails.
	The composite would then be given by the Chapman-Kolmogorov equation, and in particular the composite would assign probability to $a$ the probability
	\[
		\left(f\comp p\right)\left(a\right) = \sum_{x \in \left\{H,T\right\}} f\left(a\,\vert\,x\right) p\left(x\right) = f\left(a\,\vert\,H\right) p\left(H\right) + f\left(a\,\vert\,T\right) p\left(T\right) = 1 \cdot \frac{1}{2} + 0 \cdot \frac{1}{2} = \frac{1}{2}
	\]

	The same applies to $\Par{\cat{SetMulti}}$ and $\Rel$, where instead of $p$ assigning probabilities $\frac{1}{2}$, we interpret $p$ as having both heads and tails as possible outcomes.
	Then the composite $f\comp p$ in $\Par{\cat{SetMulti}}$ would be defined only on the events such that $f$ is defined on \emph{all} possible outcomes of $p$, hence nowhere.
	On the other hand the composite in $\Rel$ would be defined on the events such that $f$ is defined on \emph{some} possible outcome of $p$, hence on all of $\ast$.
\end{warning}

\begin{remark}\label{parCompDomainsSmall}
	The common pattern in the comparisons of \cref{parCatsExmpsDiff} is that the composition of two morphisms in $\Par{\cC}$ is defined on a ``smaller domain'' than in they would be in the categories it's being compared to.
	Indeed, as $\mathbb R_{\ge 0}$ is entire and zerosumfree, the composition of two morphisms in $\Par{\BorelStoch}$ (and $\Par{\cat{Dist}}$) will always have domain contained in that of the corresponding subprobability measure (the same holds for $\Par{\cat{SetMulti}}$ and $\Rel$).
	Considering the intuition that quasi-totality corresponds to determinism of domains, one may see this less intuitive composition as being required to ensure that the composites have deterministic domains.

	Indeed, Di Lavore and Rom\'an have shown that the quasi-total morphisms of $\BorelStoch_{\le 1}$ and $\Kl{\left(\mathbf D_{\le 1}\right)}$ are the ones that assign each point of the source subprobability measures that are either probability measures or identically zero~\cite[Remark 2.18]{dilavore2023evidential}.
	These are precisely the morphisms of those categories that correspond to ones of $\Par{\BorelStoch}$ and $\Par{\cat{Dist}}$.
	The counterexample of \cref{parCatsExmpsDiff} thus also shows that composites of quasi-total morphisms in a (positive) CD category need not be quasi-total.
	In particular, we cannot construct quasi-Markov categories by simply taking the quasi-total morphisms of an arbitrary CD category.
	However, we show in \cref{partializationsQuasiMarkov} that the partialization of a partializable Markov category is always quasi-Markov, and thus quasi-total morphisms are closed under composition in this class of category.
\end{remark}

\begin{remark}
	The composition law in $\Par{\SetMulti}$ can be seen as a more ``secure'' version of the usual composition in $\Rel$.
	Indeed, one could think of a relation as a nondeterministic program or process, which may fail (to terminate) on some inputs.
	Then the typical composition in $\Rel$ would track merely whether its possible to begin with a given input and reach a particular output.
	However, there may be settings where failing to terminate on some possible input is catastrophic, and it would be preferable to only consider input output pairs such as all possible intermediate states lead to a defined output. That is, only allow inputs that ensure the composite process does not execute the first step and then catastrophically fail.
	This is precisely what the composition in $\Par{\SetMulti}$ achieves.
\end{remark}

\begin{remark}\label{spanPartialOrder}
	As remarked in \cref{backgroundSec}, the hom-sets of $\Par{\cC}$ are partially ordered with the partial order being witnessed by morphisms between the apices of the spans.
Specifically, for a parallel pair $\left(A,i,f\right)$ and $\left(B,j,g\right)$ represented by spans $X \xleftarrow{i} A \xrightarrow{f} Y$ and $X \xleftarrow{j} B \xrightarrow{g} Y$, we have $(A,i,f) \ge \left(B,j,g\right)$ if and only if there is a morphism $B \to A$ filling in the following commutative diagram:\footnote{This morphisms is necessarily a deterministic monomorphism, as in \cref{DetMonoLeftCancellation}.}
\begin{equation}\label{eq:rest_ord}
	\begin{tikzcd}[sep=small]
		& A \ar[tail]{dl}[swap]{i} \ar{dr}{f} \\
		X && Y \\
		& B \ar[tail]{ul}{j} \ar{ur}[swap]{g} \ar{uu}
	\end{tikzcd}
\end{equation}
This partial order features in all the works mentioned in \cref{stableClassAliases}, where it is shown to be equivalent to the order of \cref{def:dom_ext} in the cartesian case.
We show that this equivalence holds for $\Par{\cC}$ as well in \cref{prop:ParExtParOrd}.
\end{remark}

\subsection{Partializations are quasi-Markov}\label{partializationsCDStructure}

Consider a partializable Markov category $\cC$ and the span category $\Par{\cC}$.
Cockett and Lack show that for any stable system of monics $\distMons$ in a category $\cC$, the category of spans $\Par{\cC,\distMons}$ has a restriction structure given by $\overline{\left(D,i,f\right)} \coloneqq \left(D,i,i\right)$~\cite[Proposition 3.1]{cockettlack2002partialmaps}.
While this is defined independently of any monoidal structure, in the cartesian case it can be shown to be equivalent to the restriction structure associated to a $p$-category~\cite[Examples 2.1.3 (6)]{cockettlack2002partialmaps}, which is in particular a special case of that described in \cref{positiveQuasiMarkovRestriction}.

In light of \cref{positiveQuasiMarkovRestriction}, the domain construction defining a restriction structure may appear to be a rather strong condition.
Indeed, even the first condition ``R.$1$'' requires the category to be quasi-Markov, which already excludes plausible candidates for ``partial stochastic map'' categories such as $\BorelStoch_{\le 1}$ and $\Kl{\left(\mathbf D_{\le 1}\right)}$ (discussed in detail in \cref{parCompDomainsSmall}).\footnote{The category $\Rel$ is in fact quasi-Markov, but fails to be a restriction category as the fourth condition ``R.$4$'' fails. In fact this domain structure doesn't even define a poset-enrichment under the restriction partial order of \cref{def:dom_ext}.}

We show that $\Par{\cC}$ becomes a (symmetric) monoidal category upon defining the monoidal structure in terms of the one of $\cC$ in the obvious way such that both the domain inclusions and the restricted morphisms are composed in parallel.
The third condition in \cref{def:partializable} ensures that the resulting span has again a deterministic monomorphism as its domain inclusion. 
We thereby obtain a \emph{monoidal restriction category} in the sense of Heunen and Lemay's~\cite[Definition~4.2]{heunenpacaudlemay2020tensor} (see also \cref{tensorRestrictionDiscussion}).\footnote{The author would like to thank Tom\'a\v{s} Gonda for this observation.}\footnote{However, $\Par{\cC}$ is not a \emph{tensor} restriction category~\cite[Section 5]{heunenpacaudlemay2020tensor}. In particular, the $\Par{\ph}$ construction is not the same as the $\mathcal S\left[\ph\right]$ construction of~\cite[Definition 3.1]{heunenpacaudlemay2020tensor}. Indeed, in~\cite[Examples 4.3, 5.21]{heunenpacaudlemay2020tensor} it is noted that the category of sets and partial functions is a monoidal restriction category that is \emph{not} a tensor restriction category.}

This monoidal structure in fact makes $\Par{\cC}$ a positive quasi-Markov category, and the induced restriction structure as in \cref{positiveQuasiMarkovRestriction} is the same as that as the one defined in terms of spans (and hence the induced poset enrichment is the same as well).
This is shown in \cref{partializationRestrictionStructuresAgree}.

\begin{proposition}\label{proposition:SpanMonStr}
	We define a (bi)functor $\Par{\cC} \times \Par{\cC}\xrightarrow{\otimes}\Par{\cC}$ acting componentwise as the monoidal product of $\cC$.
	That is, we define
	\[
	\left(X \xleftarrow{i}A\xrightarrow{f}Y\right)\otimes \left(X'\xleftarrow{j}B\xrightarrow{g}Y'\right) \coloneqq \left(X \otimes X' \xleftarrow{i \otimes j} A \otimes B \xrightarrow{f \otimes g}Y \otimes Y'\right)
	\]
	This tensor product is well-defined, functorial, and is part of a symmetric monoidal structure extending that of $\cat{C}$.
\end{proposition}
\begin{proof}
	First note that the claimed product is well-defined, i.e.\ independent of the choice of representatives of the spans.
	This is a consequence of the invertible morphisms being closed under tensoring, thus equivalent spans tensor to equivalent ones.

	Moving to functoriality, the preservation of identities is immediate.
	For the compatibility with composition consider morphisms ${\big(X\xleftarrow{i}A\xrightarrow{f}Y\big)}$, ${\big(Y\xleftarrow{j}B\xrightarrow{g}Z\big)}$, ${\big(X'\xleftarrow{i'}A'\xrightarrow{f'}Y'\big)}$, and a fourth ${\big(Y'\xleftarrow{j'}B'\xrightarrow{g'}Z'\big)}$. 
	The composites $\left(B,j,g\right) \comp \left(A,i,f\right)$ and $\left(B',j',g'\right)\comp \left(A',i',f'\right)$ can be computed by forming the pullbacks
	\[
	\begin{tikzcd}
		C \arrow[r, "v"] \arrow[d, "u"', tail]\arrow[rd, "\lrcorner"{anchor=center, pos=0.125}, draw=none] & B \arrow[d, "j", tail] & C'\arrow[rd, "\lrcorner"{anchor=center, pos=0.125}, draw=none] \arrow[d, "u'"', tail] \arrow[r, "v'"] & B' \arrow[d, "j'", tail] \\
		A \arrow[r, "f"']                      & Y                      & A' \arrow[r, "f'"']                       & Y'                      
	\end{tikzcd}
	\]
	after which the composites are represented by $\left(C,i\comp u,g\comp v\right)$ and $\left(C',i'\comp u',g'\comp v'\right)$ respectively.
	
	The computation of the composite $\left(B\otimes B',j\otimes j',g \otimes g'\right) \comp \left(A\otimes A',i\otimes i',f \otimes f'\right)$ involves first forming the pullback of $f \otimes f'$ along $j \otimes j'$.	
	By \cref{proposition:StablePullbackTensorStability} we may as well use the pullback square
	\[\begin{tikzcd}
		C \otimes C' \arrow[r, "v \otimes v'"] \arrow[d, "u \otimes u'"', tail] \arrow[rd, "\lrcorner"{anchor=center, pos=0.125}, draw=none] & B \otimes B' \arrow[d, "j \otimes j'", tail] \\
		A \otimes A' \arrow[r, "f \otimes f'"']                                 & Y \otimes Y'                                
	\end{tikzcd}\]
	The composite is then computed to be $\bigl(C\otimes C',\left(i\otimes i'\right) \comp \left(u \otimes u'\right),\left(g\otimes g'\right)\comp \left(v\otimes v'\right)\bigr)$, and functoriality follows.
	Furthermore, the unit of $\cC$ is also the unit for this tensor product on $\Par{\cC}$.
	
	The associators, unitors, and braiding of $\cC$ are additionally natural as maps of $\Par{\cC}$.
	For instance consider spans $\bigl(X_1 \xleftarrow{i_1} A_1 \xrightarrow{f_1} Y_1 \bigr)$, $\bigl(X_2 \xleftarrow{i_2} A_2 \xrightarrow{f_2} Y_2 \bigr)$, and $\bigl(X_3 \xleftarrow{i_3} A_3 \xrightarrow{f_3} Y_3 \bigr)$.
	Naturality of the associator $\alpha$ with respect to these three morphisms of $\Par{\cC,\distMons}$ would mean
	\begin{equation}
		\alpha_{Y_1,Y_2,Y_3}\comp \Bigl(\bigl(\left(i_1,f_1\right)\otimes\left(i_2,f_2\right)\bigr)\otimes \left(i_3,f_3\right) \Bigr) = \Bigl(\left(i_1,f_1\right)\otimes\bigl(\left(i_2,f_2\right)\otimes \left(i_3,f_3\right)\bigr) \Bigr)\comp \alpha_{X_1,X_2,X_3}.
	\end{equation}
	The former is $\bigl(\left(A_1\otimes A_2\right)\otimes A_3,\left(i_1\otimes i_2\right)\otimes i_3, \alpha_{Y_1,Y_2,Y_3}\comp \left(f_1\otimes f_2\right)\otimes f_3\bigr)$, while the latter can be computed using the pullback
	\[
	\begin{tikzcd}
		\left(A_1 \otimes A_2\right) \otimes A_3 \arrow[d, "\left(i_1 \otimes i_2\right) \otimes i_3"'] \arrow[rr, "{\alpha_{A_1,A_2,A_3}}"] \arrow[rrd,"\lrcorner"{anchor=center, pos=0.125}, draw=none] & & A_1\otimes \left(A_2 \otimes A_3\right) \arrow[d, "i_1 \otimes \left(i_2\otimes i_3\right)"] \\
		\left(X_1\otimes X_2\right)\otimes X_3 \arrow[rr, "{\alpha_{X_1,X_2,X_3}}"]                                              & & X_1 \otimes \left(X_2\otimes X_3\right)                                          
	\end{tikzcd}
	\]
	The naturality of the associator with respect to the morphisms $f_1$, $f_2$, $f_3$ of $\cC$ then implies that this choice of pullbacks constructs the same span, and consequently that $\alpha$ is natural in $\Par{\cC}$ as well.
	
	Naturality for the unitors and braiding is a similar consequence of them being isomorphisms.
	The coherence conditions hold because they hold in $\cat{C}$.
\end{proof}

The partialization construction thus preserves the symmetric monoidal structure (that is, the subcategory inclusion $\cC \hookrightarrow \Par{\cC}$ is a strong symmetric monoidal functor).
What is more, $\Par{\cC}$ also inherits the copy--discard structure of $\cat{C}$ via the inclusion of $\cC \hookrightarrow \Par{\cC}$.
Concretely, the copy and discard morphisms are represented by spans 
\begin{equation}
	X \xleftarrow{\id} X \xrightarrow{\cop_X} X \otimes X  \qquad  \text{and}  \qquad  X \xleftarrow{\id} X \xrightarrow{\discard_X} I.
\end{equation}
These make $\Par{\cC}$ into a CD category.

\begin{proposition}\label{partializationRestrictionStructuresAgree}
	Consider a partializable Markov category $\cC$ and a map $u \colon X \to Y$ in $\Par{\cC}$ represented by a span $\left(D,i,f\right)$.
	As a category of spans, $\Par{\cC}$ has a restriction structure given by $\overline{\left(D,i,f\right)} \coloneqq \left(D,i,i\right)$.
	The restriction $\bar u$ of $u$ in terms of spans is the same as the domain $\dom u$ of $u$ in the sense of \cref{domainDefinition}.
	In particular, the assignment $u \mapsto \dom u$ defines a restriction structure on $\Par{\cC}$.\footnote{This is the restriction structure we would expect for a positive quasi-Markov category, as in \cref{positiveQuasiMarkovRestriction}. While we do not need positivity here as the restriction structure axioms hold by hypothesis, we will see in \cref{proposition:ParPos} that $\Par{\cC}$ is indeed positive.}
\end{proposition}

\begin{proof}
	The composite $\discard_Y \comp u$ is represented by the span $\left(D,i,\discard_X\right)$ (here we use that $\cC$ is Markov).
	In light of this, we compute the domain $\dom{u}$
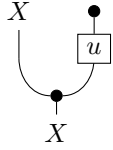
	\[
		\begin{tikzpicture}
	\begin{pgfonlayer}{nodelayer}
		\node [style=morphism] (12) at (1, 0.5) {$u$};
		\node [style=none] (13) at (-1, 0.25) {};
		\node [style=none] (14) at (1, 0.25) {};
		\node [style=none] (19) at (-1, 1) {};
		\node [style=none] (20) at (-1, 1.5) {$X$};
		\node [style=bn] (21) at (0, -0.75) {};
		\node [style=bn] (22) at (1, 1.5) {};
		\node [style=none] (23) at (0, -1.25) {};
		\node [style=none] (24) at (0, -1.75) {$X$};
	\end{pgfonlayer}
	\begin{pgfonlayer}{edgelayer}
		\draw [in=-90, out=-180] (21) to (13.center);
		\draw [in=-90, out=0] (21) to (14.center);
		\draw (12) to (22);
		\draw (13.center) to (19.center);
		\draw (23.center) to (21);
	\end{pgfonlayer}
\end{tikzpicture}
	\]
	in terms of spans using the second pullback from \cref{corollary:SPTSUniversality} by forming
	
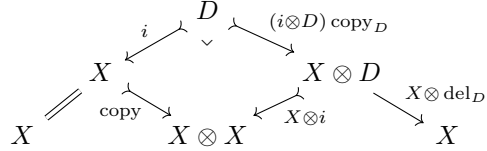
\[
	\begin{tikzcd}[sep=small]
		& & D \arrow[rd, "\left(i \otimes D\right) \comp \cop_D", tail] \arrow[ld, "i"', tail] \arrow["\lrcorner"{anchor=center, pos=0.125, rotate=-45}, draw=none, dd] & & \\
		& X \arrow[ld, Rightarrow, no head] \arrow[rd, "\cop"', tail] & & X \otimes D \arrow[ld, "X \otimes i", tail] \arrow[rd, "X \otimes {} \discard_D"] & \\
		X & & X \otimes X                                              &                                                                           & X
	\end{tikzcd}
	\]
	This is (the class represented by) the span $\left(D,i,i\right)$.
\end{proof}

\begin{corollary}\label{partializationsQuasiMarkov}
	Consider a partializable Markov category $\cC$.
	Every morphism of $\Par{\cC}$ is quasi-total, so that $\Par{\cC}$ is a quasi-Markov category.
\end{corollary}
\begin{proof}
	As noted in the proof of \cref{positiveQuasiMarkovRestriction}, this is merely the fact that the restriction structure axiom ``R.$1$'' (recall \cref{restrictionCategoryDefinition}) for this restriction structure is the statement of quasi-totality (\cref{def:quasi-total}).
\end{proof}
\begin{corollary}\label{prop:ParExtParOrd}
	Consider a partializable Markov category $\cC$.
	The extension partial order $\domext$ (\cref{def:dom_ext}) on the hom-sets of $\Par{\cC}$ coincides with the partial order $\ge$ defined in terms of spans (\cref{spanPartialOrder}).
\end{corollary}
\begin{proof}
	In light of \cref{partializationRestrictionStructuresAgree}, this is a consequence of the fact that the partial order in terms of spans and that in terms of domain idempotents are the partial orders induced by the same restriction structure.
	That is, both partial orders are defined by $u \le v$ whenever $u = v \comp \bar u$, where $\overline{\left(D,i,f\right)}$ is either $\dom{\left(D,i,f\right)}$ or $\left(D,i,i\right)$.
	The former is the definition of $\domext$ (\cref{def:dom_ext}), while the latter can be shown to be equivalent to the span definition of $\ge$ (\cref{spanPartialOrder}).
	The equivalence of the latter is standard in restriction category theory, but we include a proof for explicitness.

	For this, we must show that for parallel spans $(D_f,i,f),(D_g,j,g) \colon X \to Y$ in $\Par{\cC}$, we have $\left(D_f,i,f\right) = \left(D_g,j,g\right) \comp \left(D_f,i,i\right)$ if and only if
	
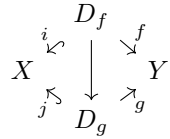
\[
		\begin{tikzcd}[sep=tiny]
		& D_f \ar[hook]{dl}[swap]{i} \ar{dr}{f} \ar{dd}\\
			X && Y \\
			  & D_g \ar[hook]{ul}{j} \ar{ur}[swap]{g}
		\end{tikzcd}
	\]
Note that the composite $\left(D_g,j,g\right) \comp \left(D_f,i,i\right)$ is computed by forming the pullback
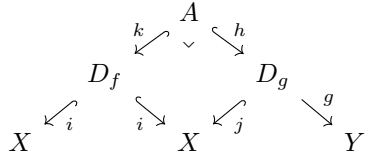
\[
	\begin{tikzcd}[sep=small]
		& & A \arrow[ld, hook, "k"'] \arrow[rd, hook, "h"] \arrow["\lrcorner"{anchor=center, pos=0.125, rotate=-45}, draw=none, dd]& & \\
		& D_f \arrow[ld, "i", hook] \arrow[rd, "i"', hook] & & D_g \arrow[rd, "g"] \arrow[ld, "j", hook] & \\
		X & & X & & Y
	\end{tikzcd}
\]
Thus, if this is equal to $\left(D_f,i,f\right)$, there must be an isomorphism $t\colon D_f \xrightarrow{\simeq} A$ such that $i = i \comp k \comp t$ and $f = g \comp h \comp t$.\footnote{This line of argument will return more generally in \cref{lemma:ParInvState}.}
In particular, $h\comp t$ is a morphism $D_f \to D_g$ witnessing the span order of \cref{spanPartialOrder}.

On the other hand given a map $D_f\xrightarrow{t} D_g$ such that $i = j \comp t$ and $f = g \comp t$, the pullback $A$ computing $\left(D_g,j,g\right) \comp \left(D_f,i,i\right)$ can be computed as
	\[
		\begin{tikzcd}
			D_f \arrow[r, "t"] \arrow[d, Rightarrow, no head] \arrow[rd, "\lrcorner"{anchor=center, pos=0.125}, draw=none] & D_g \arrow[r, Rightarrow, no head] \arrow[d, Rightarrow, no head] \arrow[rd, "\lrcorner"{anchor=center, pos=0.125}, draw=none]& D_g \arrow[d, "j", hook] \\
			D_f \arrow[r, "t"'] & D_g \arrow[r, "j"', hook] & X
		\end{tikzcd}
	\]
	which by construction witnesses the equality $\left(D_f,i,f\right) = \left(D_g,j,g\right) \comp \left(D_f,i,i\right)$.
\end{proof}

The original Markov category sits within its partialization and can be recognized in terms of spans whose domain inclusions are isomorphisms.
Alternatively, it also corresponds to the total morphisms in $\Par{\cC}$.

\begin{proposition}\label{proposition:SpanCDTotal}
	Consider a partializable Markov category $\cC$.
	A morphism of $\Par{\cC}$ is total (commutes with deletion) if and only if it is (the inclusion of) a morphism from $\cC$.\footnote{This is really about the total morphisms in the sense of $p$-categories or restriction categories being the same as the total morphisms in the sense of CD categories. In this sense it is also a corollary of \cref{partializationRestrictionStructuresAgree}.}
\end{proposition}

\begin{proof}
	The backward direction is immediate.
	
	For the converse, assume that $u$ is a total morphism of $\Par{\cC}$ represented by a span $\left(D,i,f\right)$.
	The condition $\discard\comp u = \discard$ corresponds to the existence of an isomorphism that makes the diagram
	\[
	\begin{tikzcd}
		& D \arrow[ld, "i"', tail] \arrow[r, "f"] \arrow[d, "\simeq" {anchor = south, rotate=270}] & Y \arrow[d, "\discard"] \\
		X \arrow[r, Rightarrow, no head] & X \arrow[r, "\discard"']                                    & I                      
	\end{tikzcd}
	\]
	commute.
	Commutativity of the triangle on the left is equivalently invertibility of the domain inclusion $i$, and hence $u$ can equivalently be represented by $\left(X,\id, f \comp i^{-1} \right)$, proving the claim.
\end{proof}

We can also identify copyable morphisms in $\Par{\cC}$ as those in the partialization of the deterministic subcategory $\cC_\det$.

\begin{proposition}\label{proposition:ParCDDeterminism}
	Consider a partializable Markov category $\cC$.
	A morphism $X \xrightarrow{u} Y$ of $\Par{\cC}$ represented by a span $\bigl(X \xleftarrow{i} A \xrightarrow{f} Y\bigr)$ is copyable if and only if $f$ is deterministic.
\end{proposition}
\begin{proof}
	The composite $\cop \comp u$ is represented by the span $\left(A,i,\cop \comp f\right)$.
	Using the third pullback from \cref{corollary:SPTSUniversality}, the composite $\left(u \otimes u\right) \comp \cop$ can be computed by forming the diagram
	\begin{equation}\label{eq:two_copies}
		\begin{tikzcd}[sep=small]
			&                                                       & A \arrow[ld, "i"'] \arrow[rd, "\cop"] \arrow["\lrcorner"{anchor=center, pos=0.125, rotate=-45}, draw=none, dd] &                                                                 &            \\
			& X \arrow[Rightarrow, no head, ld] \arrow[rd, "\cop"'] &                                                        & A \otimes A \arrow[rd, "f \otimes f"] \arrow[ld, "i \otimes i"] &            \\
			X &                                                       & X\otimes X                                             &                                                                 & Y\otimes Y
		\end{tikzcd}
	\end{equation}
	and is thus represented by the span $\bigl(X \xleftarrow{i} A \xrightarrow{\left(f \otimes f\right) \comp \cop} Y\bigr)$.
	Therefore, $u$ is copyable if and only if $f$ is copyable, which, in a Markov category, is equivalent to being deterministic.
\end{proof}

\begin{example}\label{example:ParBorelStochDeterminism}
	For the partializable Markov category $\BorelStoch$, the copyable morphisms of its partialization are precisely the spans whose right way maps are deterministic, i.e.\ measurable maps.
	Thus, copyable morphisms of type $X \to Y$ are given by pairs $(S,f)$ of a measurable subset $S$ of $X$ and a measurable map $f \colon S \to Y$.
\end{example}

\subsection{Positivity}

\begin{proposition}\label{proposition:ParPos}
	Consider a partializable Markov category $\cC$.
	Then, $\Par{\cC}$ is positive as well.
\end{proposition}

\begin{proof}
	We must show that for morphisms $u = \left(X \xleftarrow{i} A \xrightarrow{f} Y\right), v = \left(Y \xleftarrow{j} B \xrightarrow{g} Z\right)$ of $\Par{\cC}$ such that $v \comp u$ is copyable, we have
	\begin{equation}\label{equation:posDefReq}
		\begin{tikzpicture}
	\begin{pgfonlayer}{nodelayer}
		\node [style=none] (0) at (3, -2) {};
		\node [style=bn] (1) at (3, 0) {};
		\node [style=none] (2) at (2, 1) {};
		\node [style=none] (3) at (4, 1) {};
		\node [style=none] (4) at (4, 1) {};
		\node [style=morphism] (5) at (2, 1.25) {$v$};
		\node [style=none] (6) at (4, 1.25) {};
		\node [style=none] (7) at (-3, -2) {};
		\node [style=bn] (8) at (-3, -1.25) {};
		\node [style=none] (9) at (-4, 0) {};
		\node [style=none] (10) at (-2, 0) {};
		\node [style=none] (11) at (-2, 0) {};
		\node [style=none] (12) at (2, 2.25) {};
		\node [style=none] (13) at (4, 2.25) {};
		\node [style=none] (14) at (-4, 2.25) {};
		\node [style=none] (15) at (-2, 2.25) {};
		\node [style=none] (16) at (0, 0) {$=$};
		\node [style=morphism] (17) at (3, -1) {$u$};
		\node [style=none] (18) at (2, 2.75) {$Z$};
		\node [style=none] (19) at (4, 2.75) {$Y$};
		\node [style=none] (20) at (-4, 2.75) {$Z$};
		\node [style=none] (21) at (-2, 2.75) {$Y$};
		\node [style=none] (22) at (3, -2.5) {$X$};
		\node [style=none] (23) at (-3, -2.5) {$X$};
		\node [style=morphism] (24) at (-4, 0) {$u$};
		\node [style=morphism] (25) at (-2, 0) {$u$};
		\node [style=morphism] (26) at (-4, 1.25) {$v$};
	\end{pgfonlayer}
	\begin{pgfonlayer}{edgelayer}
		\draw [style=none] (0.center) to (1);
		\draw [style=none, bend left=45] (1) to (2.center);
		\draw [style=none, bend right=45] (1) to (3.center);
		\draw [style=none] (7.center) to (8);
		\draw [style=none, bend left=45] (8) to (9.center);
		\draw [style=none, bend right=45] (8) to (10.center);
		\draw (12.center) to (5);
		\draw (13.center) to (6.center);
		\draw (14.center) to (9.center);
		\draw (15.center) to (11.center);
		\draw (2.center) to (5);
		\draw (4.center) to (6.center);
	\end{pgfonlayer}
\end{tikzpicture}
	\end{equation}
	Form the pullback
	\[
	\begin{tikzcd}
		C \arrow[d, "\iota"', tail] \arrow[r, "h"] \arrow[rd, "\lrcorner"{anchor=center, pos=0.125}, draw=none]& B \arrow[d, "j", tail] \\
		A \arrow[r, "f"']                          & Y                     
	\end{tikzcd}
	\]
	So that $v \comp u$ can be computed as $(X \xleftarrow{i \comp \iota} C \xrightarrow{g \comp h} Z)$.
	As this is copyable, so is $g \comp h$ (in $\cC$, where copyability and determinism coincide).
	
	The left morphism of \cref{equation:posDefReq} can be computed as
	\[
	\begin{tikzcd}[sep=small]
		&                                              &                                             & C \arrow[ld, "\iota"', tail] \arrow[rd, "(C \otimes \iota)\comp \cop_C"] \arrow["\lrcorner"{anchor=center, pos=0.125, rotate=-45}, draw=none, dd]&                                                                            &                                                                        &             \\
		&                                              & A \arrow[ld, "i"', tail] \arrow[rd, "\cop"] \arrow["\lrcorner"{anchor=center, pos=0.125, rotate=-45}, draw=none, dd]&                                                                        & C \otimes A \arrow[rd, "h \otimes f"] \arrow[ld, "\iota \otimes A"', tail] \arrow["\lrcorner"{anchor=center, pos=0.125, rotate=-45}, draw=none, dd]&                                                                        &             \\
		& X \arrow[ld, Rightarrow, no head] \arrow[rd, "\cop"'] &                                             & A \otimes A \arrow[ld, "i \otimes i", tail] \arrow[rd, "f \otimes f"'] &                                                                            & B \otimes Y \arrow[ld, "j \otimes Y", tail] \arrow[rd, "g \otimes Y"'] &             \\
		X &                                              & X \otimes X                                 &                                                                        & Y \otimes Y                                                                &                                                                        & Z \otimes Y
	\end{tikzcd}
	\]
	On the other hand, the right morphism of \cref{equation:posDefReq} can be computed as
	
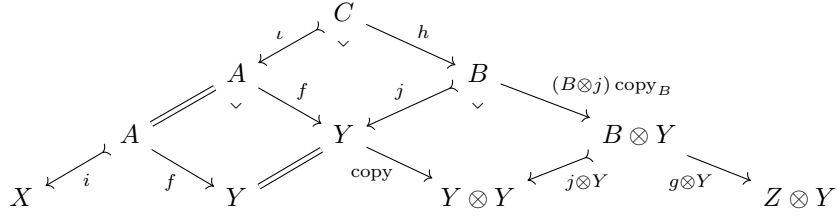
\[
	\begin{tikzcd}[row sep=small]
		&                                          &                                          & C \arrow[ld, "\iota"', tail] \arrow[rd, "h"] \arrow["\lrcorner"{anchor=center, pos=0.125, rotate=-45}, draw=none, dd]&                                                                 &                                                                        &             \\
		&                                          & A \arrow[ld, Rightarrow, no head] \arrow[rd, "f"] \arrow["\lrcorner"{anchor=center, pos=0.125, rotate=-45}, draw=none, dd]&                                              & B \arrow[ld, "j"', tail] \arrow[rd, "(B \otimes j) \comp \cop_B"] \arrow["\lrcorner"{anchor=center, pos=0.125, rotate=-45}, draw=none, dd]&                                                                        &             \\
		& A \arrow[ld, "i", tail] \arrow[rd, "f"'] &                                          & Y \arrow[ld, Rightarrow, no head] \arrow[rd, "\cop"'] &                                                                 & B \otimes Y \arrow[ld, "j \otimes Y", tail] \arrow[rd, "g \otimes Y"'] &             \\
		X &                                          & Y                                        &                                              & Y \otimes Y                                                     &                                                                        & Z \otimes Y
	\end{tikzcd}
	\]
	
	Thus (as $f \comp \iota = j \comp h$), the required equality \cref{equation:posDefReq} is immediate if it is shown that
	
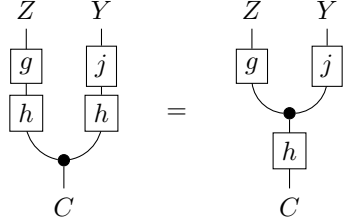
\[
	\begin{tikzpicture}
	\begin{pgfonlayer}{nodelayer}
		\node [style=none] (0) at (3, -2) {};
		\node [style=bn] (1) at (3, 0) {};
		\node [style=none] (2) at (2, 1) {};
		\node [style=none] (3) at (4, 1) {};
		\node [style=none] (4) at (4, 1) {};
		\node [style=morphism] (5) at (2, 1.25) {$g$};
		\node [style=morphism] (6) at (4, 1.25) {$j$};
		\node [style=none] (7) at (-3, -2) {};
		\node [style=bn] (8) at (-3, -1.25) {};
		\node [style=none] (9) at (-4, 0) {};
		\node [style=none] (10) at (-2, 0) {};
		\node [style=none] (11) at (-2, 0) {};
		\node [style=none] (12) at (2, 2.25) {};
		\node [style=none] (13) at (4, 2.25) {};
		\node [style=none] (14) at (-4, 2.25) {};
		\node [style=none] (15) at (-2, 2.25) {};
		\node [style=none] (16) at (0, 0) {$=$};
		\node [style=morphism] (17) at (3, -1) {$h$};
		\node [style=none] (18) at (2, 2.75) {$Z$};
		\node [style=none] (19) at (4, 2.75) {$Y$};
		\node [style=none] (20) at (-4, 2.75) {$Z$};
		\node [style=none] (21) at (-2, 2.75) {$Y$};
		\node [style=none] (22) at (3, -2.5) {$C$};
		\node [style=none] (23) at (-3, -2.5) {$C$};
		\node [style=morphism] (24) at (-4, 0) {$h$};
		\node [style=morphism] (25) at (-2, 0) {$h$};
		\node [style=morphism] (26) at (-4, 1.25) {$g$};
		\node [style=morphism] (42) at (-2, 1.25) {$j$};
	\end{pgfonlayer}
	\begin{pgfonlayer}{edgelayer}
		\draw [style=none] (0.center) to (1);
		\draw [style=none, bend left=45] (1) to (2.center);
		\draw [style=none, bend right=45] (1) to (3.center);
		\draw [style=none] (7.center) to (8);
		\draw [style=none, bend left=45] (8) to (9.center);
		\draw [style=none, bend right=45] (8) to (10.center);
		\draw [in=90, out=-90] (12.center) to (5);
		\draw (13.center) to (6);
		\draw (14.center) to (9.center);
		\draw (2.center) to (5);
		\draw (4.center) to (6);
		\draw (15.center) to (42);
		\draw (42) to (25);
	\end{pgfonlayer}
\end{tikzpicture}
	\]
	This is a consequence of
	
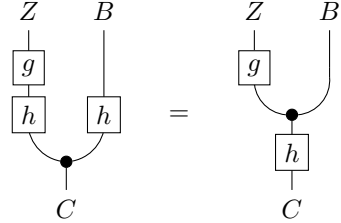
\[
	\begin{tikzpicture}
	\begin{pgfonlayer}{nodelayer}
		\node [style=none] (0) at (3, -2) {};
		\node [style=bn] (1) at (3, 0) {};
		\node [style=none] (2) at (2, 1) {};
		\node [style=none] (3) at (4, 1) {};
		\node [style=none] (4) at (4, 1) {};
		\node [style=morphism] (5) at (2, 1.25) {$g$};
		\node [style=none] (6) at (4, 1.25) {};
		\node [style=none] (7) at (-3, -2) {};
		\node [style=bn] (8) at (-3, -1.25) {};
		\node [style=none] (9) at (-4, 0) {};
		\node [style=none] (10) at (-2, 0) {};
		\node [style=none] (11) at (-2, 0) {};
		\node [style=none] (12) at (2, 2.25) {};
		\node [style=none] (13) at (4, 2.25) {};
		\node [style=none] (14) at (-4, 2.25) {};
		\node [style=none] (15) at (-2, 2.25) {};
		\node [style=none] (16) at (0, 0) {$=$};
		\node [style=morphism] (17) at (3, -1) {$h$};
		\node [style=none] (18) at (2, 2.75) {$Z$};
		\node [style=none] (19) at (4, 2.75) {$B$};
		\node [style=none] (20) at (-4, 2.75) {$Z$};
		\node [style=none] (21) at (-2, 2.75) {$B$};
		\node [style=none] (22) at (3, -2.5) {$C$};
		\node [style=none] (23) at (-3, -2.5) {$C$};
		\node [style=morphism] (24) at (-4, 0) {$h$};
		\node [style=morphism] (25) at (-2, 0) {$h$};
		\node [style=morphism] (26) at (-4, 1.25) {$g$};
	\end{pgfonlayer}
	\begin{pgfonlayer}{edgelayer}
		\draw [style=none] (0.center) to (1);
		\draw [style=none, bend left=45] (1) to (2.center);
		\draw [style=none, bend right=45] (1) to (3.center);
		\draw [style=none] (7.center) to (8);
		\draw [style=none, bend left=45] (8) to (9.center);
		\draw [style=none, bend right=45] (8) to (10.center);
		\draw [in=90, out=-90] (12.center) to (5);
		\draw (13.center) to (6.center);
		\draw (14.center) to (9.center);
		\draw [in=90, out=-90] (15.center) to (11.center);
		\draw (2.center) to (5);
		\draw (4.center) to (6.center);
	\end{pgfonlayer}
\end{tikzpicture}
	\]
	which is the result of applying the positivity of $\cC$ to the deterministic $g \comp h$.
\end{proof}

\section{Transfer of properties from a Markov category to its partialization}\label{section:transferProperties}
We now consider notions of use in categorical probability theory that can be transferred from a Markov category to its partialization.

\subsection{Representability}\label{sec:representability}
Representability is a key property of many Markov categories of interest.
Intuitively, it lets us convert between random maps and deterministic maps that instead return the distribution of the random map's output.
In a partially defined context, the intuition is the same, except only on the domain of definition of the random map.

\begin{construction}\label{construction:ParSharps}
	Consider a representable and partializable Markov category $\cC$ with distribution functor $P$.
	Then for fixed objects $X$ and $Y$, we have a well defined map
	\begin{align}
		\Par{\cC}\left(X,Y\right) \to &\,\, \Par{\cC}\left(X,PY\right)\\
		\left(X \xleftarrow{i}A\xrightarrow{f} Y\right) \mapsto & \left(X \xleftarrow{i} A \xrightarrow{f^\sharp} PY\right)
	\end{align}
	as a consequence of isomorphisms being deterministic, as then we have
	\[
	\begin{tikzcd}[row sep=small]
		A \arrow[dd, "\simeq"' {anchor = south, rotate=90}] \arrow[rd, "f"] &   &          & A \arrow[dd, "\simeq"' {anchor = south, rotate = 90}] \arrow[rd, "f^\sharp"] &    \\
		& Y & \implies &                                            & PY \\
		A' \arrow[ru, "f'"']                    &   &          & A' \arrow[ru, "{f'}^\sharp"']                  &   
	\end{tikzcd}
	\]
	As $f^\sharp$ is deterministic it follows that this map actually factors to a map
	\[
	\Par{\cC}(X,Y) \to \Par{\cC}_{\copyable}(X,PY)
	\]
\end{construction}

\begin{proposition}\label{proposition:ParDistObs}
	Consider a partializable Markov category $\cC$ that is additionally representable.
	The sampling maps of $\cC$ include to morphisms of $\Par{\cC}$ that define distribution objects, i.e.\ natural isomorphisms
	\[
		\Par{\cC}_{\copyable}\left(\ph,PY\right) \xrightarrow{\samp_*} \Par{\cC}\left(\ph,Y\right)
	\]
	for every object $Y$.
	Consequently, $\Par{\cC}$ is representable.
\end{proposition}
\begin{proof}
	This argument is more or less a reduction to the bijection on the morphisms of $\cC$ combined with the fact that two spans with the same domain inclusion are equivalent if and only if they have the same ``right way map''.
	
	It suffices to check that the component at $X$ is invertible for arbitrary $X$.
	This map has a section given by \cref{construction:ParSharps} as for any map $f$ of $\cC$, we have $\samp \comp f^\sharp = f$.

	However, this section is also a retraction.
	Indeed, a copyable morphism $X \to PY$ in $\Par{\cC}$ is represented by a span $\left(X \xhookleftarrow{i} A \xrightarrow{g} PY\right)$ with $g$ deterministic.
	For a \emph{deterministic} $A \xrightarrow{g} PY$, we have $\left(\samp \comp g\right)^\sharp = g$.
	Thus, applying the operation of \cref{construction:ParSharps} to $\samp_* \left(X \xhookleftarrow{i} A \xrightarrow{g} PY\right) = \left(X \xhookleftarrow{i} A \xrightarrow{\samp \comp g} Y\right)$ returns $\left(X \xhookleftarrow{i} A \xrightarrow{g} PY\right)$ itself.
\end{proof}

\begin{example}
	The quasi-Markov categories $\Par{\BorelStoch}$, $\Par{\Dist}$  (or more generally, $\Par{\Kl{\left(D_R\right)}}$ for $R$ an entire zerosumfree semiring), and $\Par{\SetMulti}$ are all representable since the underlying partializable Markov categories are.\footnote{Representability of the Markov categories is a consequence of~\cite[Example 3.2, Propositions 3.4 and 3.6]{fritz2023representable}.}
	On the other hand, $\Par{\FinStoch}$ (like the underlying Markov category) is not representable, as the collection of probability measures on a finite set is not finite in general.
\end{example}

\begin{lemma}\label{ParSampPullback}
	Consider a deterministic monomorphism $i \colon A \rightarrowtail X$ in a partializable Markov category $\cC$.
	If $\cC$ is representable, the following naturality square is a pullback\footnote{The author would particularly like to thank Antonio Lorenzin for suggestions here.}
	\[
	\begin{tikzcd}
		PA \arrow[r, "\samp"] \arrow[d, "Pi"', tail] \arrow[rd, "\lrcorner"{anchor=center, pos=0.125}, draw=none] & A \arrow[d, "i", tail] \\
		PX \arrow[r, "\samp"']                       & X.                    
	\end{tikzcd}	
	\]
\end{lemma}
\begin{proof}
		By the axioms of a partializable Markov category, there is a pullback in $\cC$
		\[
	\begin{tikzcd}
		B \arrow[r, "f"] \arrow[d, "j"', tail] \arrow[rd, "\lrcorner"{anchor=center, pos=0.125}, draw=none] & A \arrow[d, "i", tail] \\
		PX \arrow[r, "\samp"']                       & X.                    
	\end{tikzcd}	
	\]
	with $j$ a deterministic monomorphism.

By the commutativity of the diagram, $\samp \comp j = i \comp f$. 
	From representability, we also have $i \comp f = \samp \comp Pi \comp f^{\sharp}$. 
	Since both $j$ and $Pi \comp f^{\sharp}$ are deterministic, and $\samp$ induces a bijection between deterministic morphisms $B \to PX$ and morphisms $B\to X$, we conclude that $Pi\comp f^{\sharp}= j$. 
	In particular, we have the following commutative diagram
	\[
	\begin{tikzcd}
		PA \ar[rrd,bend left,"\samp"]\ar[ddr,bend right,"Pi"', tail] & & \\
		&B \arrow[r, "f"] \arrow[d, "j"', tail] \arrow[rd, "\lrcorner"{anchor=center, pos=0.125}, draw=none] \ar[lu,"f^{\sharp}"'] & A \arrow[d, "i", tail] \\
		&PX \arrow[r, "\samp"']                       & X.                    
	\end{tikzcd}
	\]
	The fact that $Pi$ is a monomorphism\footnote{$P$, being a right adjoint, preserves monomorphisms.} shows that $f^{\sharp}$ has a section given by applying the universal property of the pullback to the pair $\left(Pi,\samp\right)$\footnote{As for such an induced map $s$, one would have $Pi \comp f^\sharp \comp s = j\comp s = Pi$.}.
	However $f^\sharp$ is itself monic as $j$ is, and is consequently invertible.	
\end{proof}

\begin{proposition}\label{proposition:ParPushforwardComputation}
	Consider a partializable Markov category $\cC$.
	The ``pushforward'' $P\left(A,i,f\right)$ of a morphism of $\Par{\cC}$ represented by a span $\left(X \xhookleftarrow{i} A \xrightarrow{f} Y\right)$ can be computed as $PX \xhookleftarrow{Pi} PA \xrightarrow{Pf} PY$.
\end{proposition}
\begin{proof}
	First, $P\left(A,i,f\right)$ represents the left vertical morphism in
	\[
	\begin{tikzcd}
		{\Par{\cC}_{\copyable}(PX,PX)} \arrow[r, "\simeq"] \arrow[d, "{P\left(A,i,f\right)_*}"'] & {\Par{\cC}(PX,X)} \arrow[d, "{\left(A,i,f\right)_*}"] \\
		{\Par{\cC}_{\copyable}(PX,PY)} \arrow[r, "\simeq"]                          & {\Par{\cC}(PX,Y)}                       
	\end{tikzcd}
	\]
	By tracing $\id_{PX}$ through the square, $P\left(A,i,f\right)\colon PX \to PY$ can be computed in terms of the diagram
	\[
	\begin{tikzcd}[sep=small]
		&                                                & B \arrow[rd, "g"] \arrow[ld, "j"', hook] \arrow["\lrcorner"{anchor=center, pos=0.125, rotate=-45}, draw=none, dd] &                                         &   \\
		& PX \arrow[ld, Rightarrow, no head] \arrow[rd, "\samp"'] &                                          & A \arrow[ld, "i", hook] \arrow[rd, "f"] &   \\
		PX &                                                & X                                        &                                         & Y
	\end{tikzcd}
	\]
	as $PX \xleftarrow{j} B \xrightarrow{\left(f\comp g\right)^\sharp} PY$.
	However, in light of \cref{ParSampPullback} such a pullback can be constructed as
	\[
		\begin{tikzcd}[sep = small]
		& PA \arrow[rd, "\samp"] \arrow[ld, "Pi"', hook] \arrow["\lrcorner"{anchor=center, pos=0.125, rotate=-45}, draw=none, dd] & \\
		PX \arrow[rd, "\samp"'] & & A \arrow[ld, "i", hook] \\
		& X &
	\end{tikzcd}
	\]
	Consequently, $P\left(A,i,f\right)$ can be computed as $\left(PA,Pi,\left(f\comp \samp\right)^\sharp\right)$.
	The claim is now a consequence of the fact that $\left(f\comp \samp\right)^\sharp = Pf$ in $\cC$.
\end{proof}

\begin{remark}\label{totalRepresentability}
	Note that the unit and counit of the representability adjunction are total morphisms.
	Thus, the adjunction $\Par{\cC}_{\copyable} \rightleftarrows \Par{\cC}$ restricts to the representability adjunction of $\cC$,
	\[
		\cC_{\det}\cong \Par{\cC}_{\det} \rightleftarrows \Par{\cC}_{\mathrm{tot}} \cong \cC
	\]
\end{remark}

\subsubsection{Partial algebras}
Consider a representable partializable Markov category $\cC$ with distribution monad $P$.
We have seen in \cref{proposition:ParDistObs} that the unit $\delta$ and counit $\samp$ of the representability adjunction for $\cC$ include to define representability data for $\Par{\cC}$.

Thus in particular there is a distribution monad on $\Par{\cC}_{\copyable}$, which acts the same as $P$ on objects and has the same unit and multiplication maps.
Furthermore, it acts on morphisms componentwise
\[
\begin{tikzcd}
	A & D \ar{l}[swap]{d} \ar{r}{f} & B
\end{tikzcd}
\qquad\longmapsto\qquad
\begin{tikzcd}
	PA & PD \ar{l}[swap]{Pd} \ar{r}{Pf} & PB
\end{tikzcd}
\]

\begin{proposition}\label{parAlgChar}
	An algebra for this monad on $\Par{\cC}_{\copyable}$ consists of an object $A$ together with a partial morphism $PA\to A$, represented by a span in $\cat{C}_{\det}$ as 
\[
\begin{tikzcd}
	PA & D \ar[swap,hook']{l}{d} \ar{r}{a} & A
\end{tikzcd}
\]
such that in $\cC_\det$
	\begin{equation}\label{alg_diags}
		\begin{tikzcd}
			A & A \arrow[l, Rightarrow,no head] \arrow[r, tail, "\eta"] & PA \\
			  & |[label={[label distance=-3mm]20:\urcorner},label={[label distance=-3mm]200:\phantom{\llcorner}}]| A \arrow[ul, Rightarrow, no head] \arrow[rd, Rightarrow, no head] \ar[u,Rightarrow, no head] \ar[r, tail, "s"] & D \ar[u, tail, "d"'] \ar{d}{a} \\
			&& A
		\end{tikzcd}
		\qquad\qquad
		\begin{tikzcd}
			P^2A & PD \arrow[l, tail, "Pd"'] \ar{r}{Pa} & PA \\
			P^2A \arrow[u,Rightarrow, no head] \ar{d}[swap]{\mu} & |[label={[label distance=-3mm]20:\urcorner},label={[label distance=-3mm]200:\llcorner}]| D' \arrow[u,tail] \ar{r} \ar{d} \arrow[l,tail] & D \arrow[u,tail,"d"'] \ar{d}{a} \\
			PA & D \arrow[l,tail,"d"] \ar{r}[swap]{a} & A
		\end{tikzcd}
	\end{equation}
	Or in other words:
\begin{enumerate}
	\item\label{parAlgUnitCond} The unit map $\eta\colon A\to PA$ factors (uniquely) through the domain map $d$ into a section $s$ of $a$;\footnote{In this case $\eta$ and $d\colon D\to PA$ necessarily have $A$ as a pullback.}
	\item\label{parAlgMultCond} The multiplication $\mu\colon P^2A\to PA$ and the domain map $d\colon D\to PA$ have the same pullback $D'$ as $Pa$ and $d$;
	\item\label{parAlgCommCond} The diagrams of \cref{alg_diags} commute.
\end{enumerate}
We call an algebra for $P$ on $\Par{\cC}_\copyable$ a \newterm{partial algebra} (in $\Par{\cC}$).
\end{proposition}

Intuitively this means:
\begin{enumerate}
	\item The domain $D$ of the partial algebra contains the $\delta$-distributions;
	\item The distributions on distributions whose means are in $D$ are the random distributions on $D$ for which the distribution of its ``values'' under $a$ (which will typically be some sort of expectation) lies in $D$;
	\item The algebra acts on the mean of such a random distribution in a manner identical to its action on the pushforward of such a random distribution (under the algebra map itself).
\end{enumerate}
\begin{proof}
	The two diagrams correspond to the unit triangle and associativity square of the algebra axioms respectively.
	\[
		\begin{tikzcd}
A \arrow[r, "\delta", hook] \arrow[rd, Rightarrow, no head] & PA \arrow[d, "{\left(D,d,a\right)}"] \\
															& A
\end{tikzcd}
\qquad \qquad
\begin{tikzcd}
P^2A \arrow[d, "\mu"'] \arrow[r, "{P\left(D,d,a\right)}"] & PA \arrow[d, "{\left(D,d,a\right)}"] \\
PA \arrow[r, "{\left(D,d,a\right)}"']                     & A                        
\end{tikzcd}
	\]
	The equivalence of the unit triangle with the first diagram is an instance of \cref{lemma:PartialRetracts} (with $s$ being the section of the ``right way map'' $a$ that $\delta$ factors through).
	This shows claim~\ref{parAlgUnitCond}.
	So consider the associativity square, keeping in mind that $P\left(PA \xleftarrow{d} D \xrightarrow{a} A\right) = P^2A \xleftarrow{Pd} PD \xrightarrow{Pa} PA$.
	The composites $\left(D,d,a\right) \comp \mu$ and $\left(D,d,a\right) \comp \left(PD,Pd,Pa\right)$ are represented by the spans
	\[
		\begin{tikzcd}[sep=small]
     &                                                & D_1 \arrow[ld, hook] \arrow[rd] \arrow["\lrcorner"{anchor=center, pos=0.125, rotate=-45}, draw=none, dd]&                                         &   \\
     & P^2A \arrow[rd, "\mu"'] \arrow[ld, Rightarrow, no head] &                                 & D \arrow[rd, "a"] \arrow[ld, "d", hook] &   \\
P^2A &                                                & PA                              &                                         & A
\end{tikzcd}
\qquad \qquad
\begin{tikzcd}[sep=small]
     &                                              & D_2 \arrow[ld, hook] \arrow[rd] \arrow["\lrcorner"{anchor=center, pos=0.125, rotate=-45}, draw=none, dd]&                                         &   \\
     & PD \arrow[rd, "Pa"'] \arrow[ld, "Pd"', hook] &                                 & D \arrow[rd, "a"] \arrow[ld, "d", hook] &   \\
P^2A &                                              & PA                              &                                         & A
\end{tikzcd}
\]
Thus, their equivalence implies the existence of an isomorphism $D_1 \xrightarrow{\simeq} D_2$ commuting with the composite span legs.
Thus we may replace $D_1$ with $D_2$ using this isomorphism, with $D_2$ then playing the role of $D'$ in the statement of the proposition.
This is precisely claim~\ref{parAlgMultCond}.
Furthermore, once this identification is made (so that the span legs are equal on the nose, not just up to the isomorphism between the apices), the equality of span legs is precisely claim~\ref{parAlgCommCond}.
\end{proof}

\begin{proposition}\label{IntParAlg}
	In $\Par{\cat{BorelStoch}}$, the integration map gives $\Rnneg$ the structure of a partial algebra (defined on the distributions where this integral is finite).\footnote{The author would like to thank Paolo Perrone for suggesting this line of example.}
\end{proposition}

\begin{proof}
	Consider the following subset of $P\Rnneg$:
	\[
		D \coloneqq \left\{ p\in P\Rnneg \sot \int_{\Rnneg}  x\,p\left(dx\right) < \infty \right\}
	\]
	For $p\in D$, set
	\[
		r\left(p\right) \coloneqq \E{p} = \int_{\Rnneg} x\,p\left(dx\right)
	\] 
	We must show that $\left(P\Rnneg \hookleftarrow D \xrightarrow{r} \Rnneg\right)$ equips $\Rnneg$ with a $P$-algebra structure. 

	Recall that the Giry $\sigma$-algebra~\cite[Section 1.2]{giry} on $P\Rnneg$ is such that for all measurable $A \subseteq \Rnneg$, the assignment $\mathrm{ev}_A\colon P\Rnneg \to \Rnneg$, $p \mapsto p\left(A\right)$ is measurable.
	Consequently for any simple function $f = \sum_{i=1}^n x_i \IF_{A_i}$,\footnote{Where $\IF_A$ is the indicator function at $A$.} we have a measurable map
	\begin{align*}
		r_f \colon P\Rnneg \to \Rnneg & &p \mapsto \int_{\Rnneg} f\left(x\right)\, p\left(dx\right) = \sum_{i=1}^n x_ip\left(A_i\right)
\end{align*}

	Now if $f_n \uparrow \id_{\Rnneg}$ is an approximation of the identity of $\Rnneg$ by simple functions, then we also have $\int_{\Rnneg} f_n p\left(dx\right) \uparrow \int_{\Rnneg} x\,p\left(dx\right)$.
	In particular, this monotone sequence converges if and only if it is bounded, so
	\[
		D = \bigcup_{m \in \N}\bigcap_{n \in \N} \left\{p \sot \int_{\Rnneg} f_n p\left(dx\right) < m\right\}
	\]
	Here each innermost set $\left\{p \sot \int_{\Rnneg} f_n p\left(dx\right) < m\right\} = r_{f_n}^{-1}\left(-\infty,m\right]$ is measurable, so $D$ is as well.
	Furthermore, on $D$ we deduce from $r_{f_n} \uparrow r$ that $r$ is measurable as well.

	Thus $\left(P\Rnneg \hookleftarrow D \xrightarrow{r} \Rnneg\right)$ is a span in $\mathsf{BorelMeas} \cong \BorelStoch_\det$.
	We verify the conditions of \cref{parAlgChar}.
	For condition~\ref{parAlgUnitCond}, note that for all $x \in \Rnneg$, $\int_{\Rnneg} y\, \delta_x\left(dy\right) = x$.
	Thus, $\delta_{\Rnneg}$ factors across $D$ through a section of $r$.

	For condition~\ref{parAlgMultCond}, we compute the pullbacks in question
\[
	\begin{tikzcd}
C_1 \arrow[d, hook] \arrow[r] \arrow[rd, "\lrcorner"{anchor=center, pos=0.125}, draw=none]& D \arrow[d, "d", hook] \\
P^2\Rnneg \arrow[r, "\mu"']   & P\Rnneg               
\end{tikzcd}
\qquad \qquad
\begin{tikzcd}
C_2 \arrow[d, hook] \arrow[r] \arrow[rd, "\lrcorner"{anchor=center, pos=0.125}, draw=none]& D \arrow[d, "d", hook] \\
PD \arrow[r, "Pr"']           & P\Rnneg               
\end{tikzcd}
\]
As in \cref{example:BorelStochPartializable}, we may compute the pullbacks $C_1$ and $C_2$ explicitly as measurable subspaces of $P^2\Rnneg$ and $PD$ respectively.
Explicitly,
\begin{align*}
	C_1 = \left\{\pi \in P^2\Rnneg \sot \mu\left(\pi\right) \in D\right\} & & C_2 = \left\{\pi \in PD \sot r_*\pi \in D\right\}
\end{align*}

The first is the $\pi \in P^2\Rnneg$ such that
\[
	\int_{x \ge 0} x\, \mu\left(\pi\right)\left(dx\right) = \int_{x \ge 0} x \int_p \pi\left(dp\right)\, p\left(dx\right) = \int_p \left(\int_{x \ge 0} x\, p\left(dx\right)\right)\, \pi\left(dp\right)< \infty
\]

Note that a necessary condition for this to hold is that $\int_{x \ge 0}x \, p\left(dx\right)$ is almost surely finite, or in other words that $\pi\left(P^2\Rnneg\setminus D\right) = 0$.
Such a $\pi$ can therefore be identified with a distribution on $D$.

It now therefore suffices to show that for $\pi \in PD$, $\left(r\comp \mu\right)\left(\pi\right) = \left(r\comp r_*\right)\left(\pi\right)$.
Indeed, that would imply that $C_1 = C_2$ (as they are defined by finiteness of one and the other respectively), as well as demonstrate condition~\ref{parAlgCommCond}.
For this, we need only observe
\begin{align*}
	r\left(\mu\left(\pi\right)\right) =& \int_{x\ge 0} x\, \mu\left(\pi\right)\left(dx\right) = \int_{x \ge 0} x \int_{p} p\left(dx\right)\, \pi\left(dp\right) = \int_{p} \left(\int_{x \ge 0} x\, p\left(dx\right)\right)\, \pi\left(dp\right)\\
	=& \int_{p} r\left(p\right)\, \pi\left(dp\right) = \int_{y \ge 0} y\, \left(r_*\pi\right)\left(dy\right) = r\left(r_*\left(\pi\right)\right) \qedhere
\end{align*}
\end{proof}

\begin{warning}\label{BochnerIntNotStrictAlg}
	One might hope that the integration map, perhaps defined on the distributions with finite first moment would make $\R$ a partial algebra as well.
	That is, one might conjecture that the analogous span $\left(P\R \hookleftarrow D' \xrightarrow{r} \R\right)$, where $D' \coloneqq \left\{ p\in P\R \sot \int_{\R}  \vert x\vert \,p\left(dx\right) < \infty \right\}$ (and $r\left(p\right) \coloneqq \E{p}$) would satisfy the conditions of \cref{parAlgChar}.
	However, this is false.\footnote{We leave open the question of whether there is a canonical way to make $\R$ a \emph{strict} partial algebra.}

	This span does indeed define a map of $\Par{\mathsf{BorelMeas}}$, and even satisfies the unit triangle condition. But the multiplication square only commutes up to restriction of domain.
	To be precise, we only have
	\[
		\begin{tikzcd}[sep=small]
& PPX \arrow[dr, "Pr"] \arrow[dl, "\mu"'] & \\
PX \arrow[dr, "r"'] & \le & PX \arrow[dl, "r"] \\ 
& X &
\end{tikzcd}
\]

Here, the left composite is defined on those $\pi \in PD'$ such that $\int_p\int_x \left|x\right|\, p\left(dx\right)\pi\left(dp\right) < \infty$, while the right composite is defined on those where $\int_p\left|\int_x x\, p\left(dx\right)\right| \pi\left(dp\right) < \infty$.
And indeed, the latter condition does not imply the former.

For instance, if $U_n$ is the uniform distribution on $\left[-n,n\right]$ for each $n$, we may consider mixtures $\pi \coloneqq \sum_n c_n\delta_{U_n} \in PD'$ for weights $c_n$ summing to $1$.
This satisfies the latter condition as each $\E{U_n} = 0$.
However $\int_x \left|x\right|\, U_n\left(dx\right) = n$, so $\int_p\int_x \left|x\right|\, p\left(dx\right) \pi\left(dp\right) = \sum_n nc_n$ which need not be finite (for instance, when $c_n \propto n^{-2}$).
\end{warning}

\subsection{Conditionals}
Conditionals also transfer naturally from a partializable Markov category to its partialization, and can in a sense be constructed on the ``maximum possible domain''.

\begin{proposition}\label{ConditionalsExist}
	Consider a partializable Markov category $\cC$ with conditionals.
	Given a morphism $u \colon A \to X\otimes Y$ in $\Par{\cC}$ represented by a span $\bigl(A \xhookleftarrow{i} D \xrightarrow{f} X \otimes Y\bigr)$, the conditional $u_{\vert X} \colon X \otimes A \to Y$ exists and is represented by the span $\bigl(X \otimes A \xhookleftarrow{X \otimes i} X \otimes D \xrightarrow{f_{\vert X}} Y\bigr)$.
	In particular, $\Par{\cC}$ has conditionals.
\end{proposition}
\begin{proof}
	We compute the right side of \cref{ConditionalDefnEqn} (with $u$ substituted for $f$).
	This can be done by forming the diagram of pullback squares\footnote{We rotate the diagram for ease of reading. The domain inclusions run vertically whereas the ``right way maps'' run horizontal.}
	\[
		\begin{tikzcd}
			D \arrow[d, Rightarrow, no head] \arrow[r, "\cop", hook] \arrow[rd, "\lrcorner"{anchor=center, pos=0.125}, draw=none] & D\otimes D \arrow[d, "D \otimes i"', hook] \arrow[r, "f \otimes D"] \arrow[rd, "\lrcorner"{anchor=center, pos=0.125}, draw=none] & X \otimes Y \otimes D \arrow[d, "(X\otimes Y) \otimes i"', hook] \arrow[r, "\pi_{1,3}"] \arrow[rd, "\lrcorner"{anchor=center, pos=0.125}, draw=none] & X \otimes D \arrow[d, "X \otimes i"', hook] \arrow[r, "\cop \otimes D", hook] \arrow[rd, "\lrcorner"{anchor=center, pos=0.125}, draw=none] & X \otimes X \otimes D \arrow[d, "(X \otimes X) \otimes i", hook] \arrow[r, "X\otimes f_{\vert X}"] & X \otimes Y \\
			D \arrow[d, "i"', hook] \arrow[r, "(D \otimes i) \comp \cop", outer sep=2pt, hook] \arrow[rd, "\lrcorner"{anchor=center, pos=0.125}, draw=none] & D\otimes A \arrow[d, "i\otimes A", hook] \arrow[r, "f\otimes A"] & X \otimes Y \otimes A \arrow[r, "\pi_{1,3}"'] & X \otimes A \arrow[r, "\cop \otimes A"', hook] & X\otimes X \otimes A & \\
A \arrow[d, Rightarrow, no head] \arrow[r, "\cop"', hook] & A \otimes A & & & & \\
A & & & & &            
\end{tikzcd}
	\]
	Here the first square is a pullback by \cref{lemma:factorThroughMonoCart}, the second by \cref{proposition:StablePullbackTensorStability}, and the rest by \cref{corollary:SPTSUniversality}.

	Thus, the domain of the right hand side can be identified with $i$.
	Similarly, the ``right way map'' can be identified with the composite of the right side of \cref{ConditionalDefnEqn} (with $f$ as itself).
	So the ``right way map'' is $f$, and thus the right hand side is the morphism $\left(D,i,f\right) = u$ as desired.
\end{proof}

\begin{corollary}
	Consider a partializable Markov category $\cC$ with conditionals.
	Then $\Par{\cC}$ is a partial Markov category in the sense of Di Lavore and Rom\'an in~\cite[Definition 3.2]{dilavore2023evidential} (and the weaker sense of~\cite[Definition 3.1]{dilavore2024partial}).

	In particular, $\Par{\BorelStoch}$, $\Par{\FinStoch}$, $\Par{\cat{Dist}}$, and $\Par{\cat{SetMulti}}$ are all partial Markov categories (although as mentioned in \cref{parCatsExmpsDiff}, different from the examples of~\cite{dilavore2023evidential,dilavore2024partial}).\footnote{Fritz noted in~\cite[Examples 11.6 and 11.7]{fritz2019synthetic} that $\FinStoch$ and $\BorelStoch$ have conditionals (although for the special case of states Cho and Jacobs had already pointed out in~\cite[Theorem 3.11]{chojacobs2019strings} that existence was a consequence of classical results on regular conditional probabilities~\cite{faden1985conditional}). Di Lavore and Rom\'an remark in~\cite[Remark A.3]{dilavore2023evidential} that conditionals exist in $\Dist$, extending the familiar construction for states that appeared in Cho and Jacobs'~\cite[Example 3.6]{chojacobs2019strings}. Fritz and Klingler have demonstrated that the analogous construction produces conditionals in $\SetMulti$ in~\cite[Proposition 13]{fritz2022dseparation} (this is Proposition 16 in the arXiv version).}
\end{corollary}

\begin{remark}
	Another quasi-Markov category with conditionals (and hence a quasi-Markov partial Markov category) is the usual category of relations $\Rel$ (or more generally, bicategories of relations~\cite[Proposition 2.14]{dilavore2025OrderPartialMarkov}).
	However as remarked earlier (\cref{restPositivityRequirement}), $\Rel$ is not poset-enriched under the restriction partial order of \cref{def:dom_ext}, and thus cannot be positive (\cref{positiveQuasiMarkovEnriched}).
	In particular, unlike Markov categories~\cite[Lemma 11.24]{fritz2019synthetic}, quasi-Markov (and hence CD) categories with conditionals need not be positive!

On the other hand, a property equivalent to positivity for Markov categories is ``deterministic marginal independence''~\cite[Definition 2.4 and Theorem 2.8]{fritz2022dilations}.
For quasi-Markov categories, ``copyable marginal independence'' (the direct generalization of the aforementioned DMI) is still implied by positivity (the argument of~\cite[Proposition 12.4]{fritz2019synthetic} still applies).
However, it can also be shown to be a consequence of the existence of conditionals (and thus can no longer imply positivity, as it would hold for categories like $\Rel$).

	Fortunately, $\Par{\cC}$ is positive irrespective of conditionals.
\end{remark}

\subsection{Idempotent partial morphisms}
We conclude this section with a characterization of idempotent morphisms in $\Par{\cC}$, and a characterization of those idempotent morphisms with properties of importance in categorical probability.
In particular, we are interested in splitting, balance, strength, and being static.

\begin{proposition}\label{proposition:IdempSpans}
	Consider a partializable Markov category $\cC$.
	A morphism $\left(X \xhookleftarrow{i} A \xrightarrow{f} X\right)$ is idempotent if and only if it can be represented by a span $X \xhookleftarrow{i} A \xrightarrow{i\comp e} X$ for an idempotent $e$ of $A$ in $\cC$.
	In other words, the idempotent partial maps are those that act as idempotents on their domain.
\end{proposition}

We prove the statement by means of the following lemma.

\begin{lemma}\label{lemma:ParInvState}
	Consider a partializable Markov category $\cC$ and morphisms $X \xhookleftarrow{i} A \xrightarrow{f} Y$, $Y \xhookleftarrow{j} B \xrightarrow{g} Y$.
	Then, $\left(B,j,g\right) \comp \left(A,i,f\right) = \left(A,i,f\right)$ if and only if $f$ factors through $j$ via a $A \xrightarrow{h} B$ in $\cC$ such that $f = j \comp h = g \comp h$.
\end{lemma}
\begin{proof}[Proof of \Cref{lemma:ParInvState}]
	The composite $\left(B,j,g\right) \comp \left(A,i,f\right)$ can be computed by forming the following pullback.
	\[
	\begin{tikzcd}[row sep=small]
		& & C \arrow[ld, "j'"', hook] \arrow[rd, "f'"] \arrow["\lrcorner"{anchor=center, pos=0.125, rotate=-45}, draw=none, dd]& & \\
		& A \arrow[rd, "f"'] \arrow[ld, "i", hook] & & B \arrow[rd, "g"] \arrow[ld, "j", hook] & \\
		X & & Y & & Y
	\end{tikzcd}
	\]
	The resulting span is equivalent to $\left(A,i,f\right)$ if and only if there is an invertible $t\colon C \to A$ such that $i\comp t = i \comp j'$ and $f \comp t = g \comp f'$.
	In such a scenario, $j' = t$ is invertible, so we can equivalently replace $\left(C,j',f'\right)$ with a pullback
	\[
		\begin{tikzcd}[sep = small]
		& A \arrow[rd, "h"] \arrow[ld, Rightarrow, no head] \arrow["\lrcorner"{anchor=center, pos=0.125, rotate=-45}, draw=none, dd] & \\
		A \arrow[rd, "f"'] & & B \arrow[ld, "j", hook] \\
		& Y &
	\end{tikzcd}
	\]
		for some $h$.
	Since any such commuting square is a pullback (\cref{lemma:factorThroughMonoCart}), we see that the equation in the hypothesis is equivalent to the existence of a $A \xrightarrow{h} B$ such that $f = j\comp h$ and $f = g\comp h$.
\end{proof}

\begin{proof}[Proof of \Cref{proposition:IdempSpans}]
	Applying \cref{lemma:ParInvState}, we see that the idempotence of $\left(A,i,f\right)$ is equivalent to the existence of a $A \xrightarrow{h} A$ such that $f = i\comp h$ and $f = f\comp h$.
	Given the first (and the monicity of $i$), the second equation merely expresses the idempotence of $h$.
\end{proof}

Recall now that an idempotent morphism $e\colon X\to X$ in a category is \emph{split} if it can be written in the form $e=s\comp r$ for some $s\colon E\to X$ and $r\colon X\to E$ such that $r\comp s=\id_E$.

\begin{remark}[(i) of~{\cite[Lemma 2.2]{cockettlack2002partialmaps}}]\label{remark:MonomorphismsTotal}
	Every monomorphism in a restriction category is total.
	In particular, every monomorphism of $\Par{\cC}$ is total, and therefore the inclusion of a (mono)morphism of $\cC$.
	In fact, this is true of any quasi-Markov category as has been demonstrated in the companion work on empirical sampling~\cite[Lemma 2.6]{fritz2025empirical}.

	Consequently, if a map in $\Par{\cC}$ has a section, the section is the inclusion of a map in $\cC$.
	In particular, if an idempotent in $\Par{\cC}$ splits as a section--retraction pair, the section is the inclusion of a map in $\cC$.
\end{remark}

\begin{lemma}\label{lemma:PartialRetracts}
	Consider a morphism $\left(Y \xhookleftarrow{i} A \xrightarrow{r} X\right)$ in $\Par{\cC}$ for a partializable Markov category $\cC$.
	Then, a morphism $s\colon X \to Y$ of $\cC$ includes to a section\footnote{By \cref{remark:MonomorphismsTotal}, this describes all sections in $\Par{\cC}$.} of $\left(A,i,r\right)$ if and only if $s$ factors through $i$ via a section $s'\colon X \to A$ of $r\colon A \to X$.
\end{lemma}
\begin{proof}
	The composition $\left(A,i,r\right) \comp \left(X,\id_X,s\right)$ can be performed by forming the pullback 
	\[
	\begin{tikzcd}[row sep=small]
		& & B \arrow[ld, "j"', hook] \arrow[rd, "g"] \arrow["\lrcorner"{anchor=center, pos=0.125, rotate=-45}, draw=none, dd]& & \\
		& X \arrow[rd, "s"'] \arrow[ld, Rightarrow, no head] & & A \arrow[rd, "r"] \arrow[ld, "i", hook] & \\
		X & & Y & & X
	\end{tikzcd}
	\]
	As with the proof of \cref{lemma:ParInvState}, the equation $\left(A,i,r\right) \comp s = \id_X$ can be restated as the existence of a commutative square
	\[
		\begin{tikzcd}[sep = small]
		& X \arrow[rd, "s'"] \arrow[ld, Rightarrow, no head] \arrow["\lrcorner"{anchor=center, pos=0.125, rotate=-45}, draw=none, dd] & \\
		X \arrow[rd, "s"'] & & A \arrow[ld, "i", hook] \\
		& Y &
	\end{tikzcd}
	\]
	which is always a pullback, such that $r\comp s' = \id_A$.
\end{proof}

\begin{proposition}\label{proposition:SplitParIdemps}
	Consider a partializable Markov category $\cC$ and an idempotent $m$ of $\Par{\cC}$, represented by $X \xhookleftarrow{i} A \xrightarrow{i\comp e} X$ for an idempotent $e$ as in \cref{proposition:IdempSpans}.
	The idempotent $m$ splits if and only if $e$ does.
	
	If $e$ splits as $s\comp r$ for a section-retraction pair $\left(s,r\right)$ in $\cC$, the splitting of $m$ is given by the section-retraction pair $\big(\left(Y,\id_Y,i\comp s\right),\left(A,i,r\right)\big)$.
\end{proposition}
\begin{proof}
	If $e$ splits as in the hypotheses, it is an immediate consequence of the fact that monomorphisms pull back along themselves to the identity that the proposed section-retraction pair do indeed form a splitting of $m$.
	
	Conversely, if $m$ splits, then as in \cref{lemma:PartialRetracts} the splitting is given by a retraction $\left(A,i,r\right)$ and section $i \comp s$ in $\Par{\cC}$ such that $s$ is a section of $r$ in $\cC$.
	Then, the equality $\left(A,i,i\comp e\right) = i\comp s \comp \left(A,i,r\right) = \left(A,i, i\comp s\comp r\right)$ implies that $e = s\comp r$, as claimed.
\end{proof}
\begin{example}\label{Par(BorelStoch)Split}
	Every idempotent in $\Par{\BorelStoch}$ splits, with splittings induced by those in $\BorelStoch$~\cite[Corollary 4.4.5]{fritz2023supports}.
\end{example}

\begin{lemma}\label{lemma:PartialIdemp2trans}
	Consider a partializable Markov category $\cC$ and an idempotent $m$ of $\Par{\cC}$ represented by $\left(A,i,i\comp e\right)$ for an idempotent $e$ of $\cC$.
	Then
	
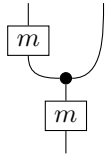
\[
		\begin{tikzpicture}
	\begin{pgfonlayer}{nodelayer}
		\node [style=bn] (0) at (0, 0) {};
		\node [style=morphism] (1) at (0, -1) {$m$};
		\node [style=morphism] (2) at (-1, 1) {$m$};
		\node [style=none] (3) at (0, -2) {};
		\node [style=none] (4) at (1, 2) {};
		\node [style=none] (5) at (-1, 2) {};
	\end{pgfonlayer}
	\begin{pgfonlayer}{edgelayer}
		\draw (1) to (0);
		\draw (3.center) to (1);
		\draw [in=-90, out=-180, looseness=1.25] (0) to (2);
		\draw (2) to (5.center);
		\draw [in=-90, out=0] (0) to (4.center);
	\end{pgfonlayer}
\end{tikzpicture}
	\]
	is represented by $\big(A, i,\left(i \otimes i\right)\comp \tilde e\big)$, where $\tilde e$ is the corresponding
	
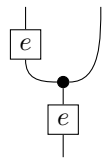
\[
		\begin{tikzpicture}
	\begin{pgfonlayer}{nodelayer}
		\node [style=bn] (0) at (0, 0) {};
		\node [style=morphism] (1) at (0, -1) {$e$};
		\node [style=morphism] (2) at (-1, 1) {$e$};
		\node [style=none] (3) at (0, -2) {};
		\node [style=none] (4) at (1, 2) {};
		\node [style=none] (5) at (-1, 2) {};
	\end{pgfonlayer}
	\begin{pgfonlayer}{edgelayer}
		\draw (1) to (0);
		\draw (3.center) to (1);
		\draw [in=-90, out=-180, looseness=1.25] (0) to (2);
		\draw (2) to (5.center);
		\draw [in=-90, out=0] (0) to (4.center);
	\end{pgfonlayer}
\end{tikzpicture}
	\]
\end{lemma}
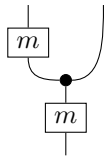
\begin{proof}
	\[
		\begin{tikzpicture}
	\begin{pgfonlayer}{nodelayer}
		\node [style=bn] (0) at (0, 0) {};
		\node [style=morphism] (1) at (0, -1) {$m$};
		\node [style=morphism] (2) at (-1, 1) {$m$};
		\node [style=none] (3) at (0, -2) {};
		\node [style=none] (4) at (1, 2) {};
		\node [style=none] (5) at (-1, 2) {};
	\end{pgfonlayer}
	\begin{pgfonlayer}{edgelayer}
		\draw (1) to (0);
		\draw (3.center) to (1);
		\draw [in=-90, out=-180, looseness=1.25] (0) to (2);
		\draw (2) to (5.center);
		\draw [in=-90, out=0] (0) to (4.center);
	\end{pgfonlayer}
\end{tikzpicture}
	\]
	can be computed by forming the pullbacks
	
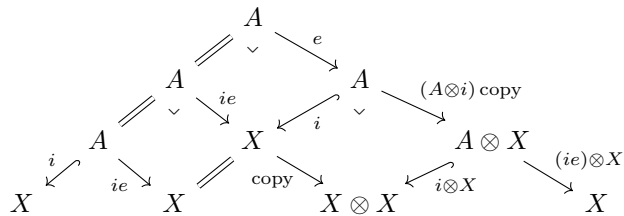
\[
		\begin{tikzcd}[sep= small]
 & & & A \arrow[ld, Rightarrow, no head] \arrow["\lrcorner"{anchor=center, pos=0.125, rotate=-45}, draw=none, dd] \arrow[rd, "e"] & & & \\
 & & A \arrow[ld, Rightarrow, no head] \arrow["\lrcorner"{anchor=center, pos=0.125, rotate=-45}, draw=none, dd]\arrow[rd, "i\comp e"] & & A \arrow[ld, "i", hook] \arrow[rd, "\left(A\otimes i\right) \comp \cop"] \arrow["\lrcorner"{anchor=center, pos=0.125, rotate=-45}, draw=none, dd] & & \\
 & A \arrow[ld, "i"', hook] \arrow[rd, "i\comp e"'] & & X \arrow[rd, "\cop"'] \arrow[ld, Rightarrow, no head] & & A \otimes X \arrow[ld, "i \otimes X", hook] \arrow[rd, "\left(i\comp e\right)\otimes X"] & \\
X & & X & & X \otimes X & & X
\end{tikzcd}
	\]
	the last pullback being an instance of \cref{corollary:SPTSUniversality}.
\end{proof}

\begin{proposition}\label{proposition:PartialIdempStaticStrongBalanced}
	Consider a partializable Markov category $\cC$, and an idempotent $m$ in $\Par{\cC}$ represented by $\left(A,i,i\comp e\right)$ for an idempotent $e$ of $\cC$.
	Then, $m$ is static/strong/balanced (in the sense of~\cite[Definition 4.1.1]{fritz2023supports}) if and only if $e$ is.
\end{proposition}
\begin{proof}
	We show the claim for strong idempotents, the cases for the other two are analogous\footnote{The case for static idempotents is arguably even simpler. In light of the equivalent characterization of balance in~\cite[Proposition 4.1.10 (ii)]{fritz2023supports} the claim for balanced idempotents reduces to an application of \cref{lemma:PartialIdemp2trans} and its symmetric counterpart.}.

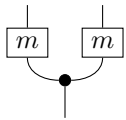
	\[
		\begin{tikzpicture}
	\begin{pgfonlayer}{nodelayer}
		\node [style=bn] (0) at (0, 0) {};
		\node [style=morphism] (2) at (-1, 1) {$m$};
		\node [style=none] (3) at (0, -1) {};
		\node [style=none] (5) at (-1, 2) {};
		\node [style=morphism] (6) at (1, 1) {$m$};
		\node [style=none] (7) at (1, 2) {};
	\end{pgfonlayer}
	\begin{pgfonlayer}{edgelayer}
		\draw [in=-90, out=-180] (0) to (2);
		\draw (2) to (5.center);
		\draw (6) to (7.center);
		\draw (0) to (3.center);
		\draw [in=-90, out=0] (0) to (6);
	\end{pgfonlayer}
\end{tikzpicture}
	\]
	can be computed by forming the pullback (as in \cref{corollary:SPTSUniversality})
	\[
		\begin{tikzcd}[sep= small]
 & & A \arrow[ld, "i"', hook] \arrow[rd, "\cop"] \arrow["\lrcorner"{anchor=center, pos=0.125, rotate=-45}, draw=none, dd] & & \\
 & X \arrow[ld, Rightarrow, no head] \arrow[rd, "\cop"'] & & A \otimes A \arrow[rd, "\left(i\comp e\right)\otimes\left(i\comp e\right)"] \arrow[ld, "i \otimes i", hook] & \\
X & & X\otimes X & & X \otimes X
\end{tikzcd}
	\]
	In light of this computation and monicity of $i \otimes i$, the strongness of $e$ corresponds to that of $m$.
\end{proof}

\begin{example}\label{Par(BorelStoch)balanced}
	As every idempotent in $\BorelStoch$ is balanced~\cite[Corollary 4.1.9]{fritz2023supports}, every idempotent in $\Par{\BorelStoch}$ is as well.
\end{example}

\begin{example}\label{domainIdempotentsStaStrBal}
	In a quasi-Markov category $\cC$, the domain idempotents are copyable, and hence static, strong, and balanced.

	In those of the form $\Par{\cC}$ for partializable $\cC$, these domain idempotents split as well.
\end{example}

\section{Kolmogorov products}\label{section:KolmogorovProducts}

Let us now turn to Kolmogorov products.
In contrast to the previous section, some aspects of the theory of Kolmogorov products change in the passage from the Markov (semicartesian) to the more general quasi-Markov setting.

In general, given a partializable Markov category $\cat{C}$, a Kolmogorov product in $\cat{C}$ induces (a strict) one in $\Par{\cC}$ (defined in the sense of \cref{InfTensorProductDefn}, see \cref{proposition:ParLaxKolProds}).
However as alluded to in \cref{KolProdsNonFunctorialWarning}, an arbitrary family of maps $\bigl(X_k \to Y_k\bigr)_{k \in K}$ in a quasi-Markov category such as $\Par{\cC}$ need not define a cone for the universal property of the (strict) Kolmogorov product to apply to.
In particular (when $\cC$ has $K$-sized Kolmogorov products), this means that the (strict) Kolmogorov product would not define a \emph{functor} $\Par{\cC}^K \to \Par{\cC}$, merely an assignment on objects.
In order to recover this functoriality we define a lax version of Kolmogorov products.

To illustrate explicitly why this is necessary, consider partial maps 
\[
\begin{tikzcd}
	X & D \ar[tail]{l}[swap]{i} \ar{r}{f} & Y 
\end{tikzcd}
\quad\mbox{and}\quad 
\begin{tikzcd}
	A & E \ar[tail]{l}[swap]{j} \ar{r}{g} & B .
\end{tikzcd}
\]
Their tensor product has domain $D\otimes E$:
\[
\begin{tikzcd}
	X\otimes A & D\otimes E \ar[tail]{l}[swap]{i\otimes j} \ar{r}{f\otimes g} & Y\otimes B .
\end{tikzcd}
\]
If we now project onto the first component by marginalizing, we have the following diagram,
\[
\begin{tikzcd}
	X\otimes A \ar{d}{\pi_1} & D\otimes E \ar{d}{\pi_1} \ar[tail]{l}[swap]{i\otimes j} \ar{r}{f\otimes g} & Y\otimes B \ar{d}{\pi_1} \\
	X & D \ar[tail]{l}[swap]{i} \ar{r}{f} & Y 
\end{tikzcd}
\]
which, as a diagram in $\Par{\cC}$, only commutes \emph{laxly}:
\[
\begin{tikzcd}[sep=small]
	& X\otimes A \ar{dl}[swap]{\pi_1} \ar{dr}{\left(D,i,f\right)\otimes\left(E,j,g\right)} & \\
	X \ar{dr}[swap]{\left(D,i,f\right)} & \ge & Y\otimes B \ar{dl}{\pi_1} \\
						   & Y &
\end{tikzcd}
\]
Indeed, the right path has domain $D\otimes E$, while the left path has the larger domain $D\otimes A$.
(This is part of the reason why $\Par{\cC}$ is not a Markov category: the discard maps are not natural.)

More generally now, let $K$ be a set and consider a $K$-indexed tuple of partial maps $\big( X_k \xrightarrow{(D_k,i_k,f_k)} Y_k \big)_{k\in K}$. 
One can form the diagrams of finite tensor products $X^F$ and $Y^F$ and their marginalizations, but because of the reason above, the maps $f_k$ do \emph{not} form a natural transformation between these two diagrams, they only form a \emph{lax} natural transformation. 
These, in general, do \emph{not} induce a morphism between the Kolmogorov products $X^K\to Y^K$ by the usual universal property of~\cite{fritzrischel2019zeroone}. 
In order to have such a map, and so to make Kolmogorov products \emph{functorial on arbitrary partial maps} by their universal property, we need to slightly extend the universal property. 

We therefore define \emph{lax Kolmogorov products}.\footnote{The author would like to thank Tobias Fritz for suggesting a lax approach.}
It can then be shown (see below) that:
\begin{itemize}
	\item Kolmogorov products in $\cC$ include to strict Kolmogorov products in $\Par{\cC}$;
	\item The inclusion of a Kolmogorov product in $\cC$ is also a lax Kolmogorov product in $\Par{\cC}$;
	\item The map $X^K \to Y^K$ induced by a family of morphisms $(X_k \to Y_k)_{k \in K}$ and the universal property of the lax Kolmogorov product is represented componentwise/leg-wise\footnote{That is, having domain the tensor of the domain inclusions in $\cC$, and similarly with the ``right way maps''.} by the corresponding induced maps in $\cC$.
\end{itemize}

Recall the definition of lax cone, which we instantiate directly for our purposes.
\begin{definition}\label{def:lax-cone}
	Let $K$ be a set, and let $\left(X_k\right)_{k\in K}$ be a $K$-indexed tuple of objects in a quasi-Markov category $\cat{C}$.
	Consider (again) the diagram $X^{\left(\phsm\right)}\colon \FinSub{K}^\op\to\cat{C}$ formed by finite products and marginalizations (as in~\cite[Section 3]{fritzrischel2019zeroone}).
	
	A \newterm{lax cone} over the diagram $X^{\left(\phsm\right)}$ is an object $A$ of $\cat{C}$ together with arrows $f_F\colon A\to X^F$ for all finite $F\subseteq K$, such that for all subsets $G\subseteq F\subseteq K$, the following diagram commutes laxly,
	\[
	\begin{tikzcd}[sep=small]
		&&& X^F \ar{dd}{\pi_{F,G}} \\ 
		A \ar{urrr}{f_F} \ar{drrr}[swap]{f_G} && |[overlay,xshift=2mm]| \ge \\
		&&& X^G
	\end{tikzcd}
	\]
	where $\pi_{F,G}$ denotes the functor action (marginalization) of $X^{\left(\phsm\right)}$ on the inclusion $G\subseteq F$.
\end{definition}

\begin{definition}\label{def:lax-kolmogorov}
	In the setting of the definition above, a \newterm{lax infinite tensor product} is a lax cone $\big(X^K \xrightarrow{\pi_F} X^F\big)_{F \subseteq K \mathrm{ finite}}$ which is universal in the following sense: for any other lax cone $\big(A \xrightarrow{f_F} X^F\big)_{F \subseteq K \mathrm{ finite}}$ there is a greatest morphism $A \xrightarrow{g} X^K$ such that $\pi_{F}\comp g \le f_F$ for each finite $F \subseteq K$, and this lax limit is further preserved by tensoring with an arbitrary object $Y$.
	
	We call such a lax infinite tensor product a \newterm{lax Kolmogorov product} when the projections $\pi_F$ are deterministic.\footnote{It can be shown that in such a case, a cone with copyable components defines a copyable induced map.}
\end{definition}

Intuitively, this can be read as asserting that a compatible (up to possibly restricting the domain) family of partial kernels to the finite products corresponds to a kernel to the infinite product, defined on the greatest possible domain. 

\begin{remark}\label{laxLimitTerminology}
	We are following here the convention in enriched or $2$-category (or higher) theory of ``strictness'' corresponding to literal equality, ``strongness'' corresponding to equality up to an invertible $2$-cell,\footnote{The prefix ``pseudo'' is often used for this notion in a $2$-categorical setting, but the distinction will be irrelevant for us as the invertible $2$-cells we consider will always be identities.} and ``laxness'' corresponding to equality up to a $2$-morphism.
	Of course, as we only consider here enrichment in posets, the distinction between strictness and strongness vanishes.

	In the case of \emph{Markov} categories, the distinction between lax and strict vanishes completely, as there are no non-identity $2$-cells (the partial order of \cref{def:dom_ext} collapses).
	There is thus only the notion introduced by Fritz and Rischel in~\cite{fritzrischel2019zeroone}.
\end{remark}

\begin{lemma}\label{lemma:ParFinSubIndLaxLim}
	Consider a partializable Markov category $\cC$ and a diagram\footnote{Note that we do not require that the diagram $J$ be that of finite tensor products and marginalizations.}
	\[
	\FinSub{K}^\op \xrightarrow{J} \cC
	\]
	for some set $K$.
	Assume that the deterministic subobject posets $\Sub_{\det}(A)$ of each object $A$ of $\cC$ have $K$-sized meets.
	If $J$ has a limit in $\cC$, then the inclusion of this limit cone defines both a lax and a strict limit in $\Par{\cC}$.
\end{lemma}
\begin{proof}
	Let $\lim J \xrightarrow{p_F} J_F$ be the limit projections in $\cC$.
	We will also denote for an inclusion $F' \subseteq F$ of finite subsets of $K$, the induced map $J_{F,F'} \colon J_F \to J_{F'}$.
	We also denote $\mathcal F \coloneqq \FinSub{K}$.
	
	Now consider a lax cone $\left(A \xrightarrow{u_F} J_F\right)_{F \in \mathcal F}$ in $\Par{\cC}$.
	If the $u_F$ are represented by $A \xleftarrow{i_F} B_F \xrightarrow{f_F} J_F$, then the lax commutation diagrams $J_{F,F'} \comp u_F \le u_{F'}$ for $F' \subseteq F \subseteq K$ correspond to diagrams in $\cC$
	\[
	\begin{tikzcd}[row sep=small]
		& B_F \arrow[dd, "{t_{F,F'}}"] \arrow[r, "f_F"] \arrow[ld, "i_F"', tail] & J_F \arrow[dd, "{J_{F,F'}}"] \\
		A &                                                                        &                              \\
		& B_{F'} \arrow[r, "f_{F'}"'] \arrow[lu, "i_{F'}", tail]                       & J_{F'}                      
	\end{tikzcd}
	\]
	With the $t_{F,F'}$'s deterministic monomorphisms uniquely determined by the triangles $i_{F'} \comp t_{F,F'} = i_F$.
	Consequently for $F'' \subseteq F' \subseteq F$, $t_{F,F''} = t_{F',F''} \comp t_{F,F'}$.\footnote{This is merely reflecting the fact that $\Sub_{\det}(A)$ is a poset}
	
	Let $B \xrightarrow{i} A$ be the meet\footnote{This exists by assumption when $K$ is infinite, as then $\vert \FinSub{K}\vert = \vert K \vert$. $\Sub_{\det}\left(A\right)$ always has finite meets, as can be seen from \cref{proposition:ParSubobMeets}.} of the various $i_F$ in $\Sub_{\det}\left(K\right)$, and let $\left(B \xrightarrow{t_F} B_F\right)_{F \in \mathcal F}$ be the limit cone maps (note that for an inclusion $F \subseteq F'$, one necessarily has $t_{F'} = t_{F,F'}t_F$).
	Then, if for $F \in \FinSub{K}$ we define $g_F \coloneqq f_F \comp t_F$, the family of maps $\left(B \xrightarrow{g_F} J_F\right)_{F \in \mathcal F}$ is such that (for $F' \subseteq F$)
	\[
	J_{F,F'}\comp g_F = f_{F'}\comp t_{F,F'} \comp t_F = g_{F'}
	\]
	In other words, the family of maps $g_F$ is compatible with restriction and thus factors through an induced map $g \colon B \to \lim J$.
	Then, $A \xleftarrow{i} B \xrightarrow{g} \lim J$ is a morphism of $\Par{\cC}$ such that for each $F \in \FinSub{K}$
	\[
	\begin{tikzcd}[row sep=small]
		& B \arrow[dd, "t_{F}"] \arrow[r, "g"] \arrow[ld, "i"', tail] & \lim J \arrow[dd, "p_F"] \\
		A &                                                             &                            \\
		& B_{F} \arrow[r, "f_{F}"'] \arrow[lu, "i_{F}", tail]         & J_{F}                     
	\end{tikzcd}
	\]
	so that $p_F \comp \left(B,i,g\right) \le u_F$ for each $F$.
	
	Further, if we have another morphism $A \xleftarrow{i'} B' \xrightarrow{g'} \lim J$ such that for each $F$, $p_F \comp \left(B',i',g'\right) \le u_F$, we have diagrams in $\cC$ 
	\[
	\begin{tikzcd}[row sep=small]
		& B' \arrow[dd, "s_F"] \arrow[r, "g'"] \arrow[ld, "i'"', tail] & \lim J \arrow[dd, "p_F"] \\
		A &                                                              &                            \\
		& B_{F} \arrow[r, "f_{F}"'] \arrow[lu, "i_{F}", tail]          & J_{F}                     
	\end{tikzcd}
	\]
	So that $i'$ is also a lower bound for the $i_F$ and factors through their meet $i$ via some $B' \xrightarrow{s} B$.
	Then, the two maps $B' \xrightarrow{s} B \xrightarrow{g}\lim J$ and $g'$ have the same projections to arbitrary $J_F$, and are thus equal as well.
	Thus, $s$ witnesses $\left(B',i',g'\right) \le \left(B,i,g\right)$.
	Consequently, $\left(B,i,g\right)$ is the greatest map $A \to \lim J$ such that $p_F \comp \left(B,i,g\right) \le u_F$ for each $F$.
	In other words, we have the claimed lax limit.
	
	Now if the lax cone $\left(u_F\right)_{F \in \mathcal F}$ is strict, the $t_{F,F'}$ are all invertible (identity, even).
	So, the diagram in $\Sub_{\det}\left(A\right)$ whose meet is $i$ is contractible, and hence $i$ can just be constructed as one of the $i_F$.
	In any case, the meet projections $t_F$ are invertible as well, so we have strict equality $p_F \comp \left(B,i,g\right) = u_F$ for arbitrary $F$.
	As in the lax case, any factorization $\left(B',i',g'\right)$ of the cone $\left(u_F\right)_{F \in \mathcal F}$ through $\lim J$ is bounded above by $\left(B,i,g\right)$.
	If this factorization is strict (i.e.\ $p_F \comp \left(B',i',g'\right) = u_F$) then as $p_F$ is total, $\left(B',i',g'\right)$ has the same domain (as a subobject) as $u_F$ and hence $\left(B,i,g\right)$.
	Thus the inequality $\left(B',i',g'\right)\le \left(B,i,g\right)$ is an equality.
	Consequently, the inclusion is a strict limit as well.
\end{proof}

\begin{proposition}\label{proposition:ParLaxKolProds}
	Consider a partializable Markov category $\cC$ admitting Kolmogorov products of size $K$.
	Given a family of objects $\left(X_k\right)_{k \in K}$ in $\cC$, the inclusion of the Kolmogorov product projections into $\Par{\cC}$ define both a lax and a strict Kolmogorov product.
\end{proposition}
\begin{proof}
	Note first that for an arbitrary object $A$ of $\cC$, the deterministic subobject poset $\Sub_{\det}\left(A\right)$ has $K$-sized meets by \cref{proposition:ParSubobMeets}.
	
	Consequently for an arbitrary object $Y$, we can apply \cref{lemma:ParFinSubIndLaxLim} to the diagram
	\[
	\FinSub{K}^{\op} \xrightarrow{Y \otimes X^{\left(\phsm\right)}} \cC\colon  F \mapsto Y \otimes X^F
	\]
	which has by hypothesis limit $Y \otimes X^K$ in $\cC$.
	This means that the Kolmogorov product $X^K$ (and its projections) tensored with $Y$ in $\cC$ include to both a lax and a strict limit cone.
	Thus, the inclusion of the Kolmogorov product projections (the particular case where $Y$ is the unit) and their tensor with an arbitrary object $Y$ are both lax and strict limit cones.
	In other words, $X^K$ is both a lax and strict infinite product.	
	To see that this is further both a lax and a strict Kolmogorov product it suffices to observe that the projections $X^K \to X^F$ are deterministic by construction.
\end{proof}

\begin{proposition}\label{proposition:ParMorInfProd}
	Consider a partializable Markov category $\cC$ with Kolmogorov products of size $K$.
	Given a $K$-indexed family of morphisms $\left(u_k\right)_{k \in K} = \big(X_k \xleftarrow{i_k} A_k \xrightarrow{f_k} Y_k\big)_{k \in K}$ of $\Par{\cC}$, the family of maps $\big(X^K \xrightarrow{\pi_F} X^F \xrightarrow{u^F} Y^F\big)_{F \subseteq K \text{ finite}}$ defines a lax cone over the diagram of finite products $Y^F$ and marginalizations.
	Furthermore, the morphism $X^K \to Y^K$ induced by the universal property of the lax Kolmogorov product is $\big(X^K \xleftarrow{i^K} A^K \xrightarrow{f^K} Y^K\big)$ (that is, computed by the $K$-sized product componentwise).
\end{proposition}
\begin{proof}
	First, for an inclusion $F' \subseteq F$ of finite subsets of $K$, we can compute $\pi_{F,F'} \comp u^F \comp \pi_F$ by forming
	\[
	\begin{tikzcd}[sep=small]
		&                                                 & B_F \arrow[ld, "\iota_F"', tail] \arrow[rd, "p_F"] \arrow["\lrcorner"{anchor=center, pos=0.125, rotate=-45}, draw=none, dd]&                                                                                     &        \\
		& X^K \arrow[ld, Rightarrow, no head] \arrow[rd, "\pi_F"'] &                                                    & A^F \arrow[rd, "{\pi_{F,F'}f^F = f^{F'} \pi_{F,F'}}"] \arrow[ld, "i^F", tail] &        \\
		X^K &                                                 & X^F                                                &                                                                                     & Y^{F'}
	\end{tikzcd}
	\]
	On the other hand, $u^{F'} \comp \pi_{F'}$ is computed simply as
	\[
	\begin{tikzcd}[sep=small]
		&                                                    & B_{F'} \arrow[ld, "\iota_{F'}"', tail] \arrow[rd, "p_{F'}"] \arrow["\lrcorner"{anchor=center, pos=0.125, rotate=-45}, draw=none, dd]&                                                           &        \\
		& X^K \arrow[ld, Rightarrow, no head] \arrow[rd, "\pi_{F'}"'] &                                                             & A^{F'} \arrow[rd, "f^{F'}"] \arrow[ld, "i^{F'}", tail] &        \\
		X^K &                                                    & X^{F'}                                                      &                                                           & Y^{F'}
	\end{tikzcd}
	\]
	Consequently, the universal factorization $s_{F,F'}$ in
	\[
		\begin{tikzcd}
		B_F \arrow[rdd, "\iota_F"', tail, bend right] \arrow[rr, "p_F"] \arrow[rd, "{s_{F,F'}}", dashed] &                                                           & A^F \arrow[d, "{\pi_{F,F'}}"] \arrow[rd, "i^F", tail, bend left] &                                \\
		& B_{F'} \arrow[d, "\iota_{F'}"', tail] \arrow[r, "p_{F'}"] & A^{F'} \arrow[d, "i^{F'}", tail]                                 & X^F \arrow[ld, "{\pi_{F,F'}}"] \\
		& X^K \arrow[r, "\pi_{F'}"']                                & X^{F'}                                                           &                               
	\end{tikzcd}
	\]
	witnesses $\pi_{F,F'}\comp u^F \comp \pi_F \le u^{F'} \comp \pi_{F'}$.
	To see that such a factorization exists, note that the two outer paths $B_F \rightsquigarrow X^{F'}$ in the diagram commute by the construction of the two spans.
	The compatibility of the finite products $i^F,i^{F'}$ with marginalization then means that this translates to the given maps defining a cone over the cospan $\big(X^K \xrightarrow{\pi_{F'}} X^{F'} \xleftarrow{i^{F'}} A^{F'}\big)$.
	Thus, the given family of maps defines a lax cone, and therefore corresponds to a map $X^K \to Y^K$.
	
	Following the construction of this map in the proof of \cref{lemma:ParFinSubIndLaxLim}, this map is represented by a span $\big(X^K \xleftarrow{\iota} B \xrightarrow{g} Y^K\big)$, where $\iota$ is the meet of the various $\iota_F$ (with projections denoted $B \xrightarrow{t_F} B_F$), and the map $g$ is the map induced by the universal property of the Kolmogorov product applied to the family of maps $\big(g_F \coloneqq B \xrightarrow{t_F} B_F \xrightarrow{f^Fp_F} Y^F\big)_{F \subseteq K \text{ finite}}$.
	
	From \cref{lemma:ParFiniteMarginalFibre}, we see that $X^K \xleftarrow{\iota_F} B_F \xrightarrow{p_F} A^F$ can be computed as $K$-sized products $X^K \xleftarrow{\iota^K} B^K \xrightarrow{p^K} A^F$.
	Here, when $k \in F$, $B_k = A_k$, $p_k$ is the identity and $\iota_k = i_k$.
	Otherwise, $B_k = X_k$, $p_k$ is the deletion map and $\iota_k$ is the identity.
	Furthermore, from \cref{lemma:ParSubobBoxCutout}, we see that the meet $\iota$ is then given by $A^K \xrightarrow{i^K} X^K$.
	In other words, the induced map $X^K \to Y^K$ has the claimed domain inclusion.
	
	Using this presentation, $g_F$ can be identified as the map $A^K \xrightarrow{\pi_F} A^F \xrightarrow{f^F} Y^F$. Thus the induced map $g$ induced by the universal property of the product applied to the family of maps $g_F$ is precisely $f^K$.
	In other words, the induced map has the claimed ``right way map'' as well.
\end{proof}

To summarise, while we cannot in general construct infinite tensor products of morphisms in $\Par{\cC}$ by following the construction used for $\cC$ verbatim, we can still construct an infinite tensor product functor componentwise.
This infinite tensor product of morphisms is still a universal property induced map, only now arising from a lax universal property.

\begin{remark}\label{remark:NormInfCopy}
	Consider a quasi-Markov category $\cC$ with $K$-sized \emph{strict} Kolmogorov products and a morphism $f\colon X \to Y$.
	One may attempt to define an infinite copy $f^{\left(K\right)}\colon X \to Y^K$, intuitively given by ``$K$ independent samples of $f$'' by the universal property of the strict Kolmogorov product, as with infinite products in the Markov case.
	For instance, when $\cC$ is representable and $f$ is the sampling map, the infinite copy $\samp^{\left(K\right)}$ can be seen as encoding ``iterated sampling'', which comes up in a variety of contexts such as de-Finetti's theorem~\cite[Theorem 4.4]{fritz2021definetti} (where exchangeable morphisms are characterized as those that can be given by iterated sampling from some prior), or a strengthening of representability called ``observational representability'' where one can intuitively distinguish distributions by iteratively sampling from them~\cite{moss2022probability}.
	
	Explicitly, one would define $f^{\left(K\right)}\colon X \to Y^K$ to be the unique morphism whose finite projections $\pi_F \comp f^{\left(K\right)}$ are given by $f^{\left(F\right)}\coloneqq f^F \comp \id^{\left(F\right)}$, where $\id^{\left(F\right)}$ is the $\vert F\vert$-output morphism given by iterated copying.
	
	One might expect this to produce similar problems as with infinite products, that is, that this family of morphisms might not satisfy the marginalization coherence that would be required in order to apply the universal property of the (strict) infinite product.
	While this is a point that would have to be addressed in a general CD category, this family of maps in the quasi-Markov $\cC$ is indeed compatible with marginalization.
	Repeated application (to arbitrary but finite degree) of \cref{eq:quasi-total} shows that the family of morphisms $\{f^{\left(F\right)}\}$ is compatible with marginalization (where we define $f^{\left(\emptyset\right)}\coloneq \discard \comp f$).
	
	In fact, given \emph{lax} Kolmogorov products we have further the equation $f^{(K)} = f^K \comp \id^{(K)}$, with the caveat that $f^K$ is only the map produced by the \emph{lax} universal property (and computed componentwise).
\end{remark}

\begin{proposition}\label{proposition:ParCDInfCopies}
	Consider a partializable Markov category $\cC$.
	Assume that $\cC$ has $K$-sized Kolmogorov products (so that $\Par{\cC}$ has them induced from $\cC$).
	Let $\left(D,i,f\right)$ represent a morphism of $\Par{\cC}$.
	Then, the infinite copy $\left(D,i,f\right)^{\left(K\right)}$ of $\left(D,i,f\right)$ is represented by $\left(D,i,f^{\left(K\right)}\right)$.
\end{proposition}
\begin{proof}
	Let $\left(D,i,f\right)^{\left(K\right)}$ be represented by a span $\left(E,j,g\right)$.
	We first identify the domain inclusions, i.e $j$ with $i$.
	By repeated application (to any arbitrary finite degree) of the copyability of domains (quasi-totality), the domains of $\left(D,i,f\right)$ and $\left(D,i,f\right)^{\left(K\right)}$ coincide.
	Thus, we have $\left(D,i,\discard{} \comp f\right) = \left(E,j,\discard{} \comp g\right)$ and therefore $i$ and $j$ are equal as subobjects.
	Consequently, there is a $h$ such that $\left(D,i,f\right)^{\left(K\right)} = \left(D,i,h\right)$.
	To understand $h$, for finite $F \subseteq K$ we can compute
	\[\left(D,i,\pi_F \comp h\right) = \pi_F \comp \left(D,i,h\right) = \pi_F \comp \left(D,i,f\right)^{\left(K\right)} = \left(D,i,f\right)^{\left(F\right)} = \left(D,i,f^{\left(F\right)}\right)\]
	Consequently $\pi_F \comp h = f^{(F)}$ for arbitrary finite $F$, and therefore $h = f^{(K)}$ as claimed.
\end{proof}

\appendix

\section{Some results on partializable Markov categories}\label{sec:par_facts}

\begin{lemma}\label{lemma:factorThroughMonoCart}
	In any category, any commutative square
	\[
	\begin{tikzcd}
		A \arrow[d, Rightarrow, no head] \arrow[r, "f"] & X \arrow[d, "i", tail] \\
		A \arrow[r, "g"']                      & Y                     
	\end{tikzcd}
	\]
	with $i$ monic is a pullback.
\end{lemma}
\begin{proof}
	\[
	\begin{tikzcd}
		A \arrow[d, Rightarrow, no head] \arrow[r, "f"] \arrow[rd, "\lrcorner"{anchor=center, pos=0.125}, draw=none]& X \arrow[d, Rightarrow, no head] \arrow[r, Rightarrow, no head] \arrow[rd, "\lrcorner"{anchor=center, pos=0.125}, draw=none]& X \arrow[d, "i", tail] \\
		A \arrow[r, "f"']                      & X \arrow[r, "i"', tail]                       & Y                     
	\end{tikzcd}
	\]
\end{proof}

See also the remarks at the start of Cockett and Lack's~\cite[Section 3.1]{cockettlack2002partialmaps}.

\begin{lemma}\label{DetMonoLeftCancellation}
	Consider a partializable Markov category $\cC$, a monomorphism $m$ and a deterministic $u$ such that $u = m \comp f$.
	Assume that either:
	\begin{enumerate}
		\item\label{triangleMonicCond} $u$ is monic;
		\item\label{triangleDeterministicCond} $m$ is deterministic.
	\end{enumerate}
	Then, $f$ is deterministic as well (and in case~\ref{triangleMonicCond} a deterministic monomorphism).
\end{lemma}
\begin{proof}
	Case~\ref{triangleMonicCond} is an immediate corollary of \cref{lemma:factorThroughMonoCart} as $f$ is then a pullback of the deterministic monomorphism $u$.

	Consider case~\ref{triangleDeterministicCond}.
	The fact that $m \otimes m$ is assumed to be monic lets us deduce the claim from
	\[
	\begin{tikzpicture}
	\begin{pgfonlayer}{nodelayer}
		\node [style=none] (0) at (0, 0) {$=$};
		\node [style=none] (1) at (-4, 0) {$=$};
		\node [style=bn] (3) at (-6, 0) {};
		\node [style=morphism] (4) at (-7, 1.5) {$m$};
		\node [style=morphism] (5) at (-5, 1.5) {$m$};
		\node [style=morphism] (6) at (-6, -1.25) {$f$};
		\node [style=none] (7) at (-7, 3) {};
		\node [style=none] (8) at (-5, 3) {};
		\node [style=none] (9) at (-6, -2.75) {};
		\node [style=bn] (10) at (-2, 0) {};
		\node [style=morphism] (11) at (-2, -0.75) {$m$};
		\node [style=morphism] (12) at (-2, -2) {$f$};
		\node [style=none] (13) at (-3, 3) {};
		\node [style=none] (14) at (-1, 3) {};
		\node [style=bn] (15) at (2, 0) {};
		\node [style=morphism] (16) at (1, 1) {$f$};
		\node [style=morphism] (17) at (1, 2.25) {$m$};
		\node [style=morphism] (18) at (3, 2.25) {$m$};
		\node [style=morphism] (19) at (3, 1) {$f$};
		\node [style=none] (20) at (2, -2.75) {};
		\node [style=none] (21) at (1, 3) {};
		\node [style=none] (22) at (3, 3) {};
		\node [style=none] (23) at (-2, -2.75) {};
	\end{pgfonlayer}
	\begin{pgfonlayer}{edgelayer}
		\draw (6) to (9.center);
		\draw (3) to (6);
		\draw (7.center) to (4);
		\draw [in=-180, out=-90, looseness=0.75] (4) to (3);
		\draw [in=-90, out=0, looseness=0.75] (3) to (5);
		\draw (5) to (8.center);
		\draw [in=-90, out=0, looseness=0.75] (10) to (14.center);
		\draw [in=-90, out=-180, looseness=0.75] (10) to (13.center);
		\draw (10) to (11);
		\draw (11) to (12);
		\draw (12) to (23.center);
		\draw (15) to (20.center);
		\draw [in=-90, out=180] (15) to (16);
		\draw (16) to (17);
		\draw (17) to (21.center);
		\draw (22.center) to (18);
		\draw [in=90, out=-90] (18) to (19);
		\draw [in=0, out=-90] (19) to (15);
	\end{pgfonlayer}
\end{tikzpicture}
	\]
	The second equality is a consequence of the fact that we assumed $u = m \comp f$ deterministic.
\end{proof}

\begin{remark}\label{remark:detSubobCat}
	Consider a partializable Markov category $\cC$.
	As a consequence of \cref{DetMonoLeftCancellation}, the category $\Sub_{\det}(X)$ of deterministic subobjects (in $\cC$) of an object $X$ is the same as $\Sub(X)$ with $X$ considered as an object of $\cC_{\det}$. Both can further be equivalently identified as the subcategory of the slice $\cC_{/X}$ of the deterministic monomorphisms and deterministic monomorphisms between them (which is a full subcategory).
	
	Thus, meets in $\Sub_{\det}\left(X\right)$ (when they exist) correspond to limits of diagrams of deterministic monomorphisms over $X$.
	In particular, the limit projections are also deterministic monomorphisms.
\end{remark}

\begin{proposition}\label{proposition:DetInclusionCartPreservation}
	Let $\cC$ be a partializable Markov category.
	The inclusion $\cC_{\det} \to \cC$ creates pullbacks along monomorphisms.
\end{proposition}
\begin{proof}
	We first show that a pullback (in $\cC$) along a deterministic monomorphism
	\[
	\begin{tikzcd}
		W \arrow[d, "j"', tail] \arrow[r, "g"] \arrow[rd, "\lrcorner"{anchor=center, pos=0.125}, draw=none]& Y \arrow[d, "i", tail] \\
		X \arrow[r, "f"']                      & Z                     
	\end{tikzcd}
	\]
	such that all morphisms involved are deterministic\footnote{In fact, by \cref{DetMonoLeftCancellation} any pullback in $\cC$ of a deterministic $f$ along a deterministic monomorphism $i$ is of this form.} is a pullback in $\cC_\det$.
	For this, consider $A \xrightarrow{u} X$, $A \xrightarrow{v} Y$ deterministic such that $f \comp u = i \comp v$.
	We need only show that the induced map $h$ to the pullback $W$ is deterministic; this is a consequence of \cref{DetMonoLeftCancellation} applied to $u = j \comp h$.
	
	Conversely, a pullback along a monomorphism in $\cC_\det$
	\[
	\begin{tikzcd}
		A \arrow[r, "g"] \arrow[d, "j"', tail] \arrow[rd, "\lrcorner"{anchor=center, pos=0.125}, draw=none]& C \arrow[d, "i", tail] \\
		B \arrow[r, "f"']                      & D                     
	\end{tikzcd}
	\]
	is a pullback in $\cC$.
	To see this, note that there is by hypotheses a pullback
	\[
	\begin{tikzcd}
		X \arrow[r, "v"] \arrow[d, "u"', tail] \arrow[rd, "\lrcorner"{anchor=center, pos=0.125}, draw=none]& C \arrow[d, "i", tail] \\
		B \arrow[r, "f"']                      & D                     
	\end{tikzcd}
	\]
	in $\cC$, where $u$ is a deterministic monomorphism, hence with $v$ deterministic as well by \cref{DetMonoLeftCancellation}.
	Thus by the first assertion of the claim, this square is also a pullback in $\cC_\det$.
	As $\left(X,u,v\right)$ and $\left(A,j,g\right)$ both define limits of the same diagram in $\cC_\det$, there is an isomorphism between the two cones. This remains an isomorphism between them in $\cC$, completing the proof.
\end{proof}

\begin{corollary}
	\label{corollary:SPTSUniversality}
	Consider a partializable Markov category. For every deterministic monomorphism $X\xrightarrow{i} Y$, the following squares are pullbacks.
	\[
	\begin{tikzcd}
		X \otimes A \arrow[d, "i \otimes A"', tail] \arrow[r, "\pi_X"] \arrow[rd, "\lrcorner"{anchor=center, pos=0.125}, draw=none] & X \arrow[d, "i", tail] \\
		Y \otimes A \arrow[r, "\pi_Y"']                                 & Y                     
	\end{tikzcd}
	\qquad\qquad
	\begin{tikzcd}
		A \arrow[d, "i"', tail] \arrow[r, "(A \otimes i)\comp\cop"] \arrow[rd, "\lrcorner"{anchor=center, pos=0.125}, draw=none] & A \otimes X \arrow[d, "i \otimes A", tail] \\
		X \arrow[r, "\cop"']                & X \otimes X                         
	\end{tikzcd}
	\qquad\qquad
	\begin{tikzcd}
		A \arrow[d, "i"', tail] \arrow[r, "\cop"] \arrow[rd, "\lrcorner"{anchor=center, pos=0.125}, draw=none] & A \otimes A \arrow[d, "i \otimes i", tail] \\
		X \arrow[r, "\cop"']                & X \otimes X                         
	\end{tikzcd}
	\]
\end{corollary}

\begin{corollary}\label{corollary:ParDetSubcategory}
	For $\cC$ a partializable Markov category, so is $\cC_\det$.
	The inclusion $\cC_\det \to \cC$ induces an inclusion $\Par{\cC_\det} \to \Par{\cC}$.
\end{corollary}

\begin{remark}
	A commuting square in which either of the two parallel pairs of morphisms are isomorphisms is a pullback.
\end{remark}

\subsection{Kolmogorov products in partializable Markov	categories}

\begin{proposition}\label{proposition:InfProdDetMonos}
	Consider a partializable Markov category $\cC$ that has Kolmogorov products of size $K$.
	Given a family of deterministic monomorphisms $\left(X_t \xrightarrow{i_t} Y_t\right)_{t \in K}$, the product $i^K \coloneqq \bigotimes_{t \in K} i_t$ is a (deterministic) monomorphism.
\end{proposition}
\begin{proof}
	Consider $f,g$ such that $i^K\comp f = i^K \comp g$.
	Then, for any finite $F \subseteq K$,
	\[
	i^F \comp \pi_F \comp i^K \comp f = \pi_F \comp i^K \comp g = i^F \comp \pi_F \comp g
	\]
	so that $\pi_F f = \pi_F g$ since the finite product $i^F = \bigotimes_{t \in F} i_t$ is monic.
	As $F$ was arbitrary, the universal property of the Kolmogorov product ensures $f = g$.
\end{proof}

\begin{proposition}\label{proposition:ParSubobMeets}
	Consider a partializable Markov category $\cC$ with $K$ sized Kolmogorov products.
	For an object $B$ of $\cC$, the meet in $\Sub_{\det}(B)$ of a family $A_k \xrightarrow{i_k} B$ of deterministic monomorphisms over $B$ exists, and is given by the pullback $\iota$ in
	\[
	\begin{tikzcd}
		A \arrow[d, "\iota"', tail] \arrow[r, "j", tail] \arrow[rd, "\lrcorner"{anchor=center, pos=0.125}, draw=none]& A^K \arrow[d, "i^K", tail] \\
		B \arrow[r, "\cop"', tail]                      & B^K                 
	\end{tikzcd}
	\]
	(where we again use the notation $i^K \coloneqq \bigotimes_{k \in K} i_k$)
\end{proposition}
\begin{proof}
	The morphisms $\iota$ and $j$ are pullbacks of deterministic monomorphisms and hence themselves deterministic monomorphisms (so in particular $\iota \in \Sub_{\det}(B)$).
	Now if $j_k$ is the marginal of $j$ on to $A_k$ for some $k$, we have $i_k \comp j_k = \iota$ so that $\iota \le i_k$ for each $k$.
	
	Conversely, given a subobject $C \xrightarrow{i} B$ such that $i \le i_k$ for each $k$, there are $s_k \colon C \to A_k$ such that $i_k \comp s_k = i$ (necessarily deterministic monomorphisms by \cref{DetMonoLeftCancellation}).
	Then, the morphism $s \coloneqq \left(s_k\right)_{k \in K}$ satisfies $i^K \comp s = i^{(K)}$ and therefore witnesses $i \le \iota$ by the universal property of the pullback.
	Consequently, $\iota$ is the infimum of the $i_k$'s, as claimed.
\end{proof}

\begin{lemma}\label{lemma:ParSubobBoxCutout}
	Consider a partializable Markov category $\cC$ and a family of deterministic monomorphisms $\left(A_k \xrightarrow{i_k} X_k\right)_{k \in K}$.
	For finite $F \subseteq K$ and arbitrary $t\in K$, let
	\[
		B_{F,t} \coloneqq \begin{cases} A_t & \text{if } t \in F \\ X_t & \text{otherwise} \end{cases}
	\]
	and set $j_{F,t}\colon B_{F,t} \to X_t$ to be $i_t$ when $t \in F$ and the identity otherwise.
	Let $j_F^K \coloneqq \bigotimes_{k \in K}j_{F,k}$ and use the convention $i^K \coloneqq \bigotimes_{k \in K} i_k$.
	Then $i^K$ is the meet in $\Sub_{\det}\left(X^K\right)$ of the $j_F^K$ (ranging over all finite subsets $F$ of $K$).
\end{lemma}
\begin{proof}
	First, $i^K$ is a lower bound as for each finite subset $F$, there is a witness to $i^K \le j_F^K$ given by the map $A^K \to B_F^K$ (where $B_F^K \coloneqq \bigotimes_{k \in K} B_{F,k}$) formed by tensoring $i_k$ on the $k \notin F$ and the identity on $k \in F$.
	
	Conversely, consider an arbitrary deterministic monomorphism $C \xrightarrow{\iota} X^K$ that is a lower bound for the $j_F^K$'s.
	Consider witnesses for this, say $C \xrightarrow{u_F} B_F^K$ such that $j_F^K \comp u_F = \iota$.
	Then, the marginal $v_k \coloneqq \left(u_{\left\{k\right\}}\right)_k$ of $u_{\{k\}}$ on to $B_{{\{k\}},k} = A_k$ is such that $i_k \comp v_k$ is the marginal $\iota_k$ of $\iota$ on to $X_k$.
	As $\iota$ and the $i_k$'s are deterministic, we have
	\[
	\iota = \left(\iota_k\right)_{k \in K} = \left(i_k \comp v_k\right)_{k \in K} = i^K \comp \left(v_k\right)_{k \in K}
	\]
	In particular $\iota \le i^K$, and as $\iota$ was an arbitrary lower bound of the $j_F^K$'s, $i^K$ is the infimum, as claimed.
\end{proof}

\begin{remark}
	In light of \cref{remark:detSubobCat,proposition:DetInclusionCartPreservation}, the \cref{proposition:InfProdDetMonos,proposition:ParSubobMeets,lemma:ParSubobBoxCutout} are really statements about cartesian products, as they refer to pullbacks and subobjects in the subcategory of deterministic morphisms.
	In any case, the proofs are essentially the same whether the products are considered as cartesian or Kolmogorov.
\end{remark}

On the other hand, we will also have cause to consider pullbacks along morphisms that need not be deterministic.
Here, the difference from the cartesian case will have to be dealt with. 

\begin{proposition}\label{proposition:StablePullbackTensorStability}
	Consider a partializable Markov category with Kolmogorov products of size $K$ (when $K$ is finite, this holds for any partializable Markov category).
	
	The tensor product of $\vert K \vert$ pullbacks along deterministic monomorphisms is a pullback.
	That is, given pullbacks
	\[\begin{tikzcd}
		A_k \arrow[r,"g_k"] \arrow[d, "j_k"', tail] \arrow[rd, "\lrcorner"{anchor=center, pos=0.125}, draw=none] & B_k \arrow[d, "i_k", tail] \\
		C_k \arrow[r,"f_k"']                                            & D_k                      
	\end{tikzcd}\]
	for each $k \in K$ with all the $i_k$'s deterministic monomorphisms, the following square is a pullback as well (where for a family of objects (resp.\ morphisms) $\left(X_k\right)_{k\in K}$ we use the shorthand $X^K \coloneqq \bigotimes_{k \in K} X_k$).\footnote{The vertical maps are deterministic monomorphisms by \cref{proposition:InfProdDetMonos}.}
	\[\begin{tikzcd}
		A^K \arrow[r, "g^K"] \arrow[d, "j^K"'] \arrow[rd, "\lrcorner"{anchor=center, pos=0.125}, draw=none] & B^K \arrow[d, "i^K"] \\
		C^K \arrow[r,"f^K"']                                  & D^K
	\end{tikzcd}\]
\end{proposition}
\begin{proof}
	As $\cC$ is partializable, there is a pullback square
	\[\begin{tikzcd}
		X \arrow[r, "v"] \arrow[d, "u"', tail] \arrow[rd, "\lrcorner"{anchor=center, pos=0.125}, draw=none] & B^K \arrow[d, "i^K", tail] \\
		C^K \arrow[r, "f^K"'] & D^K                        
	\end{tikzcd}\]
	with $u$ a deterministic monomorphism.
	Consequently, the universal property produces a factorization
	\[
	\begin{tikzcd}
		A^K \arrow[rd, "t"] \arrow[rrd, "g^K", bend left] \arrow[rdd, "j^K"', bend right, tail] &                                         &                                        \\
		& X \arrow[r, "v"] \arrow[d, "u"',tail] \arrow[rd, "\lrcorner"{anchor=center, pos=0.125}, draw=none] & B^K \arrow[d, "i^K", tail] \\
		& C^K \arrow[r, "f^K"'] & D^K                          
	\end{tikzcd}
	\]
	Thus in particular, $t$ establishes $j^K \le u$ as subobjects of $C^K$.
	
	Let $u$ have marginals $u_k$ on to $C_k$ for arbitrary $k \in K$, and those of $v$ be $v_k$.
	Then, we have commutative squares given by the marginals that induce (for each $k$) factorizations
	\[\begin{tikzcd}
		X \arrow[rdd, "u_k"', bend right] \arrow[rrd, "v_k", bend left] \arrow[rd, "s_k"] &                                        &                        \\
		& A_k \arrow[d, "j_k"', tail] \arrow[r, "g_k"] \arrow[rd, "\lrcorner"{anchor=center, pos=0.125}, draw=none]& B_k \arrow[d, "i_k", tail] \\
		& C_k \arrow[r, "f_k"']                      & D_k
	\end{tikzcd}\]
	As $u$ is deterministic, it is determined by its marginals as $\left(u_k\right)_{k \in K}$.
	That is, 
	\[
	\begin{tikzpicture}
	\begin{pgfonlayer}{nodelayer}
		\node [style=medium box] (0) at (-6.5, 0) {$\quad u \quad $};
		\node [style=bn] (1) at (0, -0.75) {};
		\node [style=bn] (2) at (7.5, -1.5) {};
		\node [style=morphism] (3) at (-2.5, 0.75) {$u_0$};
		\node [style=morphism] (4) at (0, 0.75) {$u_{k}$};
		\node [style=morphism] (5) at (7.5, 0) {$s_k$};
		\node [style=morphism] (6) at (5, 0) {$s_0$};
		\node [style=morphism] (7) at (7.5, 1.5) {$j_k$};
		\node [style=morphism] (8) at (5, 1.5) {$j_0$};
		\node [style=none] (9) at (-6.5, -1.5) {};
		\node [style=none] (10) at (-7.25, 1.5) {};
		\node [style=none] (11) at (-5.75, 1.5) {};
		\node [style=none] (12) at (-7.25, 0.5) {};
		\node [style=none] (13) at (-5.75, 0.5) {};
		\node [style=none] (14) at (0, -1.75) {};
		\node [style=none] (15) at (7.5, -2.5) {};
		\node [style=none] (16) at (5, 2.5) {};
		\node [style=none] (17) at (7.5, 2.5) {};
		\node [style=none] (18) at (0, 1.75) {};
		\node [style=none] (19) at (-2.5, 1.75) {};
		\node [style=none] (20) at (-4.25, 0) {$=$};
		\node [style=none] (21) at (3.5, 0) {$=$};
		\node [style=none] (22) at (-6.5, 1) {$\dots$};
		\node [style=none] (23) at (-1.25, 0.75) {$\dots$};
		\node [style=none] (24) at (6.25, 0) {$\dots$};
		\node [style=none] (25) at (1.5, 0.75) {$\dots$};
		\node [style=none] (26) at (2.5, 0.75) {};
		\node [style=none] (27) at (10, 0) {};
		\node [style=none] (28) at (9, 0) {$\dots$};
	\end{pgfonlayer}
	\begin{pgfonlayer}{edgelayer}
		\draw [in=-90, out=180, looseness=0.75] (1) to (3);
		\draw [in=90, out=-90, looseness=0.50] (4) to (1);
		\draw (8) to (6);
		\draw [in=90, out=-90] (7) to (5);
		\draw [in=180, out=-90, looseness=0.75] (6) to (2);
		\draw [in=-90, out=90, looseness=0.25] (2) to (5);
		\draw (19.center) to (3);
		\draw (4) to (18.center);
		\draw (12.center) to (10.center);
		\draw (11.center) to (13.center);
		\draw (0) to (9.center);
		\draw (1) to (14.center);
		\draw (2) to (15.center);
		\draw (16.center) to (8);
		\draw (17.center) to (7);
		\draw [in=-90, out=0, looseness=0.75] (1) to (26.center);
		\draw [in=-90, out=0, looseness=0.75] (2) to (27.center);
	\end{pgfonlayer}
\end{tikzpicture}
	\]
	Thus, $u$ factors through $j^K$, or in other words $\left(s_k\right)_{k \in K}$ establishes $u \le j^K$ as subobjects of $C^K$ as well.
	Therefore, the comparison map $t$ is invertible, proving the claim.
\end{proof}

\begin{lemma}\label{lemma:ParFiniteMarginalFibre}
	Consider a partializable Markov category $\cC$ with Kolmogorov products of size $K$.
	Given a family $\big(A_k \xrightarrow{i_k} X_k\big)_{k \in K}$ of deterministic monomorphisms in $\cC$ and a finite subset $F \subseteq K$, set for each $k \in K$
	\[
		B_{F,k} \coloneqq \begin{cases} A_k & \text{if } k \in F \\ X_k & \text{otherwise} \end{cases}
	\]
	and $B_F^K \coloneqq \bigotimes_{k \in K} B_{F,k}$.
	Let $j_F^K \colon B_F^K \to X^K \coloneqq \bigotimes_{k \in K} j_{F,k}$, where $j_{F,k}\colon B_{F,k} \to X_k$ is $i_k$ for $k \in F$ and identity otherwise. 
	Then, the following square is a pullback\footnote{This generalizes \cref{corollary:SPTSUniversality}. As with \cref{corollary:SPTSUniversality}, this can be shown by an appeal to \cref{proposition:DetInclusionCartPreservation}. However, there is again the immediate direct argument that we present.}
	\[
		\begin{tikzcd}
		B_F \arrow[r, "p_F"] \arrow[d, "j_F^K"', tail] & A^F \arrow[d, "i^F", tail] \\
		X^K \arrow[r, "\pi_F"']                      & X^F                       
	\end{tikzcd}
	\]
	where $p_F$ (like $\pi_F$) is given by the identity on components belonging to $F$ and deletion elsewhere.
\end{lemma}
\begin{proof}
	For each finite $F \subseteq K$, there are pullbacks
	\[
		\begin{tikzcd}
		A^F \arrow[r, Rightarrow, no head] \arrow[d, "i^F"', tail] \arrow[rd, "\lrcorner"{anchor=center, pos=0.125}, draw=none]& A^F \arrow[d, "i^F", tail] \\
		X^F \arrow[r, Rightarrow, no head]                         & X^F                 
	\end{tikzcd}
	\qquad \qquad \qquad \qquad
		\begin{tikzcd}
		X^{K\setminus F} \arrow[r, "\discard"] \arrow[d, Rightarrow, no head] \arrow[rd, "\lrcorner"{anchor=center, pos=0.125}, draw=none]& I \arrow[d, Rightarrow, no head] \\
		X^{K\setminus F} \arrow[r, "\discard"']                      & I                      
	\end{tikzcd}
	\]
	Up to permuting factors, the commutative square of the hypothesis is precisely the tensor of these two pullbacks, and is hence also a pullback by \cref{proposition:StablePullbackTensorStability}.
\end{proof}

\bibliographystyle{alphaabbrdoi}
\bibliography{markov,specificRefs}

\end{document}